\documentclass[11pt, a4paper, english]{article}

\usepackage{geometry}

\usepackage[utf8]{inputenc}
\usepackage{parskip}
\usepackage{todonotes}
\usepackage{caption}
\usepackage[labelformat=brace,position=top]{subcaption}

\usepackage{graphicx}

\usepackage{xxcolor}
\usepackage{color}

\usepackage{appendix}

\usepackage{mathrsfs}

\usepackage{amsmath}
\usepackage{amssymb}
\usepackage{amsthm}
\usepackage{scrextend}
\usepackage{bbm}
\usepackage{stmaryrd}
\usepackage{physics}
\usepackage{enumerate}
\usepackage{tikz, pgffor}
\usepackage[linktoc=true,linkcolor=red,citecolor=magenta]{hyperref}
\usepackage{cleveref}
\usepackage{mathtools}
\usepackage{amsfonts}
\usepackage{comment}

\usepackage{esint} %% for dashint (command is \fint)

\newcommand{\eps}{\varepsilon}

\newcommand{\dist}{\mathrm{dist}}
\renewcommand{\{}{\left\lbrace}
\renewcommand{\}}{\right\rbrace}

\DeclareMathOperator{\id}{id}
\DeclareMathOperator{\Id}{Id}

\DeclareMathOperator*{\Lip}{Lip}
\DeclareMathOperator{\Vol}{vol}

\DeclareMathOperator*{\argmin}{arg\,min}

\DeclareMathOperator{\sym}{sym}

\allowdisplaybreaks % allow pagebreak in long multiline equations

\theoremstyle{plain}
\newtheorem{theorem}{Theorem}
\newtheorem{lemma}[theorem]{Lemma}
\newtheorem{corollary}[theorem]{Corollary}
\newtheorem{proposition}[theorem]{Proposition}

\newtheorem{definition}[theorem]{Definition}

\theoremstyle{definition}
\newtheorem{remark}[theorem]{Remark}
\newtheorem{step}{Step}

\begin{document}

\numberwithin{equation}{section}

\title
{Non-Euclidean elasticity for rods  and almost isometric embeddings of geodesic tubes\\
}

\author{Milan Kr\"omer and Stefan Müller
\thanks{MK and SM have been supported by the Deutsche Forschungsgemeinschaft (DFG, German Research Foundation) through
the Hausdorff Center for Mathematics (GZ 2047/1, Projekt-ID 390685813)   }}

\maketitle

\tableofcontents

\section{Introduction}
\label{sec:intro}

Non-Euclidean elasticity is a theory of elastic bodies which do not have a compatible stress-free reference configuration, and was originally developed to formulate and analyze models of metal plasticity and  dislocations. 
Recently, it attracted a lot of attention  in the study of biomechanical systems, see for example, 
\cite{Klein2007, Sharon2007, Liang2009, Levin2021, Maor2025}.
In this context,  bodies which are small in one or several dimensions are of particular interest because 
they often arise in applications and they show universal behavior based on the underlying geometric 
nonlinearities rather than the detailed material laws.

The goal of this article is twofold.
First, we  extend the derivation of a low energy theory for inextensible rods by 
$\Gamma$-convergence~\cite{Mora2003} 
 from a Euclidean to a Riemannian setting. Secondly, for the case for thin geodesic tubes
 we extend the results of Maor and Shachar~\cite{Maor2019} on optimal 
 embeddings of thin geometric objects  in two ways. 
 First, the target is a general Riemannian manifold rather than $\mathbb R^n$ and,  secondly, 
 we give  a precise description of the asymptotic energy as a quadratic form in the difference of the 
 curvature tensors in the domain and the target (see \cite[Question 3, page 154]{Maor2019}).
 This also extends the work of Maor and Shachar~\cite{Maor2019}
 us~\cite{Kroemer2025} on the optimal embedding of small balls.

 To set the stage, let $(\mathcal M,g)$ and $(\tilde{\mathcal M}, \tilde g)$ be smooth, oriented, $n$ dimensional, Riemannian manifolds with $n \ge 2$ 
 and assume that $\mathcal M$ is complete  
  and $\tilde{\mathcal M}$ is compact.
  Let $\psi:[-L, L] \to \mathcal M$ be an injective unit speed geodesic and let $\underline \nu = (\nu_2, \ldots, \nu_n):
  (-L,L) \to( T\mathcal M)^{n-1}$ be a parallel  orthonormal frame of the normal bundle of $\psi$, such that 
  $(\gamma', \underline \nu)$ is positively oriented. We call the pair $(\gamma', \underline \nu)$ a framed unit speed geodesic. 
  Define geodesic normal coordinates (also known as Fermi coordinates) by 
  $$ \psi(x) \coloneqq \psi_{\gamma, \underline \nu}(x) \coloneqq  \exp_{\gamma(x_1)} \sum_{j=2}^n x_j \nu_j(x_1)$$
 and consider the  geodesic tube 
  $$\Omega^*_h \coloneqq \psi((-L,L) \times B_h(0)).$$
  Note that $\Omega^*_h$ depends only on the geodesic $\gamma$ and not on the frame $\underline \nu$.
  For sufficiently small $h >0$, the set $\Omega^*_h$ can be identified with the set of  those vectors in the normal bundle
  $N\gamma = \bigcup_{t \in (-L,L)} \{   a \in T_{\gamma(t)} : \, (a, \gamma'(t))= 0\}$ which have norm less than $h$. 
  To a map $\tilde u: \Omega^*_h \to \tilde{\mathcal{M}}$ we assign an energy
  \begin{align}    \label{eq:define_energy_Eh}
    E_h(\tilde u) \coloneqq \fint_{\Omega^*_h} \dist^2(d\tilde u, SO(g, \tilde g)) \, d\Vol_\mathcal M.
  \end{align}
Here the integrand is given more explicitly  by $\dist^2( d\tilde u(p), SO(g(p), \tilde g(\tilde u(p)))$,
where\\ $SO(g(p), \tilde g(\tilde u(p)))$ denotes the set of orientation preserving linear isometries
for $T_p \mathcal M$ to $T_{u(p)} \tilde{\mathcal M}$ and the distance is taken with respect to the Riemannian
Hilbert Schmidt norm for $1$-$1$ tensors. Moreover, $\fint_{\Omega^*_h}$ denotes the average with respect to
the Riemannian volume measure on $\mathcal M$.

Our first result is a compactness result. For every family of maps  $\tilde u_h: \Omega_h^* \to \tilde{\mathcal M}$
 with $\sup_{0 < h < h_0} h^{-4} E_h(\tilde u_h) < \infty$ there exists a sequence $h_k \to 0$
 and framed unit speed geodesics $(\tilde\gamma_k, \underline{\tilde\nu_k})$ in $\tilde{\mathcal M}$ which converge
 to a framed unit speed geodesic $(\tilde \gamma, \underline{\tilde \nu})$ and have the  following properties.
 If $v_k(x_1, h x')$ is the expression of $\tilde u_k$ in geodesic normal coordinates, i.e., if 
 \begin{align*}
  (\psi_{\hat \gamma_k, \underline{\hat \nu_k}} \circ v_k)(x_1, x') =
 \tilde u_k \circ  \psi_{\gamma, \underline \nu}(x_1, h x'),
 \end{align*}
 then $v_k$ essentially satisfies
 \begin{align*}
  v_k(x_1, x') = \binom{x_1}{h x'} + \binom{h_k^2 w_k(x_1) - h_k^2\partial_{x_1} y_k(x_1) \cdot x'}{h_k y_k(x_1) +  h_k^2 Z_k(x_1) x'} +
  h_k^3 \beta_k(x_1, x'),
 \end{align*}
  where $Z_k(x_1)$ is a skew symmetric $(n-1) \times (n-1)$ matrix. Moreover,  a subsequence of  the tuple $(w_k, y_k, Z_k, \beta_k)$ 
  converges to $(w,y,Z,\beta)$ in a suitable topology.
  
  Our second result is a $\Gamma$-convergence result. 
  With respect to the convergence indicated above,  the functionals $h^{-4} E_h$ $\Gamma$-converges to a limit functional given by 
  \begin{equation}\label{eq:def-I}
  \mathcal{I}^{\tilde\gamma,\underline{\tilde\nu}}(w,y,Z,\beta)
  \coloneqq\fint_{-L}^L\fint_{B_1(0)}|\sym G(x)-\mathcal{T}(x)|^2\,dx,
\end{equation}
where 
\begin{equation}\label{eq:def-G-T}
  \begin{split}
    G(x_1,x')
    &\coloneqq\left(\begin{array}{c|c|c|c}
      \left(
        \partial_{x_1}w 
      \right)e_1+(\partial_{x_1}A)\begin{pmatrix}
        0 \\ x'
      \end{pmatrix}
      & \partial_{x_2}\beta & \dots & \partial_{x_n}\beta
    \end{array}\right)
    -\frac{1}{2}A^2, \\
    \mathcal{T}(x_1,x')&\coloneqq
    \left(\mathcal{T}^{\mathcal{R}}_{\gamma,\underline\nu}(x_1,x')
    -\mathcal{T}^{\tilde{\mathcal{R}}}_{\tilde\gamma,\underline{\tilde\nu}}(x_1,y(x_1)+x')\right),
  \end{split}
\end{equation}
and
\begin{align}  \label{eq:def-I_def-A}
  A&=\begin{pmatrix}
    0 & -\partial_{x_1}y^T \\ 
    \partial_{x_1}y & Z
  \end{pmatrix},
\end{align} 
   see Theorem~\ref{thm:main-thm} below.
   Here $\mathcal{T}^{\mathcal{R}}_{\gamma,\underline\nu}(x_1,x')$
   is a quadratic form in $x'$ which describes the rescaled deviation of the metric introduced by the 
   Fermi coordinates  in $\mathcal M$ from the Euclidean metric. 
   This deviation is a quadratic form in $x'$ and the coefficients 
   of the quadratic form depend on the curvature tensor $\mathcal R$ along $\gamma$, 
   see Lemma~\ref{lem:metric-taylor-expansion} below. 
   Similarly, $\mathcal{T}^{\tilde{\mathcal{R}}}_{\tilde\gamma,\underline{\tilde\nu}}(x_1,y(x_1)+x')$ describes
   the metric deficit in the target manifold $\tilde{\mathcal{M}}$. 
   The  $\Gamma$-limit deviates from the $\Gamma$-limit in the Euclidean setting~\cite{Mora2003} exactly by the additional term $\mathcal{T}(x)$.

  Actually, there is subtlety.
  The map $\tilde u_h$ is only in $W^{1,2}$ and hence may not be uniformly close
  to any unit speed geodesic and we may not be able to express $\tilde u_h$ in Fermi coordinates around such a geodesic.
  To overcome this difficulty, we use the same idea as in~\cite{Kroemer2025}: instead of $\tilde u_h$ we consider
  a uniformly Lipschitz approximations $\tilde v_h$ which agree with $\tilde u_h$ except 
  on a set of volume fraction $\mathcal O(h^4)$. We then define the limit $(w,y,Z,\beta)$ in terms of the 
  Lipschitz approximations $\tilde v_{h}$. It is not difficult to see that the limit does not depend on the 
  choice of the approximation. Moreover, the limit is also independent on the precise choice of the 
  geodesics $\tilde\gamma_h$, up to a natural finite dimensional equivalence relation related to 
  an infinitesimal change of framed geodesic, see Definition~\ref{def:equiv-relation}
  and Proposition~\ref{pr:unique_limit}.
  
 Using the compactness of $\tilde{\mathcal{M}}$,
 it is easy to see by standard arguments for $\Gamma$-convergence that 
  \begin{align*}
    \lim_{h \to 0}  \inf_{\tilde u} h^{-4} E_h(\tilde u) =
    \min_{(\tilde \gamma, \underline{\tilde \nu}) \in \mathcal G}
    m^{\tilde \gamma, \underline{\tilde \nu}},
  \end{align*}
  where $\mathcal G$ denotes the set of framed unit speed geodesics and where
  \begin{align*}
    m^{\tilde \gamma, \underline{\tilde \nu}} =\inf_{w,y,Z,\beta}
   \mathcal{I}^{\tilde\gamma,\underline{\tilde\nu}}(w,y,Z,\beta),
  \end{align*}
  see Theorem~\ref{thm:lim-E-minimizer}. We note that $\mathcal G$ is finite dimensional,
  since a framed unit speed geodesic  is determined by $p =\tilde \gamma(0)$
  and a positively oriented orthonormal frame at $p$. 
  
  Finally, we show that  $m^{\tilde \gamma, \underline{\tilde \nu}}$ can be expressed  as quadratic form
  of the difference of the curvature tensors in 
  $\mathcal M$ and $\tilde{\mathcal M}$ along the corresponding  framed unit speed geodesics.  Set
  \begin{align*}
    \mathcal A(x_1) \coloneqq\overline {\tilde{ \mathcal R}}(x_1,0) -  \overline{  \mathcal R}(x_1,0),
  \end{align*}
  where $\overline{ \mathcal R}$ is the pullback of the curvature tensor on $\mathcal M$ with respect to
  the Fermi coordinates for  $(\gamma, \underline \nu)$
and $\overline {\tilde{ \mathcal R}}$ is the pullback of the curvature tensor $\tilde{\mathcal R}$ of $\tilde{\mathcal{M}}$
  with respect to
  the Fermi coordinates for   $(\tilde \gamma, \underline{\tilde \nu})$.
  Then 
  \begin{align*}
    m^{\tilde \gamma, \underline{\tilde \nu}} = \fint_{-L}^L \mathbb Q( \mathcal A(x_1)) \, dx_1.
  \end{align*}  
  Here the quadratic form $\mathbb Q$ is positive definite on the space of $4$-tensors which have the symmetries
of the curvature tensor and on those tensors it can be written as
\begin{equation*}  \label{eq:decomposition_mathbbQ_intro}
\mathbb Q(\mathcal A) = \mathbb Q_1(\mathcal A^{\mathrm{par}}) +
\mathbb Q_2(\mathcal A') + \mathbb Q_3(\mathcal A''),
\end{equation*}
where
\begin{equation*}
\mathcal A^{\mathrm{par}}_{kl} = \mathcal A_{1k1l}, \quad 
\mathcal A'_{jkl} = \mathcal A_{1kjl}, \quad (\mathcal A'')_{ikjl} = \mathcal A_{ikjl}, 
\quad \text{for $i,j,k,l \ge 2$,}
\end{equation*}
and
\begin{equation*} \mathbb Q_1(\mathcal A^{\mathrm{par}})
 \frac{1}{2 (n+1)(n+3)}  | \mathcal A^{\mathrm{par}}|^2 -   \frac{1}{2(n+1)^2(n+3)}  \left(
\tr  
\mathcal A^{\mathrm{par}} \right)^2,
\end{equation*}
\begin{equation*}  \label{eq:formula_Q2}
\mathbb Q_2(\mathcal A') = \min_{b_1}   \frac{1}{2} \fint_{ \{0\} \times  B_1(0)}
 \sum_{j \ge 2}  \left(\partial_j b_1(x')  - \frac13 \mathcal A'_{jkl} x'_k x'_l \right)^2  \, dx',
\end{equation*}
and 
\begin{equation*}   \label{eq:formula_Q3}
\mathbb Q_3(\mathcal A'') = \min_{b'}   \fint_{\{0\}  \times B_1(0)} 
\sum_{i,j \ge 2}   \left((\sym  \nabla' b')_{ij}(x') - \frac{1}{6} 
\mathcal A''_{ikjl} x'_k x'_l  \right)\, dx', 
\end{equation*}
 see Theorem~\ref{th:charaterization_minimum} and Corollary~\ref{eq:Q_positive_definite}.
 Note that $\mathbb Q_1$ is positive definite since
 $(
\tr  
\mathcal A^{\mathrm{par}})^2 \le (n-1)  |\mathcal A^{\mathrm{par}}|^2$.
 The expression of $\mathbb Q_3$ captures the behavior perpendicular to 
 the geodesics $\gamma$ and $\tilde \gamma$ and  agrees with the formulae for asymptotic energy
 of  the optimal embedding of small balls derived in~\cite{Maor2019} and~\cite{Kroemer2025}.

 In term of elasticity,  the linearization of the  energy density $\dist^2(d\tilde u, SO(g, \tilde g))$ 
 leads to an energy in linear elasticity which is isotropic and has the special Poisson's ratio $0$. 
 One can start instead with a general isotropic energy. To do so,  one uses polar decomposition 
 to write $d\tilde u(p) = R(p) S(p)$,  where $R \in SO(g(p), \tilde g(p))$ and $S(p): T_p M \to T_p M$ is  symmetric 
 (with respect to the metric $g(p)$).
   Then one considers an energy density $f(\lambda_1(p), \ldots, \lambda_n(p))$
   where $\lambda_1(p), \ldots, \lambda_n(p)$  are the eigenvalues of $S(p)$ and $f$ is a symmetric function
 of its $n$ arguments, of class $C^2$ and satisfies $f(1, \ldots, 1) = 0$, $Df(1, \ldots, 1) = 0$ and $D^2 f(1, \ldots, 1) > 0$.
 Then the $\Gamma$-limit is given by
 \begin{align*}
  \fint_{\Omega} \mu |G - \mathcal T|^2 + \frac{\lambda}{2} (\mathrm{tr}  ( G - \mathcal T))^2  \, dx
 \end{align*}
 with suitable coefficients $\mu > 0$ and $\lambda >-  \frac2n \mu$.
 The particular  form of the $\Gamma$ limit arises from the fact that the space of quadratic symmetric functions
 $h: \mathbb R^n \to \mathbb R$ is spanned by the two functions $ h_1(s) =  \sum_{i=1}^n  s_i^2$ and
  $ h_2(s) = \left( \sum_{i=1}^n s_i  \right)^2$, see~\cite{Ball1984} for a general discussion of symmetric functions
  in nonlinear elasticity.

  Finally,  we briefly comment on the assumption that $\tilde{\mathcal M}$ be compact. 
  This assumption is not needed for the $\Gamma$-convergence result, but of course for the compactness statement,
  since otherwise the maps $\tilde u_h$ might escape to infinity. If $\tilde{\mathcal{M}}$ has constant curvature we can use 
  isometric translations to restore compactness. In particular we may take $\tilde{\mathcal{M}} = \mathbb R^n$. 
  If $\tilde{\mathcal{M}}$ is asymptotic to a constant curvature space at $\infty$, then in the spirit of concentration compactness
   the limit $\lim_{h \to 0 } h^{-4} \inf_{\tilde u} E_h(\tilde u)$ is either given by 
   $ \min_{(\tilde \gamma, \underline{\tilde \nu}) \in \mathcal G}
  m^{\tilde \gamma, \underline{\tilde \nu}}$  
  or by a limit problem on the space of constant curvature.

  \bigskip

\section{Preliminaries}

\subsection{Fermi coordinates around a geodesic in $\mathcal{M}$}

Let $(\mathcal{M},g)$ be a complete $n$-dimensional Riemannian manifold and let $\gamma:\mathbb{R}\rightarrow\mathcal{M}$ be a geodesic with $|\gamma'|=1$.
Let $\underline\nu\coloneqq(\nu_2,\dots,\nu_n)$ be a parallel orthonormal frame of the normal bundle of $\gamma$. 
Then the Fermi coordinates $\psi_{\gamma,\underline\nu}$ are given by 
\begin{equation}\label{eq:def-fermi-coords}
  \psi_{\gamma,\underline\nu}(x_1,x')\coloneqq\exp_{\gamma(x_1)}\left( 
    \sum_{k=2}^n x_k\nu_k(x_1)
  \right).
\end{equation}
We call $(\gamma,\underline\nu)$ a \emph{framed unit speed geodesic}.

\begin{remark}  \label{re:bilipschitz_fermi} The definition of 
$ \psi\coloneqq  \psi_{\gamma,\underline\nu}$ implies that 
that $D\psi(x_1,0)$ is an isometry from $\mathbb R^n$ to $T_{\gamma(x_1)} \mathcal M$.
  If $\gamma$ has no self-intersections on the  interval $[-L,L]$, 
  then it follows from the inverse function theorem and a compactness argument that there exists $\eps>0$
  such that $\psi$ is a diffeomorphism from  $(-L,L) \times B_\eps(0)$ 
   to its image with   $|D\psi|\leq 2$ and $|(D\psi)^{-1}|\leq 2$
   and such that $\psi|_{(-L,L) \times B_\eps(0)}$ is bilipschitz. 
   \end{remark}

   \begin{remark} For the main results of this paper we consider {\it oriented} Riemannian manifolds
   and in this setting we will always assume that the frame $(\gamma'(x_1), \nu_2(x_1), \ldots, \nu_n(x_1))$
   is positively oriented.
   \end{remark}

We denote by $\overline{g}$  the pullback metric of $g$ under $\psi_{\gamma,\underline\nu}$, i.e.,
\begin{align*}
  \overline{g}(x)(X,Y)
  &=(\psi_{\gamma,\underline\nu}^*g)(x)(X,Y)
  =g_{\psi_{\gamma,\underline\nu}(x)}(d\psi_{\gamma,\underline\nu}(x)X,d\psi_{\gamma,\underline\nu}(x)Y)
  \, .
\end{align*}
Furthermore, we let $\mathcal{R}$ denote the Riemann curvature tensor on $\mathcal{M}$, i.e. 
\begin{equation}\label{eq:def-R_new}
  \mathcal{R}(Z,X,Y)\coloneqq\nabla_{X}\nabla_{Y}Z-\nabla_{Y}\nabla_{X}Z-\nabla_{[X,Y]}Z.
\end{equation}

This is a $1$-$3$ tensor and we denote the associated $0$-$4$ tensor also by $\mathcal R$, i.e., 
\begin{equation} \label{eq:riemann_0_4}
\mathcal R(W,Z,X,Y) = g(W, \mathcal R(Z,X,Y)).
\end{equation}
 We also define the curvature operator by
\begin{equation}  \label{eq:riemann_curvature_operator}
R(X,Y) Z := R(Z,X,Y).
\end{equation}

\begin{lemma}[{\cite{Manasse1963},\cite[Chapter 9]{Gray2004}}]\label{lem:metric-taylor-expansion}
 Let $(\mathcal M, g)$ be a smooth Riemannian manifold.
 Given a framed unit speed geodesic $(\gamma, \underline{\nu})$ on $[-L,L]$ in $\mathcal M$ there exists
 $\rho >0$ and $C>0$ such that the Fermi coordinates $\psi_{\gamma, \underline{\nu}}$ are defined
 in $\Omega_\rho = (-L, L) \times B_\rho$ and the pullback metric $\overline g =\psi_{\gamma, \underline{\nu}}^*g$
 and its Christoffel symbols $\overline\Gamma$ satisfy, for $|x'| < \rho$, 
  \footnote{In our earlier paper~\cite{Kroemer2025} we used the definition $\mathcal R(X,Y,Z) = \mathcal R(X,Y) Z$
  instead of~\eqref{eq:def-R_new}. This gives a formula with the opposite sign in~\eqref{eq:taylor-metric}}
  \begin{equation}\label{eq:taylor-metric}
    \begin{split}
      \left|\overline{g}_{11}(x)-1 +\sum_{k,l=2}^n\overline{\mathcal{R}}_{1k1l}(x_1)x'_k x'_l\right|&\leq C|x'|^3, \\
      \left|\overline{g}_{1j}(x) + \frac{2}{3}\sum_{k,l=2}^n \overline{\mathcal{R}}_{jk1l}(x_1)x'_k x'_l\right|&\leq C|x'|^3,\quad j=2,\dots,n, \\
      \left|\overline{g}_{ij}(x) - \delta_{ij} + \frac{1}{3}\sum_{k,l=2}^n \overline{\mathcal{R}}_{jkil}(x_1)x'_k x'_l\right|&\leq C|x'|^3,\quad i,j=2,\dots,n,
    \end{split}
  \end{equation}
  and
  \begin{equation}\label{eq:taylor-christoff}  
    \begin{split}
      \left|\overline\Gamma_{ij}^k(x) +
      \frac{1}{3} \left(  \overline{\mathcal R}_{kjil}(x_1) +  \overline{\mathcal R}_{kijl}(x_1) \right) x'_l
          \right|
        &\leq C|x'|^2 \qquad\text{if $i,j\geq 2$,} \\
      \left|\overline\Gamma_{i1}^k(x)   
      + \overline{\mathcal R}_{ki1l}(x_1)  x'_l
        \right|
        &\leq C|x'|^2  \, ,      \end{split}
  \end{equation}
  with $\overline{\mathcal{R}}_{jkil}(x_1) = \psi_{\gamma,\underline\nu}^*\mathcal R(x_1,0)(e_j, e_k, e_i, e_l)$ where
  $e_1, \ldots, e_n$ is the canonical basis for $\mathbb R^n$.
  If $K \subset \mathcal M$ is compact then the constants $\rho> 0$ and $C$ can be 
  chosen uniformly for all unit speed geodesics $\gamma$ with $\gamma((-L,L)) \subset K$.
\end{lemma}

\begin{proof}[Proof of Lemma~\ref{lem:metric-taylor-expansion}]
 This follows by  standard Jacobi field estimates. Indeed, the Taylor expansions for
$\overline g$ and $\overline\Gamma$ can be found, for example,  in~\cite[eq.\ (66a)-(66c)]{Manasse1963},
 and~\cite[eq.\ (58b) and (58c)]{Manasse1963}, respectively. 
 Higher order Taylor expansions of the metric are established in~\cite[Chapter 9]{Gray2004}.
  The uniformity of $C$ and $\rho$ follows from the fact
 that $\psi_{\gamma, \underline{\nu}}$ is smooth (by standard results ODEs for  geodesic and parallel transport)
 and the fact that $K$ can be covered by finitely many charts.
  For the convenience of the reader we recall the proof in Appendix~\ref{se:fermi_proof}. The argument shows that $\rho$ and $C$
  can be controlled by the $C^0$ norm of the curvature tensor $\mathcal R$ and its covariant derivative in a neighborhood of $\gamma([-L,L])$.
\end{proof}

In order to simplify the notation below, we associate a matrix 
$\mathcal T^{\mathcal A}_{\gamma,\underline\nu}\in\mathbb{R}^{n\times n}$ to every tuple $(\mathcal{A},\gamma,\underline\nu)$,
where $\mathcal{A}$ is a $(3,1)$-tensor on $\mathcal{M}$ and $(\gamma,\underline\nu)$ is a framed unit speed geodesic, as follows.
For $i,j \in \{2, \ldots, n\}$ we set    
\begin{equation}\label{eq:def-T}
  \begin{split}
    (\mathcal{T}^{\mathcal{A}}_{\gamma,\underline\nu})_{11}(x_1, x')&=-\frac{1}{2}\sum_{k,l=2}^n\overline{\mathcal{A}}_{1k1l}(x_1)x'_k x'_l, \\
    (\mathcal{T}^{\mathcal{A}}_{\gamma,\underline\nu})_{1j}(x_1, x')&=-\frac{1}{3}\sum_{k,l=2}^n\overline{\mathcal{A}}_{jk1l}(x_1)x'_k x'_l, \\
    (\mathcal{T}^{\mathcal{A}}_{\gamma,\underline\nu})_{j1}(x_1, x')&=-\frac{1}{3}\sum_{k,l=2}^n\overline{\mathcal{A}}_{1kjl}(x_1)x'_k x'_l, \\
    (\mathcal{T}^{\mathcal{A}}_{\gamma,\underline\nu})_{ij}(x_1, x')&=-\frac{1}{6}\sum_{k,l=2}^n\overline{\mathcal{A}}_{jkil}(x_1)x'_k x'_l.
  \end{split}
\end{equation}
Here $\overline{\mathcal{A}}=\psi_{\gamma,\underline\nu}^*\mathcal{A}$ is the pullback 
 of the  restriction of the  tensor $\mathcal A$ to $\gamma$.

For  a metric $\overline g$ on $\mathbb R^n$ we define a positive 
definite  symmetric  matrix $A_{\overline g}$ by 
$(A_{\overline g})_{ij }\coloneqq  \overline g_{ij} \coloneqq  \overline g(e_i, e_j)$. 
By  $A_{\overline g}^{1/2}$ 
we denote the unique positive definite symmetric matrix $M$ with
$M^2 = A_{\overline g}$ and we set 
$A_{\overline g}^{-1/2} \coloneqq  M^{-1}$. We use the shorthand notations
\begin{equation}  \label{eq:definition_sqrt_g}  (\overline g^{1/2})_{ij} \coloneqq  (A_{\overline g}^{1/2})_{ij}, \quad  (\overline g^{-1/2})_{ij} \coloneqq  (A_{\overline g}^{-1/2})_{ij}. 
\end{equation}
Note that  neither $A_{\overline g}^{1/2}$ nor 
$A_{\overline g}^{-1/2}$ are $0$-$2$ or  $1$-$1$ tensors. 
Indeed,  if $\varphi(x) = \lambda x$, then $\varphi^*\overline g = \lambda^2 \overline g$ 
and  $(A_{\varphi^*\overline g})^{1/2} = \lambda A_{\overline g}^{1/2}$
while $(A_{\varphi^*\overline g})^{-1/2} = \lambda^{-1} A_{\overline g}^{-1/2}$ which is incompatible with the behaviour of $0$-$2$
or $1$-$1$ tensors.

An immediate consequence of Lemma~\ref{lem:metric-taylor-expansion} is
that the Taylor expansions of the matrices $\overline{g}^{1/2}$ and $\overline{g}^{-1/2}$ are given by 
   \begin{equation}\label{eq:sqrt-metric-taylor}
    \begin{split}
      \left|(\overline{g}^{1/2})(x)-\Id-\mathcal{T}^{\mathcal{R}}_{\gamma,\underline\nu}(x)\right|&\leq C|x'|^3, \\
    \end{split}
  \end{equation}
  and 
  \begin{equation}\label{eq:sqrt-metric-inverse-taylor}
    \begin{split}
      \left|(\overline{g}^{-1/2})(x)-\Id+\mathcal{T}^{\mathcal{R}}_{\gamma,\underline\nu}(x)\right|&\leq C|x'|^3 \, . \\
    \end{split}
  \end{equation}

\subsection{Lift of maps from a small geodesic tube in $\mathcal{M}$}

If $\tilde \gamma: (-L,L) \to \tilde{\mathcal M}$ is a unit speed geodesic  in a smooth Riemannian manifold $\tilde{\mathcal M}$,
then in general the Fermi coordinates associated with $\tilde \gamma$ are not globally invertible on a small cylinder,
because $\tilde \gamma$ may have self-intersections. 
We will overcome this difficulty by the following lifting result.

\begin{lemma}\label{lem:existence-lift_new}
  Let $\tilde{\mathcal{M}}$ be a  smooth  $n$-dimensional Riemannian manifold and let $K \subset \tilde{\mathcal M}$ be compact. Let $(\tilde \gamma, \underline{\tilde \nu})$ be a framed unit speed geodesic on $(-L,L)$ with $\tilde \gamma((-L,L)) \subset K$
and let $\tilde \psi \coloneqq \psi_{\tilde \gamma,  \underline{\tilde \nu}}$ be the associated Fermi map. 

For $\eta > 0$ define
\begin{eqnarray}
U_\eta &\coloneqq&  (-L -\eta, L + \eta)  \times B_\eta(0) \subset \mathbb R \times \mathbb R^{n-1}, \\
V_\eta &\coloneqq & \{ (s, q) \in (-L , L) \times \tilde{\mathcal M} : \dist_{\tilde{\mathcal{M}}}(q, \tilde \gamma(s)) < \eta\}.
\end{eqnarray}

Then there exists a $\delta > 0$  such that $\tilde \psi$ can be extended to $U_\delta$,
the operator norm satisfies $|D\tilde \psi| \le 2$ and there exists a unique smooth map $\tilde \Phi: V_{\delta/2} \to U_{\delta}$
satisfying
\begin{align*}
  \tilde \psi (\tilde \Phi(s,q)) = q\qquad\text{and}\qquad\tilde\Phi(s,\tilde\gamma(s))=(s,0) 
  \quad \text{for all $(s,q) \in V_{\delta/2}$.}
\end{align*}
Moreover $\partial_s\tilde  \Phi = 0$ in $V_{\delta/2}$,  the operator norm satisfies $|D\tilde \Phi| \le 2$,  and $\tilde \Phi$ is Lipschitz with
 Lipschitz constant at most  $6$.  
Moreover,   $\delta >0$ can be chosen independently of $L$ and of  the framed unit speed geodesic
$(\tilde \gamma, \underline{\tilde \nu})$ as long as $\tilde \gamma((-L,L)) \subset K$.
\end{lemma}

We will frequently use the following  immediate consequence of Lemma~\ref{lem:existence-lift_new}.

\begin{lemma} \label{le:lifting_maps_new} Let $\mathcal M$ and $\tilde{\mathcal   M}$
be   smooth, complete $n$-dimensional Riemannian manifolds and let $(\gamma, \underline{\nu})$
and $(\tilde \gamma, \underline{\tilde \nu})$ be framed unit speed geodesics on $[-L,L]$ in $\mathcal M$
and $\tilde{\mathcal M}$, respectively.  Let $\psi$ and $\tilde \psi$ denote the Fermi coordinate maps
for the framed geodesics. Let $\delta > 0$ be as in Lemma~\ref{lem:existence-lift_new} and let $\eps > 0$ be as in Remark~\ref{re:bilipschitz_fermi}.
If a map $\tilde v: \psi(\Omega_\eps) \to \tilde{\mathcal{M}}$ satisfies 
\begin{equation}  \label{eq:close_to_tube}
  \sup_{x \in \Omega_\eps} d_{\tilde{\mathcal{M}}}( v (\psi(x)), \tilde \gamma(x_1)) < \frac{1}{2} \delta
\end{equation}
then there exists a lift
$v: (-L, L) \times B_\eps(0) \to \mathbb R \times B_\delta(0)$
such that 
$$ \tilde \psi \circ v = \tilde v \circ \psi \quad \text{and} \quad 
 |v(x_1) - x_1| < \delta \quad \text{for all $x_1 \in (-L,L)$.}
 $$
If, in addition, $\tilde v$ is Lipschitz with Lipschitz constant $l$, then $v$ is Lipschitz with Lipschitz
at most  $4 l$.
\end{lemma}

\begin{proof} It follows from  \eqref{eq:close_to_tube} that $(v \circ \psi)(\Omega_\eps) \subset V_{\delta/2}$.
Thus we can take $v(x) = \tilde \Phi\big(x_1, v (\psi(x))\big)$. To show the assertion on  Lipschitz constant, 
we use that Lipschitz functions are differentiable almost everywhere. Then the chain rule
and the bounds $|D\tilde \Phi| \le 2$ and $|D \psi^{-1}| \le 2$ imply that 
$|Dv(x)| \le 4 l$ almost everywhere. Since $\Omega_\eps$ is convex, it follows that $v$ has Lipschitz constant at most $4l$.
\end{proof}

\begin{proof}[Proof of Lemma~\ref{lem:existence-lift_new}]
  Since $\tilde\gamma((-L,L))$ is contained in a compact subset there exists a $\delta_0 > 0$ such that $\tilde \gamma$ and the orthonormal frame can be extended to 
$(-L - 2 \delta_0, L + 2 \delta_0)$. Let $s  \in [-L,L]$.  
Then $D\tilde \psi(s,0)$ is a linear  isometry from $\mathbb R^n$ to $T_{\tilde \gamma(s)} \tilde{\mathcal M}$. By the inverse function theorem, there exists a $\delta \in (0, \delta_0)$ such that 
the restriction of $\tilde \psi$ to $(s- 2\delta, s + 2\delta) \times B_\delta(0)$ satisfies $|D\tilde \psi| \le 2$ and
$|(D\tilde \psi)^{-1}| \le 2$, 
 is injective, and $\tilde \psi \left((s- \delta, s + \delta) \times B_\delta(0)\right) 
 \supset B_{\delta/2}(\tilde \gamma(s))$.
 Thus there exists a smooth map 
 \begin{align*}
  \tilde \Psi_s: B_{\delta/2}(\tilde \gamma(s)) \to (s- \delta, s + \delta) \times B_\delta(0)
 \end{align*}
 with $\tilde \psi \circ \tilde  \Psi_s = \id$ and $|D \tilde \Psi_s| \le 2$. 
By the usual compactness argument, such a $\delta > 0$ can be chosen uniformly for all $s \in [-L,L]$
and for all framed unit speed geodesics which satisfy  $\tilde\gamma((-L,L)) \subset K$.
Now define $\tilde \Phi: (-L,L) \times V_{\delta/2} \to U_\delta$ by
\begin{align*}
  \tilde \Phi(s,q) = \tilde \Psi_s(q).
\end{align*}
Then $\tilde \Phi(s,q) \in (s-\delta, s+ \delta) \times B_\delta(0)$ and 
\begin{align*}
  \tilde \psi (\Phi(s,q)) = q.
\end{align*}

Assume that $|s' - s| < \min(\delta, \delta/2- \dist_{\tilde{\mathcal{M}}}(q,s))$.  Then $q \in B_{\delta/2, \tilde \gamma(s')}$,
since $\dist_{\tilde{\mathcal{M}}}(\tilde\gamma(s), \tilde\gamma(s')) \le |s-s'|$. Thus $\tilde \Phi(s',q)$ is defined and
$\tilde \Phi(s',q) \in  (s'- \delta, s' + \delta) \times B_\delta(0) \subset  (s- 2\delta, s + 2\delta) \times B_\delta(0).$
Since $\tilde \psi$ is injective on  $(s- 2\delta, s + 2\delta) \times B_\delta(0)$ it follows that
\begin{equation}  \label{eq:lift_constant_in_s}
\Phi(s',q) = \Phi(s,q) \quad \text{whenever $|s' - s| < \min(\delta, \delta/2- \dist_{\tilde{\mathcal{M}}}(q,s))$.}
\end{equation}
 Hence  $\partial_s \Phi(s,q) =0$ and
 $\tilde \Phi$ is smooth function on $V_{\delta/2}$. Moreover, $|D\tilde \Phi| = |D_q \tilde  \Phi|
= |D\tilde \Psi_s| \le 2$. 

\medskip

To show uniqueness, assume that $\hat \Phi : V_{\delta/2} \to U_\delta$ is smooth and satisfies
 $\hat \Phi(s, \gamma(s)) = (s,0)$
for $s \in (-L,L)$ and $\tilde \psi (\hat \Phi(s,q)) = q$ for all $(s,q)$ in $V_{\delta/2}$. 
Fix $\overline s \in (-L,L)$ and let $W \coloneqq  \{ q \in B_{\delta/2}(\gamma(\overline s)) : \hat \Phi(\overline s, q) = \tilde \Phi(\overline s, q) \}$.
Then $W$ is open, since 
$\tilde \Phi(\{\overline{s}  \}\times W) \subset (\overline s- \delta, \overline s + \delta) \times B_{\delta}(0)$
and $\tilde \psi$ is injective on $(\overline s- \delta, \overline s + \delta) \times B_{\delta}(0)$. Moreover, 
$W$ is relatively closed in $B_{\delta/2}(\gamma(\overline s))$ since $\hat \Phi$ and $\tilde \Phi$ are continuous. 
Finally $W$ is not empty since $\gamma(\overline s) \in W$, by assumption. Since $B_{\delta/2}(\gamma(\overline s))$
is connected (for sufficiently small $\delta >0$) we get $W = B_{\delta/2}(\gamma(\overline s))$. This holds for all 
$\overline s \in (-L,L)$ and hence $\hat \Phi = \tilde \Phi$.

\medskip

Finally, we show that $\tilde \Phi$ is globally Lipschitz.
By reducing $\delta > 0$, if needed, we may assume, in addition, that the balls $B_{\delta/2}(\tilde \gamma(s))$
are geodesically convex for all $ s \in (-L,L)$. 
To prove the Lipschitz estimate, 
let $(s,q), (s', q') \in V_{\delta/2}$. By symmetry, we may assume that 
$\dist_{\tilde{\mathcal{M}}}(q', \tilde\gamma(s')) \le \dist_{\tilde{\mathcal{M}}}(q, \tilde\gamma(s))$.

 Assume first that $|s'-s| < \dist_{\tilde{\mathcal{M}}}(q, \tilde \gamma(s))$.
Consider the length minimizing  unit speed geodesic $\eta$ from $q$ to $\tilde \gamma(s)$ and let $q'' \coloneqq \eta(|s'-s|)$. 
Then 
$$ \dist_{\tilde{\mathcal{M}}}(q'', \tilde \gamma(s)) =  \dist_{\tilde{\mathcal{M}}}(q, \tilde \gamma(s)) - |s'-s|$$
 and hence
$ \dist_{\tilde{\mathcal{M}}}(q'', \tilde \gamma(s')) \le \dist_{\tilde{\mathcal{M}}}(q, \tilde \gamma(s')) < \delta/2$ 
since $|\tilde \gamma(s') - \tilde \gamma(s)| \le |s'-s|$.  Thus $\tilde \Phi(s,q'') = \tilde \Phi(s,q'')$.
 Since $|D\tilde \Phi| \le 2$ and since the balls $B_{\delta/2}(\tilde \gamma(\sigma))$  and 
 $B_{\delta/2}(\tilde \gamma(s'))$ are geodesically convex, we get 
 \begin{align*}
& \,   |\tilde \Phi(s,q) - \tilde \Phi(s',q)| \le |\tilde \Phi(s,q) - \tilde \Phi(s, q'')| +  |\tilde \Phi(s',q') - \tilde \Phi(s', q'')| \\
  \le  & \, 2 |s'-s| +  2\,  \dist_{\tilde{\mathcal{M}}}(q',q'') \le 4 |s-s'| + 2  \, \dist_{\tilde{\mathcal{M}}}(q',q) < 6 \,  \dist_{\tilde{\mathcal{M}}}(q',q).
  \end{align*}
  If $|s'-s| \ge \dist_{\tilde{\mathcal{M}}}(q, \tilde \gamma(s))$,  we set $q'' = \tilde \gamma(s)$ and $q''' = \tilde \gamma(s')$.
  Then $\tilde \Phi(s,q'') = (s, 0)$ and $\tilde \Phi(s',q''') = (s',0)$. Hence
   \begin{align*}
  |\tilde \Phi(s,q) - \tilde \Phi(s',q)| \le  & \, |\tilde \Phi(s,q) - \tilde \Phi(s, q'')| +  |\tilde \Phi(s',q') - \tilde \Phi(s', q''')| 
+ |s-s'|\\
  \le  & 2 \,  \dist_{\tilde{\mathcal{M}}}(q, \tilde \gamma(s)) + 2\,  \dist_{\tilde{\mathcal{M}}}(q', \tilde \gamma(s')) + |s' -s| \\
  \le & 4 \,  \dist_{\tilde{\mathcal{M}}}(q, \gamma(s)) +  |s'-s|   \le 5 |s'-s|. \
  \end{align*}
  Here we used the assumption $\dist_{\tilde{\mathcal{M}}}(q', \gamma(s')) \le \dist_{\tilde{\mathcal{M}}}(q, \gamma(s))$.
  Combining the estimates in the two cases, we see that
  $\tilde \Phi$ is Lipschitz with Lipschitz constant at most $6$.
 \end{proof}

\subsection{Estimates for the center of mass in normal coordinates}

\begin{definition}[Riemannian center of mass~\cite{Karcher1977}]   \label{def:center-of-mass}
  Let $(A, \mu)$ be a probability space, let $\tilde{\mathcal{M}}$ be a complete Riemannian manifold
  and let $B_\rho\subset\tilde{\mathcal{M}}$ be a convex geodesic ball.
  For a measurable function $f:A\rightarrow B_\rho$ define 
  \begin{equation}\label{eq:def-riemannian-center-of-mass-energy}
    P_f:\overline{B}_\rho\rightarrow\mathbb{R},\qquad 
    P_f(m)=\frac{1}{2}\int_A\dist^2(m, f(a))\,d\mu(a) \, .
  \end{equation}
  We call $\mathcal{C}_f\coloneqq \argmin_{m\in\overline{B}_\rho}P_f(m)$ the Riemannian center of mass of $f$ in $A$. 
\end{definition}

\begin{remark}
  The function $P_f$ in~\eqref{eq:def-riemannian-center-of-mass-energy} is convex 
  and has a unique minimum which lies in the interior of $\overline{B}_\rho$.
  A proof of this result and more details on the Riemannian center of mass can be found in~\cite{Karcher1977}.
\end{remark}

\begin{proposition}\label{prop:comparison_exp_maps}
  Let $\mathcal{N}$ be a complete, smooth Riemannian manifold. Let $p \in\mathcal{N}$.
  Then there exist $\rho > 0$ and $C' > 0$ with the following property.
  The ball $B_{2 \rho}(p)$ is geodesically convex and the exponential map $\exp_p$ is a smooth diffeomorphism
  from $B_{2\rho}(0) \subset T_p N \to B_{2 \rho}(p)$. 
  Moreover, the map  $\psi \coloneqq  \exp_p^{-1} \circ \exp_q$ satisfies, for every $q \in B_\rho(p)$ and every $v \in T_q N$ with $|v|\le \rho$
  \begin{equation}  \label{eq:comparison_exp_p_and_exp_q}
    |\psi(v) - \psi(0) - d\psi(0) v|  \le C' (d(p,q) + |v|) |v|^2.
  \end{equation}
  If $N$ is compact, then $\rho$ and $C$ can be chosen independently of the point $p \in N$. 
\end{proposition}

\begin{proof}
  Consider the map $c:[0,1]\to N$ given by  $c(t) = \exp_q tv$.
  Then $c$ is a geodesic arc which connects $c(0) = q$ and
$c(1) = \exp_q v$. Moreover,  $|c'(t)| = |v|$ and thus $c([0,1]) \subset B_{2\rho}(p)$.
Hence we can define $\gamma\coloneqq\varphi\circ c$,  
where $\varphi = \exp_p^{-1}$. Then $|\gamma'(t)| \le C |c'(t)| \le C |v|$.

Let $\Gamma$ denote the Christoffel symbols corresponding to the chart $\varphi$. 
Then 
\begin{equation}\label{eq:geodesic_normal}
  \gamma''(t) = - \Gamma(\gamma(t))(\gamma'(t), \gamma'(t))
\end{equation}
and 
\begin{equation}\label{eq:geodesic_initial_normal}
  \gamma(0) = \varphi(q)= \psi(0), \quad \gamma'(0) = d\psi(0) v,
  \quad \gamma(1) = \psi(v).
\end{equation}
Since the chart  $\varphi$ defines normal coordinates,  we have  $\Gamma(0) =0$. 
Moreover $D\Gamma$ is controlled in terms of the second derivatives of the metric. Hence
\begin{equation}
  |\Gamma(\gamma(t))(\gamma'(t), \gamma'(t))| \le  C |\gamma(t)| \, |\gamma'(t)|^2 \le C' 
  (d(p,q) + |v|)|v|^2.
\end{equation}
Hence the assertion~\eqref{eq:comparison_exp_p_and_exp_q} follows from~\eqref{eq:geodesic_normal}
and~\eqref{eq:geodesic_initial_normal}.
The fact that $\rho$ and $C$ can be chosen uniformly for compact $\mathcal{N}$ follows from the usual argument.
\end{proof}

\begin{corollary}\label{cor:estimate-center-of-mass} Let $p$,$\rho$  and $C'$ be as in Proposition~\ref{prop:comparison_exp_maps}. Let $r \le \rho/2$.
Let $(A, \mu)$ be a probability space  and let $f: A \to B_r(p)$ be a measurable map. 
Then the center of mass $\mathbf{C}_f$ satisfies
\begin{equation}
  \left| \exp_p^{-1}\mathbf{C}_f - \int_A  \exp_p^{-1} f(a) \, d\mu(a)\right| \le C' r^3.
\end{equation}
\end{corollary}

\begin{proof}
  Let $q = \mathbf{C}_f$.
  Then by~\cite[Thm.\ 1.2]{Karcher1977} we have
\begin{equation}\label{eq:com_error_definition}
  \int_A  \exp^{-1}_q f(a) \, d\mu(a) = 0.
\end{equation}
Recall that $\psi = \exp_p^{-1} \circ \exp_q$. Thus
$ \exp_p^{-1} \circ f = \psi \circ (\exp_q^{-1} \circ f)$.
It follows from \eqref{eq:com_error_definition} and   \eqref{eq:comparison_exp_p_and_exp_q}
that 
$$
  \left|  \int_A  \exp_p^{-1} f(a) \, d\mu(a) - \psi(0)  \right| \le C' r^3.
  $$
  Since $\psi(0) = \exp_p^{-1}(q) = \exp_p^{-1}  \mathbf{C}_f$, this concludes the proof.
\end{proof} 

\subsection{Setup}

Let $L>0$.
For $h>0$ define $\Omega_{h}\coloneqq(-L,L)\times B_{h}(0)\subset\mathbb{R}^n$,
and let $\Omega\coloneqq(-L,L)\times B_1(0)$.
We define the rescaled gradient $d_h v$ of a function $v:\Omega\rightarrow\mathbb{R}^n$ by
\begin{align*}
  d_h v\coloneqq\left( \partial_{x_1}v\bigg|\frac{1}{h}\partial_{x_2}v\bigg|\dots\bigg|\frac{1}{h}\partial_{x_n}v \right).
\end{align*}
If $I_h:\Omega\rightarrow\Omega_h$ is given by
$I_h(x_1,x')=(x_1,hx'),\,v_h:\Omega_h\rightarrow\mathbb{R}^n$ and $v=v_h\circ I_h$ then 
\begin{align*}
  d_h v=(dv_h)\circ I_h.
\end{align*}
Let $(\gamma,\underline\nu)$ be a framed unit speed geodesic in a complete Riemannian manifold $\mathcal{M}$ and let 
\begin{equation}\label{eq:def-omega-h-star}
  \Omega_h^*=\psi_{\gamma,\underline\nu}(\Omega_h),
\end{equation}
where $\psi_{\gamma,\underline\nu}$ is the Fermi coordinate system defined in~\eqref{eq:def-fermi-coords}.
For functions $\tilde u:\Omega_h^*\rightarrow\tilde{\mathcal{M}}$ we define the energy
\begin{equation}\label{eq:def-energy}
  E_{h}(\tilde u)\coloneqq\fint_{\Omega_{h}^*}\dist{}^2(d\tilde u,SO(g,\tilde g))\,d\Vol_g.
\end{equation}

\begin{remark}
  If $(\gamma,\underline\nu),\,(\tilde\gamma,\tilde{\underline\nu})$ are as in Lemma~\ref{lem:existence-lift_new} and $v$ is a lift of $\tilde u$
  with respect to these geodesics, i.e.\ if \ $\tilde\psi_{\tilde\gamma,\underline{\tilde\nu}}\circ v= \tilde u\circ\psi_{\gamma,\underline\nu}$, then, 
  \begin{equation}\label{eq:rewrite-energy-SOn}
    E_h(\tilde u)=\frac{1}{\int_{\Omega_h}\sqrt{\det\overline{g}}\,dx}
    \int_{\Omega_h}\dist{}^2\left(   (\overline{\tilde{g}}\circ  v)^{1/2}  \, \, dv\, \, \overline{g}^{-1/2},SO(n)\right)\sqrt{\det\overline{g}}\,dx,
  \end{equation}
    see for example~\cite[Eq.\ 2.5]{Kroemer2025}.
    Here $dv$ is identified with the matrix with components $(dv)_{ij} = \partial_{x_j} v_i$ and the 
    matrices $\overline g^{-1/2}$ and $\overline{\tilde g}^{1/2}$ are defined in  \eqref{eq:definition_sqrt_g}.
\end{remark}

If $|(v_2,\dots,v_n)|\leq Ch$, then the metrics in~\eqref{eq:rewrite-energy-SOn} are $O(h^2)$ close to the identity. 
This will allow us to use the following compactness result in the Euclidean setting. 

\begin{theorem}[{\cite[Thm.\ 2.2]{Mora2003}}]\label{thm:muller-mora-3d}
  Let $n\geq 2$ and let $h_k\rightarrow 0,\,\overline{v}_k:\Omega\rightarrow\mathbb{R}^n$ such that
  \begin{equation}\label{eq:lem-3d-uniform-energy-bound}
    \sup_{k}\frac{1}{h_k^4}\int_{\Omega}\dist^2(d\overline{v}_k,SO(n))\,dx<\infty.
  \end{equation}
  Then there exist maps $Q_{k}:(-L,L)\rightarrow SO(n),\,\tilde{Q}_k:(-L,L)\rightarrow\mathbb{R}^{n\times n}$ with $|\tilde{Q}_k|\leq c$ 
  and constants $\overline{Q}_{k}\in SO(n),\,c_{k}\in\mathbb{R}^n$ such that the functions
  $v_{k}\coloneqq\overline{Q}_{k}^T \overline{v}_{k}-c_{k}$ satisfy
  \begin{equation}\label{eq:muller-mora-matrix-estimates}
    \begin{split}
      \|d_{h_k}v_{k}-Q_{k}\|_{L^2(\Omega)}
      &\leq Ch_k^2, \\
      \|Q_k-\tilde{Q}_k\|_{L^2(-L,L)}
      &\leq Ch_k^2,
      \qquad\|\nabla\tilde{Q}_k\|_{L^2(-L,L)}\leq Ch_k, \\
      \|Q_{k}-\Id\|_{L^\infty}
      &\leq Ch_k.
    \end{split}
  \end{equation}
  Furthermore, set
  \begin{align*}
    w^k(x_1)&\coloneqq\int_{B_1(0)}\frac{(v_k)_1(x)-x_1}{h_k^2}\,dx', \\
    y^k_{i}(x_1)&\coloneqq\int_{B_1(0)}\frac{(v_k)_i(x)}{h_k}\,dx',\quad\text{$i=2,\dots,n$,} \\
    z_{ij}^k(x_1)&\coloneqq    - \frac{1}{\mu}\int_{B_1(0)}\frac{x_i(v_k)_j(x)-x_j(v_k)_i(x)}{h_k^2}\,dx',
    \quad\text{$i,j=2,\dots,n$,}
  \end{align*}
  where 
   $\mu\coloneqq  2 \int_{B_{h_k}(0)}  x_n^2 \,dx'$. 
  By $Z^k$ we denote the skew-symmetric matrix with entries $z^k_{ij}$.
  Then, up to subsequences, it holds:
  \begin{enumerate}
    \item\label{it:mora-n_convergence_w}
    $w^k\rightharpoonup w$ weakly in $W^{1,2}(-L,L)$,
  \item $y^k_{i}\rightarrow y_i$ in $W^{1,2}(-L,L)$ with $y_i\in W^{2,2}(-L,L)$
    for $i=2,\dots,n$,
  \item $z_{ij}^k\rightharpoonup z_{ij}$ weakly in $W^{1,2}(-L,L)$.
  \item\label{item:muller-mora-conv-A} There exists a matrix $A\in W^{1,2}((-L,L);\mathbb{R}^{n\times n})$ such that
    $(d_{h_k}v_k-\Id)/h_k\rightarrow A$ in $L^2(\Omega)$,
  \item\label{item:muller-mora-prop-5} $\sym(Q_k-\Id)/h_k^2\rightarrow A^2/2$ uniformly on $(-L,L)$ with
    \begin{equation}\label{eq:def-A}
      A=\begin{pmatrix}
        0 & -\partial_{x_1}y^T \\
        \partial_{x_1}y & Z
      \end{pmatrix},
    \end{equation}
    where $Z$ is the skew-symmetric matrix with entries $z_{ij}$ and 
    $y=\begin{pmatrix}
      y_2 \\ \vdots \\ y_n
    \end{pmatrix}$.
  \item\label{it:convergence_beta} The sequence $\beta^k$ defined by
    \begin{align*}
      \beta^k_1(x)&\coloneqq
      \frac{1}{h_k}\left(
        \frac{(v_k)_1(x)-x_1}{h_k^2}
        -w^k(x_1)+
        \sum_{i=2}^n x_i\partial_{x_1}y_i^k(x_1)
      \right), \\
      \beta^k_j(x)&\coloneqq
      \frac{1}{h_k^2}\left(
        \frac{(v_k)_j(x)-h_k x_j}{h_k}
        -y^k_j(x_1)-h_k(Z^k x')_j
      \right)\quad\text{for $j=2,\dots,n$}
    \end{align*}
    converges weakly in $L^2(\Omega;\mathbb{R}^n)$ to a function $\beta$ belonging to the space
    \begin{align}\label{eq:mora_define_B}
      \mathcal{B}\coloneqq\bigg\lbrace\alpha\in L^2(\Omega;\mathbb{R}^n):
        &\int_{B_1(0)}\alpha(x)\,dx'=0,\,\partial_{x_i}\alpha\in L^2(\Omega;\mathbb{R}^n),  \nonumber  \\
        &\int_{B_1(0)}\alpha_j x_i-\alpha_i x_j\,dx'=0
      \bigg\rbrace.
    \end{align}
    Moreover, $\partial_{x_i}\beta^k\rightharpoonup\partial_{x_i}\beta$ in $L^2(\Omega)$ for $i=2,\dots,n$.
  \end{enumerate}
\end{theorem}
 In~\cite{Mora2003} Theorem~\ref{thm:muller-mora-3d} is proved for  the case $n=3$.
The proof is analogous for all $n\geq 2$.
Note that for $n=2$ we have $Z=0$ by antisymmetry.

\medskip

Note that the limit $(w,y,Z,\beta)$ is not unique, not even after choice of a subsequence. 
The point is that the choice of the constant rotations $\overline{Q}_k$ and the constants $c_k$ is not unique. 
Indeed, if we replace $\overline{Q}_k$ by $\check{Q}_k=e^{h_k W}\overline{Q}_k$ and $\check{c}_k=c_k+h_k\check{c}'+h_k^2 w_0 e_1$ 
with $W$ skew-symmetric, and $v_k$ by 
$\check{v}_k\coloneqq\check{Q}_k^T\overline{v}_k-\check{c}_k$, then the limits
in the convergences 1.\ to 6.\ get replaced by 
\begin{align*}
  \check{w}(x_1)&=w_0+w(x_1)+\frac{1}{2}(W^2)_{11}x_1+\sum_{j=2}^n W_{1j}y_j(x_1), \\
  \check{y}_i(x_1)&=y_i(x_1)+W_{1i}x_1+\check{c}'_i,\quad i=2,\dots,n, \\
  \check{z}_{ij}&=z_{ij}+W_{ij},\quad i,j=2,\dots,n \, .
\end{align*}

It is not difficult to show that modulo this equivalence relation, the limit is unique (for a suitable subsequence). 
We will show this in a more general setting in Proposition~\ref{pr:unique_limit} 
 below.
 
 \medskip
 
 Geometrically, the choice of $\overline Q_k$ and $c_k$ corresponds to the choice of a line segment 
 $\gamma_k$ given by
 $$ \gamma_k(t) = \overline Q_k  ( t e_1 + c_k)$$
 and a constant  orthonormal frame $\nu_j = \overline Q_k e_j$ for $j \ge 2$ perpendicular to $\gamma_k$
 such that $ \overline v_k$ stays close to $\gamma_k$ and $d_{h_k} v_k$ is close to $\overline Q_k$ on
 average.  
 
 In the Riemannian setting,  the straight line $\gamma_k$ is replaced by a  unit speed 
 geodesic and the constant
 frame is replaced by a parallel orthonormal frame. In this setting we want the expression of $\overline v_k$ in Fermi
 coordinates around this framed geodesic to satisfy
 Properties~\ref{it:mora-n_convergence_w} to~\ref{it:convergence_beta}  of Theorem~\ref{thm:muller-mora-3d} without application of an additional
 rotation $\overline Q_k$ and translation $c_k$. 
 Here the following simple consequence of Theorem~\ref{thm:muller-mora-3d} will be useful.

 \begin{corollary}  \label{eq:cor_mora-m} Assume that $h_k \to 0$ and that   the maps $\overline v_k: \Omega \to \mathbb R^n$
 satisfy  \begin{equation}\label{eq:lem-3d-uniform-energy-bound_corollary}
    \sup_{k}\frac{1}{h_k^4}\int_{\Omega}\dist^2(d\overline{v}_k,SO(n))\,dx<\infty.
  \end{equation} 
  Assume in addition that 
 \begin{eqnarray}
\int_{\Omega} \left|  \overline v_k (x) - \binom{x_1}{0}\right|^2  \, dx & \le& C h_k^2,   \quad \text{for $j \ge 2$,}
\label{eq:reduced_mora_m_vk}\\
\left| \int_{\Omega} (\overline v_k)_1(x) - x_1  \, dx \right|& \le& C h_k^2.
\label{eq:reduced_mora_m_wk}
 \end{eqnarray}
 Then there exists $R_k \in SO(n-1)$ such that a subsequence of  the maps
 \begin{equation}
 v^\sharp_k \coloneqq  \begin{pmatrix}  1 & 0 \\
0 & R_k   
\end{pmatrix}^T  \overline v_k
 \end{equation}
 satisfy  \eqref{eq:muller-mora-matrix-estimates}
 and Properties~\ref{it:mora-n_convergence_w} to~\ref{it:convergence_beta}  of Theorem~\ref{thm:muller-mora-3d}. 
 \end{corollary}
 Note that the action of  $\overline R_k \in SO(n-1)$ simply corresponds to a fixed  rotation  of the orthonormal
 frame around the curve. 
 
 \begin{proof}  Let $\overline Q_k$, $c_k$ and $v_k$ be as in    Theorem~\ref{thm:muller-mora-3d}.
 Set $d_k = \overline Q_k c_k$. Then
 $$\overline v_k = \overline Q_k v_k + d_k,   $$
 and 
 thus, for $j \ge 2$, 
 $$ \int_{B_1(0)} (\overline v_k)_j(x_1, x') dx' =  (d_k)_j  + \sum_{l=2}^n  h_k  (\overline Q_k)_{jl} (y_k)_l(x_1)+ 
 ( \overline Q_k)_{j1} x_1  + (\overline Q_k)_{j1} h_k^2 w_k(x_1).$$
 Since $y_k$ and $w_k$ are bounded in $L^2((-L,L)$ it follows from  
 \eqref{eq:reduced_mora_m_vk} that
 $$ \int_{(-L,L)} \left|  (d_k)_j + (\overline Q_k)_{j1} x_1 \right|^2 \, dx_1 \le C h_k^2.$$
 Hence 
 \begin{equation}  |(d_k)_j|  + |(\overline Q_k)_{j1}|  \le C h_k \quad \text{for $j \ge 2$.}
 \end{equation}
 Similarly, we have
 \begin{equation}    \label{eq:cor_mora_v1}  \int_{B_1(0)} (\overline v_k)_1(x_1, x')  - x_1 \,  dx' =  (d_k)_1  + \sum_{l=2}^n  h_k  (\overline Q_k)_{1l} (y_k)_l + 
  \left( (\overline Q_k)_{11} - 1\right) x_1  + (\overline Q_k)_{11} h_k^2 w_k
  \end{equation}
 and we get 
 \begin{equation} |(\overline Q_k)_{11} - 1|  \le C h_k   \, .
 \end{equation}
 Since $\overline Q_k \in SO(n)$ we have
 $ \sum_{l=1}^n (\overline Q_k)_{lj}  (\overline Q_k)_{l1} = 0$,  for $j\ge 2$, 
 and thus we get
 \begin{equation}   |(\overline Q_k)_{1j}| \le C h_k \quad \text{for $j \ge 2$.}
 \end{equation}
 Since $y_k$ and $w_k$ are bounded in $L^2$,  it follows from  \eqref{eq:cor_mora_v1}  
 and the assumption \eqref{eq:reduced_mora_m_wk}
 that 
 \begin{equation}
 |(d_k)_1| \le C h_k^2.
 \end{equation}

 Summarizing and recalling that $c_k = \overline Q_k^T d_k$, we get
 \begin{equation}  \label{eq:mora_cor_ck}
 |(c_k)_1| \le C h_k^2,  \quad |(c_k)_j| \le C h_k, \quad  |(\overline Q_k)^T e_1 - e_1| \le C h_k, \quad \text{for $j \ge 2$.}
 \end{equation}
It follows that there exist $Q^{\sharp}_k \in SO(n)$ such that
$$ Q^{\sharp}_k e_1 = (\overline Q_k)^T e_1, \quad |Q_k^\sharp - \Id| \le C h_k.$$
Then $\overline Q_k (Q_k^\sharp) e_1 = e_1$.
Since $\overline Q_k (Q_k^\sharp) \in SO(n)$,
 we have 
$$ \overline Q_k Q_k^\sharp = \begin{pmatrix}  1 & 0 \\
0 & R_k   
\end{pmatrix}$$
for some $R_k \in SO(n-1)$. 
Set 
$$v^\sharp_k \coloneqq  (\overline Q_k Q_k^\sharp)^T \overline v_k = (Q_k^\sharp)^T  (\overline Q_k)^T \overline v_k =
(Q_k^\sharp)^T (v_k + c_k). $$
Since $Q_k^\sharp \in SO(n)$,  the maps $v^\sharp_k$ satisfy~\eqref{eq:muller-mora-matrix-estimates}
(with with $Q_k$ and $\tilde Q_k$ replaced by
$ (Q_k^\sharp)^T  Q_k$ and $ (Q_k^\sharp)^T \tilde Q_k$, respectively).
Moreover, it is easy to see that Properties~\ref{it:mora-n_convergence_w} to~\ref{it:convergence_beta} 
for the sequence $v_k$ imply the corresponding properties for  a subsequence of  $v^\sharp_k$. 
\end{proof}

\section{Main theorem}\label{sec:main-thm}

From now on we always assume that $\mathcal{M}$ and $\tilde{\mathcal{M}}$ are  complete
smooth, oriented $n$-dimensional Riemannian manifolds  and that 
$\tilde M$ has a positive injectivity radius. 
We will often assume that $\tilde M$ is compact. Then the assumption on the injectivity radius is automatically satisfied.

The main result of this paper is the identification of the $\Gamma$-limit of the energy functional~\eqref{eq:def-energy}.
We define the spaces $\mathcal{X}$ and $\mathcal{X}^{\mathcal{G}}$ by
\begin{equation}\label{eq:def-X}
  \begin{split}
    \mathcal{X}
    &\coloneqq W^{1,2}((-L,L))
    \times W^{2,2}((-L,L);\mathbb{R}^{n-1})
    \times(W^{1,2}((-L,L);\mathbb{R}^{n-1\times n-1})\times\mathcal{B}, \\
    \mathcal{X}^{\mathcal{G}}
    &\coloneqq \mathcal{G}\times\mathcal{X}.
  \end{split} 
\end{equation}
Here the space $\mathcal B$ is defined in~\eqref{eq:mora_define_B} and
$\mathcal{G}$ is the set of all framed unit speed geodesics $(\tilde\gamma,\underline{\tilde\nu})$. 
Note that $\mathcal{G}$ is isomorphic to $\tilde{\mathcal{M}}\times SO(n)$,
since any framed unit speed geodesic is uniquely determined by the initial conditions
$(\tilde\gamma(0),\tilde\gamma'(0),\underline{\tilde\nu}(0))$.

\begin{definition}[Notion of convergence]\label{def:new-convergence}
  Let $h_k>0$ with $\lim_{k\rightarrow\infty}h_k=0$, let $\gamma:(-L,L)\rightarrow\mathcal{M}$ be a geodesic
  and let $\tilde{u}_k$ be a sequence of maps in $W^{1,2}(\Omega_{h_k}^*;\tilde{\mathcal{M}})$ with $\Omega_{h_k}^*=\psi_{\gamma,\underline\nu}(\Omega_{h_k})$. 
  Let $(\tilde\gamma,\underline{\tilde\nu},w,y,Z,\beta)\in\mathcal{X}^{\mathcal{G}}$.
  We say that $u_k$ converges to the tuple
  \begin{align*}
    (\tilde\gamma,\underline{\tilde\nu},w,y,Z,\beta),
  \end{align*}
  if   there exists a $k_0$ such that for $k \ge k_0$  the following assertions hold:
  \begin{enumerate}
    \item there exist Lipschitz maps $\tilde{v}_k:\Omega_{h_k}^*\rightarrow\tilde{\mathcal{M}}$ such that
      \begin{equation}\label{eq:def-conditions-lipschitz-approx}
        \begin{split}
          &\ell\coloneqq\sup_k\Lip\tilde{v}_k<\infty, \\
          \sup_k\frac{1}{h_k^4}\frac{1}{\mu(\Omega_{h_k}^*)}&\mu(\{x\in\Omega_{h_k}^*:\tilde{v}_k(x)\neq \tilde u_k(x)\})<\infty,
        \end{split}
      \end{equation}
      where $\mu$ is the volume measure on $\tilde{\mathcal{M}}$,
    \item there exist geodesics $\tilde\gamma_k:(-L,L)\rightarrow\tilde{\mathcal{M}}$ such that
      \begin{equation}\label{eq:def-convergence-estimate-dist-from-geod}
        \sup_k\sup_{x\in\Omega_{h_k}}\frac{1}{h_k}\dist(\tilde\gamma_k(x_1),\tilde{v}_k(x))<\infty
      \end{equation}
      and $\tilde{\gamma}_k\rightarrow\tilde\gamma$ uniformly as $k\rightarrow\infty$,
    \item there exist parallel orthonormal frames $\underline{\tilde\nu}_k$ of the normal bundle of $\tilde\gamma_k$
      that converge to $\underline{\tilde\nu}$ with the following property.
      Let $\overline{v}_k$ denote the lift of $\tilde{v}_k$ in the sense of Lemma~\ref{lem:existence-lift_new} 
      and let $v_k=\overline{v}_k\circ I_{h_k}$ with $I_{h_k}(x_1,x')=(x_1,h_k x')$. 
      Then the maps $v_k$ 
      satisfy  \eqref{eq:muller-mora-matrix-estimates} 
      and Properties~\ref{it:mora-n_convergence_w} to~\ref{it:convergence_beta} of Theorem~\ref{thm:muller-mora-3d}
      with limits $(w,y,Z,\beta)$. 
  \end{enumerate}
  We denote this convergence by $u_k\rightarrow(\tilde\gamma,\underline{\tilde\nu},w,y,Z,\beta)$.
\end{definition}

In view of the discussion of  non-uniqueness after Theorem~\ref{thm:muller-mora-3d}
we define an equivalence relation on the space $\mathcal{X}^{\mathcal{G}}$ as follows.

\begin{definition}\label{def:equiv-relation}
  We say that $(\hat\gamma,\underline{\hat\nu},\hat w,\hat Z,\hat y,\hat\beta)  
  \sim(\check\gamma,\underline{\check\nu},\check w,\check Z,\check y,\check\beta)$ if
  \begin{equation} ( \hat \gamma, \hat{\underline \nu}) = (\check \gamma, \check{\underline \nu})
  \end{equation}
and 
  if there exist $J:(-L,L)\rightarrow\mathbb{R}^n$ and $B:(-L,L)\rightarrow\mathbb{R}^{n\times n}_{\mathrm{skew}}$ such that 
  \begin{align}
    J_1&=0, \\ 
    Be_1&=\partial_{x_1}J, \\
    \partial_{x_1}B(x_1)
    &=-\overline{\tilde{\mathcal R}}(J(x_1),e_1),\label{eq:equiv_rel_cond_dx1_B}
  \end{align}
  and, for $j,l\geq 2$,
  \begin{eqnarray}
\hat Z_{jl} &=& \check Z_{jl} + B_{jl}     \label{eq:equiv_def_trafo_Z}  \\
\hat y &=& \check y + J'  \label{eq:equiv_def_trafo_y} \\
\partial_{x_1} \hat w(x_1) +\tilde{ \mathcal  T}_{11}(x_1, \hat y(x_1))  
       &=&  \partial_{x_1} \check w(x_1)  + \tilde{ \mathcal T}_{11}(x_1, \check y(x_1)) \nonumber   \\
 & & + \frac{1}{2} (\hat A^2 - \check A^2)_{11}(x_1)    \label{eq:equiv_def_trafo_w} \\
 \partial_{x_j} \hat \beta_1(x_1, x') 
+  2 \tilde{\mathcal T}_{1j}(x_1, \hat y(x_1) + x' )   &= & \partial_{x_j} \check \beta_1(x) + 2 \tilde{\mathcal T}_{1j}(x_1, \check y(x_1) + x') 
\nonumber \\
                                                       & & e_j \cdot \partial_{x_1}( \check Z  - \hat Z)(x_1)  x'   + (\hat A^2 -  \check A^2)_{1j}(x_1)
 \label{eq:equiv_def_trafo_beta1} \\
(\sym \nabla' \hat \beta')_{jl}(x_1,x') +\tilde{ \mathcal T}_{jl}(x_1, \hat y(x_1) + x' ) 
 &=&(\sym  \nabla'  \hat \beta')_{jl}(x_1, x')  +  \tilde{ \mathcal T}_{jl}(x_1, \check y(x_1) + x') \nonumber  \\
& &  +  \frac{1}{2} (\hat A^2 -  \check A^2)_{jl}(x_1),
  \label{eq:equiv_def_trafo_beta_prime} 
  \end{eqnarray}
  where 
  $\nabla'=(\partial_{x_2},\dots,\partial_{x_n}),\,\hat\beta'=(\hat\beta_2,\dots,\hat\beta_n),\,J'=(J_2,\dots,J_n)$ and 
  \begin{align*}
    \hat A=\begin{pmatrix}
    0 & -\partial_{x_1}(\hat y)^T \\ 
    \partial_{x_1}\hat y & \hat Z
  \end{pmatrix},\qquad 
      \check A=\begin{pmatrix}
      0 & -\partial_{x_1}(\check y)^T \\ 
      \partial_{x_1}\check y & \check Z
  \end{pmatrix}.
  \end{align*}
  Here we used the abbreviation  $\tilde {\mathcal T} := \mathcal T^{\tilde{\mathcal R}}_{\tilde\gamma,\underline{\tilde\nu}}$. Moreover $\overline{\tilde{\mathcal R}} := \psi_{\tilde \gamma, \underline{ \tilde \nu}}^*
  \tilde{ \mathcal R}$ where $ \tilde{ \mathcal R}$ is the curvature tensor of $\tilde M$. Equivalently,
  $\overline{\tilde{\mathcal R}}$ is the curvature tensor of the pullback metric 
  $\psi_{\tilde \gamma, \underline{ \tilde \nu}}^* \tilde g$.
  \end{definition}
  
  If the underlying framed geodesic is clear from the context we sometimes also write
  $(\hat w,\hat Z,\hat y,\hat\beta)  
  \sim(\check w,\check Z,\check y,\check\beta)$ to express the equivalence relation.

 Note that for $\hat \beta \in \mathcal B$, the map $\hat \beta: \Omega \to \mathbb R^n$  is uniquely determined
  by $\nabla'  \hat  \beta_1$ and $\sym \nabla' \hat{\beta'}$.  Indeed,  $\hat \beta_1$  is uniquely determined up to 
  the addition of a function of $x_1$ and the condition $\int_{B_1(0)}\hat  \beta_1(x) \, dx' =0$  in the definition
  of $\mathcal B$ determines
  $\beta_1$ uniquely.  Similarly,  by Korn's  inequality, $\hat{\beta'}$ is determined up to the addition of  a map $c(x_1) +  M(x_1) x'$  where $M: \mathbb R^{n-1} \to \mathbb R^{n-1}$  is linear and skew-symmetric. The conditions
  $\int_{B_1(0)}  \hat \beta_j(x) \, dx' =0$  for $j \ge 2$ and $\int_{B_1(0)}  x_i  \hat \beta_j(x) - x_j  \hat \beta_i(x) \, dx' =0$   for $i,j \ge 2$ thus determine $\hat{\beta'}$ uniquely.

\begin{remark}
  It follows from~\eqref{eq:equiv_rel_cond_dx1_B}, the map $J$ is a Jacobi field (with respect to the pullback metric
  $\psi_{\tilde \gamma, \underline{ \tilde \nu}}^* \tilde g$),  i.e., 
  $J$ corresponds to an infinitesimal (normal) variation of the geodesic $\hat\gamma$. 
  Similarly,  $B$ corresponds to an infinitesimal variation of the orthonormal frame $(\hat\gamma',\hat\nu_k,\dots,\hat\nu_n)$ of $\hat\gamma$.
\end{remark}

A key observation is that the  limit  $(\tilde\gamma,\underline{\tilde\nu},w,y,Z,\beta)$
defined in \eqref{def:new-convergence} is unique up to the equivalence relation $\sim$. 
More precisely, the following result holds.

\begin{proposition} \label{pr:unique_limit}
Suppose that there exist Lipschitz maps $\hat{\tilde v}$ and $\check{\tilde v}$
from $\Omega^*_{h_k} = \psi(\Omega_{h_k})$ to $\tilde{\mathcal{M}}$ with the following properties:
\begin{equation}  \label{eq:unique_lip_restated} \sup_{k} \max(\Lip \hat{\tilde v}_k, \Lip \check{\tilde v}_k) \le l ,
\end{equation}
\begin{equation}  \label{eq:unique_small_set_restated}
 \lim_{k \to \infty}\frac{1}{h_k^4} \frac{\mu  \{x \in \Omega_{h_k}^* : 
\hat{\tilde v}_k(x) \ne \check{\tilde v}_k(x) \} }{\mu (\Omega_{h_k})}= 0,
\end{equation}
where $\mu$ is the volume measure on $\mathcal M$.

Assume further  that there  exist  framed  unit speed geodesics $(\hat \gamma_k, \underline{\hat \nu})$ and
$(\check \gamma_k, \underline{\check \nu})$ such that 
\begin{equation}
\sup_k   \frac1{h_k} \sup_{x \in \Omega_{h_k}}   \dist(\hat{ \tilde v}_k \circ \psi(x), \hat \gamma_k(x_1))
+  \frac1{h_k}  \sup_{x \in \Omega_{h_k}}   \dist(\check{ \tilde v}_k \circ \psi(x), \check \gamma_k(x_1)) < \infty
\end{equation}
 and lifts  $\hat v_k: \Omega \to \mathbb R^n$ and $\check  v_k:  \Omega \to \mathbb R^n$ such that
$$ \hat \psi_k \circ \hat v_k = \hat{\tilde v}_k \circ \psi \circ I_{h_k}, \qquad
\check \psi_k \circ  \check v_k = \check{\tilde v}_k \circ \psi \circ I_{h_k}$$
where $\hat \psi_k$ and $\check \psi_k$ are the Fermi coordinate maps for the framed geodesics 
 $(\hat \gamma_k, \underline{\hat \nu})$ and
$(\check \gamma_k, \underline{\check \nu})$, respectively,  and where $I_{h_k}(x_1, x') = (x_1, h_k x')$. 
Assume that $\hat v_k$ and $\check  v_k$ satisfy Properties~\ref{it:mora-n_convergence_w}
to~\ref{it:convergence_beta}   of Theorem~\ref{thm:muller-mora-3d} and let 
$(\hat \gamma, \underline{\hat \nu}, \hat y, \hat w, \hat Z, \hat \beta)$ and 
$(\check \gamma, \underline{\check \nu}, \check y, \check w, \check Z, \check  \beta)$ 
denote the corresponding limits defined in  Definition~\ref{def:new-convergence}.
Then 
\begin{align*}
  (\hat \gamma, \underline{\hat \nu} )= (\check \gamma, \underline{\check \nu})
\quad \text{and} \quad 
 (\hat w,\hat y, \hat Z, \hat \beta) \sim  (\check w, \check y, \check Z, \check  \beta).
\end{align*}
\end{proposition}

\begin{proof} The proof uses standard calculations, but is a bit lengthy. Thus we  postpone it to 
Section~\ref{se:unique_limit} below.
\end{proof}

\begin{definition}\label{def:energy}
Let $(\gamma,\underline\nu)$ and $(\tilde\gamma,\underline{\tilde\nu})$ be framed, unit speed geodesics
in $\mathcal{M}$ and $\tilde{\mathcal{M}}$, respectively.
For $(\tilde\gamma,\underline{\tilde\nu},w,y,Z,\beta)\in\mathcal{X}^{\mathcal{G}}$ define 
\begin{equation}\label{eq:def-I}
  \mathcal{I}^{\tilde\gamma,\underline{\tilde\nu}}(w,y,Z,\beta)
  \coloneqq\fint_{-L}^L\fint_{B_1(0)}|\sym G(x)-\mathcal{T}(x)|^2\,dx,
\end{equation}
where 
\begin{equation}\label{eq:def-G-T}
  \begin{split}
    G(x_1,x')
    &\coloneqq\left(\begin{array}{c|c|c|c}
      \left(
        \partial_{x_1}w 
      \right)e_1+(\partial_{x_1}A)\begin{pmatrix}
        0 \\ x'
      \end{pmatrix}
      & \partial_{x_2}\beta & \dots & \partial_{x_n}\beta
    \end{array}\right)
    -\frac{1}{2}A^2  \, ,  \\
    \mathcal{T}(x_1,x')&\coloneqq
    \left(\mathcal{T}^{\mathcal{R}}_{\gamma,\underline\nu}(x_1,x')
    -\mathcal{T}^{\tilde{\mathcal{R}}}_{\tilde\gamma,\underline{\tilde\nu}}(x_1,y(x_1)+x')\right)  \, , 
  \end{split}
\end{equation}
and
\begin{align}  \label{eq:def-I_def-A}
  A&=\begin{pmatrix}
    0 & -\partial_{x_1}y^T \\ 
    \partial_{x_1}y & Z
  \end{pmatrix}.
\end{align} 
\end{definition}

\begin{proposition}\label{prop:invariance_energy}
  Suppose that $(\hat\gamma,\underline{\hat\nu},\hat w,\hat y,\hat Z,\hat\beta)
  \sim(\check\gamma,\underline{\check\nu},\check w,\check y,\check Z,\check\beta)$
  Then
  \begin{align*}
    \mathcal{I}^{\hat\gamma,\underline{\hat\nu}}(\hat w,\hat y,\hat Z,\hat\beta)
    &=\mathcal{I}^{\check\gamma,\underline{\check\nu}}(\check w,\check y,\check Z,\check\beta).
  \end{align*}
\end{proposition}

\begin{proof} By definition of the the equivalence relation,  we have $(\hat \gamma, \hat{\underline \nu})
=(\check \gamma, \check{\underline \nu})$ and in the following we denote this framed geodesic by
$(\tilde\gamma, \tilde{\underline \nu})$.

Let 
$$ \hat E :=  \sym \hat G -   \mathcal{T}^{\mathcal{R}}_{\gamma,\underline\nu}(x_1,x')
    +\mathcal{T}^{\tilde{\mathcal{R}}}_{\tilde\gamma,\underline{\tilde\nu}}(x_1,\hat y(x_1)+x'),
 $$
$$  \hat E :=  \sym \check G -   \mathcal{T}^{\mathcal{R}}_{\gamma,\underline\nu}(x_1,x')
    +\mathcal{T}^{\tilde{\mathcal{R}}}_{\tilde\gamma,\underline{\tilde\nu}}(x_1,\check y(x_1)+x')  \, , $$
    where $\hat G$  is computed using $\hat w$, $\hat A$ and $\hat \beta$ and similarly for $\check G$.
Set $E := \hat E - \check E$.
It  follows directly from    \eqref{eq:equiv_def_trafo_beta_prime} that
$$ E_{jl} = 0 \quad\text{if $j,l \ge 2$.}$$
For  $j \ge 2$
we get
\begin{eqnarray*}
2 E_{1j}  &=& 
\sum_{k=2}^n  \partial_{x_1}  \hat A_{jk} \, x_k + \partial_{x_j} \hat \beta_1 +  2 \tilde{ \mathcal{T}}_{1j}
(x_1, x' + \hat y(x_1)) - (\hat A)^2_{1j}    \\ 
 &- & 
 \left( \sum_{k=2}^n   \partial_{x_1}  \check A_{jk}  \,  x_k + \partial_{x_j} \check \beta_1 +  2 \tilde{ \mathcal{T}}_{1j}
(x_1, x' + \check y(x_1)) - (\check A)^2_{1j}   \right) \\
&\underset{ \eqref{eq:equiv_def_trafo_beta1} }{=} &   \sum_{k=2}^n  \partial_{x_1}  \hat A_{jk} \, x_k -  \sum_{k=2}^n   \partial_{x_1}  \check A_{jk} \, x_k   
- \sum_{k=2}^n \partial_{x_1}  (\check Z - \hat Z)_{jk}  \, x_k \\
&= & 0,
\end{eqnarray*}
since  $\hat A_{jk}=\hat Z_{jk}$  for $j, k \ge 2$  and similarly for $\check A $  and $\check Z$.

Finally we have
\begin{eqnarray*} 
E_{11} &=&  
\partial_{x_1} \hat w +
 \sum_{l=2}^n  \partial_{x_1} \hat A_{1l}  \, x_l - \frac12 (\hat A)^2 _{11}  +
\tilde {\mathcal T}_{11}(x_1, x' + \hat y(x_1))\\
&-&  \left(  \partial_{x_1} \check w + \sum_{l=2}^n  \partial_{x_1} \check A_{1l} \,  x_l - \frac12 (\check A)^2 _{11} 
+ \tilde {\mathcal T}_{11}(x_1, x' + \check y(x_1))  \right)\\
&= &       \partial_{x_1} \hat w -
 \sum_{l=2}^n  \partial_{x_1}^2 \hat y_l \, x_l
- \frac12 (\hat A)^2 _{11}  +
\tilde {\mathcal T}_{11}(x_1, x' + \hat y(x_1))\\
&-&  \left(  \partial_{x_1} \check w - \sum_{l=2}^n  \partial_{x_1}^2 \check y_l \, x_l - \frac12 (\check A)^2 _{11}  +
\tilde {\mathcal T}_{11}(x_1, x' + \check y(x_1))  \right) \\
&\underset{ \eqref{eq:equiv_def_trafo_y},  \eqref{eq:equiv_def_trafo_w}}{=}& 
- \tilde{\mathcal T}_{11}(x_1, \hat y(x_1)) -  \sum_{l=2}^n \partial_{x_1}^2 J_l(x_1)  \, x_l  
+ \tilde {\mathcal T}_{11}(x_1, x' + \hat y(x_1))  \\
&-&  \left(  - \tilde{\mathcal T}_{11}(x_1, \check y(x_1)) 
+ \tilde {\mathcal T}_{11}(x_1, x' + \check y(x_1))   \right)   \, .
\end{eqnarray*}
Since $v \mapsto \tilde{\mathcal T}(x_1, v)$  is quadratic, the derivative $\tilde{\mathcal T}$ with respect to
$v$ is bilinear.  Moreover $\hat y - \check y =J'$ and $J_1 =0$ so we get
\begin{eqnarray}
E_{11} =  D\tilde{\mathcal T}_{11}(x_1)(x', J(x_1)) -    \sum_{l=2}^n \partial_{x_1}^2 J_l(x_1)  x_l.
\end{eqnarray}
It follows from the definition   \eqref{eq:def-T}
 of  $\tilde{\mathcal T}$ and the fact that $J_1 =0$ that 
 $$ D\tilde{\mathcal T}_{11} (x_1)(x', J(x_1)) = -  \sum_{k,l=2}^n \overline{ \tilde{\mathcal R}}(x_1)_{1k1l} J_k(x_1) x_l.
$$
 On the other hand  $J$  is  a Jacobi field and  \eqref{eq:equiv_rel_cond_dx1_B} gives
 $$ \partial_{x_1}^2  J  = \partial_{x_1}  B e_1  =  -  \overline{ \tilde{\mathcal R}}(x_1)(J(x_1), e_1) e_1.$$
 Thus, for $l \ge 2$
 $$ \partial_{x_1}^2  J_l =  -  \overline{ \tilde{\mathcal R}}(x_1)(e_l, e_1, J, e_1) =
 -  \sum_{k,l=2}^n \overline{ \tilde{\mathcal R}}_{l1k1}(x_1) J_k(x_1) 
 = - \sum_{k,l=2}^n \overline{ \tilde{\mathcal R}}_{1k1l}(x_1) J_k(x_1)
 $$
 and  thus
 $$  \sum_{l=2}^n \partial_{x_1}^2  J_l   \, x_l = - \sum_{k,l=2}^n \overline{ \tilde{\mathcal R}}_{1k1l}(x_1) J_k(x_1) x_l.$$
 Hence  $E_{11} =0$. 
 \end{proof}

In the following we usually understand the convergence in Definition~\ref{def:new-convergence} as a convergence in $\mathcal{X}^{\mathcal{G}}/\sim$. 

\begin{theorem}\label{thm:main-thm}
  Let $(\mathcal{M},g)$ and $(\tilde{\mathcal{M}},\tilde{g})$ be 
  complete oriented, $n$-dimensional Riemannian manifolds with $n\geq 2$.
  Then the following statements hold:
  \begin{enumerate}[(i)]
    \item\label{item:main-thm-1} Compactness:
      Assume in addition that $\tilde{\mathcal{M}}$ is compact.
      Let $h_k\rightarrow 0,\, \tilde u_k\in W^{1,2}(\Omega_{h_k}^*,\tilde{\mathcal{M}})$ 
      and suppose that $\limsup_{k\rightarrow\infty}\frac{1}{h_k^4}E_{h_k}(\tilde u_k)<\infty$.
      Then there exists a subsequence $h_{k_j}\rightarrow 0$ such that 
      \begin{align*}
        u_{k_j}\rightarrow(\tilde{\gamma},\underline{\tilde\nu},w,y,Z,\beta)
      \end{align*}
      in the sense of Definition~\ref{def:new-convergence}.
    \item\label{item:main-thm-2} $\Gamma$-liminf inequality:
      If $h_k\rightarrow 0$ and $\tilde u_k\rightarrow(\tilde\gamma,\underline{\tilde\nu},w,y,Z,\beta)$ 
      in the sense of Definition~\ref{def:new-convergence}, then
      \begin{align*}
        \liminf_{k\rightarrow\infty}\frac{1}{h_k^4}E_{h_k}(\tilde u_k)\geq\mathcal{I}^{\tilde\gamma,\underline{\tilde\nu}}(w,y,Z,\beta).
      \end{align*}
    \item\label{item:main-thm-3}
      Recovery sequence: Given a tuple $(\tilde\gamma,\underline{\tilde\nu},w,y,Z,\beta)$ and a sequence $h_k\rightarrow 0$, there exists a sequence 
      $\tilde u_k$ such that $\tilde u_k\rightarrow(\tilde\gamma,\underline{\tilde\nu},w,y,Z,\beta)$ in the sense of Definition~\ref{def:new-convergence} and 
      \begin{equation}\label{eq:thm-recovery-seq-convergence}
        \limsup_{k\rightarrow\infty}\frac{1}{h_k^4}E_{h_k}(\tilde u_k)
        \leq\mathcal{I}^{\tilde\gamma,\underline{\tilde\nu}}(w,y,Z,\beta).
      \end{equation}
  \end{enumerate}
\end{theorem}

Combining properties~\eqref{item:main-thm-2} and~\eqref{item:main-thm-3} 
with Proposition~\ref{prop:limsup-ineq} and Lemma~\ref{lem:muller-mora-lem-2.3},
we obtain that the functionals $\frac{1}{h_k^4}E_{h_k}$ are $\Gamma$-convergent to the functional $\mathcal{I}^{\tilde\gamma,\underline{\tilde\nu}}$
with respect to the convergence in Definition~\ref{def:new-convergence} modulo the equivalence relation of Definition~\ref{def:equiv-relation}.
Regarding the compactness statement,
we give an explicit construction of the paths $\tilde\gamma_k$ in step 2 of the proof of Theorem~\ref{thm:main-thm}.
First, we prove the $\Gamma$-liminf inequality. 

\begin{lemma}\label{lem:muller-mora-lem-2.3}
  Assume that $\tilde u_k\rightarrow(\tilde\gamma,\underline{\tilde\nu},w,y,Z,\beta)$ in the sense of Definition~\ref{def:new-convergence}. 
  Let $\tilde{v}_k,\overline{v}_k$ and $v_k$ be as in Definition~\ref{def:new-convergence},
  let $Q_k:(-L,L)\rightarrow SO(n)$ be such that~\eqref{eq:muller-mora-matrix-estimates} holds
  and define $G_k:\Omega\rightarrow\mathbb{R}^{n\times n}$ by
  \begin{align*}
    G_k\coloneqq\frac{Q_k^T \, (\overline{\tilde g}_k^{1/2}\circ v_k) \, \,d_{h_k} v_k \,  \, (\overline{g}^{-1/2}
    \circ I_{h_k})  -\Id}{h_k^2}.
  \end{align*}
 where $I_{h_k}(x_1, x') = (x_1, h_k x')$.
  Then there exists a subsequence (not relabeled) and $G\in L^2(\Omega;\mathbb{R}^{n\times n})$ such that 
  $G_k \rightharpoonup G$ weakly in $L^2$ and 
  \begin{equation}\label{eq:lem-sym-G}
    \begin{split}
      \sym G(x)
      &=\partial_{x_1}w(x_1) e_1\otimes e_1
      -\frac{1}{2}A^2(x_1)
      +\sym\left(\begin{array}{c|c|c|c}
        \partial_{x_1}A(x_1)\begin{pmatrix}
        0 \\ x'
      \end{pmatrix}
      & \partial_{x_2}\beta & \dots & \partial_{x_n}\beta
    \end{array}\right) \\
      &\quad -\mathcal{T}^{\mathcal{R}}_{\gamma,\underline\nu}(x_1,x')
      +\mathcal{T}^{\tilde{\mathcal{R}}}_{\tilde\gamma,\underline{\tilde\nu}}(x_1,y(x_1)+x').
    \end{split}
  \end{equation}
  Moreover,
  \begin{equation}\label{eq:liminf-ineq}
    \liminf_{k\rightarrow\infty}\frac{1}{h_k^4}\fint_{\Omega_{h_k}^*}
    \dist{}^2(d\tilde u_k,SO(g,\tilde{g}))\,d\Vol_g
    \geq\fint_\Omega|\sym G|^2\,dx.
  \end{equation}
\end{lemma}

\begin{proof}  
The reasoning is similar to  the proof of  Lemma 2.3 in~\cite{Mora2003},  but we need into account the extra
contributions coming  from the deviation of the pullback metric in Fermi coordinates from the Euclidean metric.

  We have
    \begin{equation}\label{eq:rewriting_Gk}
      G_k=
      \frac{Q_k^T  \, (\overline{\tilde g}_k^{1/2}  \circ v_k - \Id)  \,  d_{h_k}v_k \,  (\overline{g}^{-1/2}\circ I_{h_k})}{h_k^2} 
      + \frac{Q_k^T d_{h_k}v_k-\Id}{h_k^2}
      + \frac{ Q_k^T  \,  d_{h_k} v_k  \, (\overline{g}^{-1/2} \circ I_{h_k} - \Id ) }{h_k^2}.
    \end{equation}
  The maps $\tilde v_k$ and hence $\overline v_k$ are uniformly Lipschitz and therefore $d_{h_k}  v_k$ 
  is bounded in $L^\infty$. 
   Moreover, the condition \eqref{eq:def-convergence-estimate-dist-from-geod} implies that
   \begin{equation} \label{eq:lsc_bound_vk}  \sup_{x \in \Omega}  |v_k(x) - x_1|  \le  C h_k.
   \end{equation}

  It follows from the Taylor expansion of the metric in Fermi coordinates
  that 
  \begin{align*}
    \frac{\overline{g}^{-1/2} \circ I_{h_k} - \Id  }{h_k^2} \to - \mathcal T^{\mathcal R}_{\gamma,\underline\nu}
    \quad \text{uniformly in $\Omega$,}
  \end{align*}
  see  \eqref{eq:sqrt-metric-inverse-taylor}.
 By~\eqref{eq:muller-mora-matrix-estimates} we have $Q_k^T  \, d_{h_k} v_k \to \Id$ in $L^2(\Omega)$ and thus
 \begin{equation} \frac{ Q_k^T   \,d_{h_k} {v}_k  \,(\overline{g}^{-1/2} \circ I_{h_k} - \Id ) }{h_k^2} \to - \mathcal T^{\mathcal R}_{\gamma,\underline\nu}
 \quad \text{in $L^2(\Omega)$.}
 \end{equation}
 
 To treat the term involving $\overline{\tilde g}$,  we define
$ \imath(y)(x)\coloneqq (x_1, y(x_1) + x')$, where $y : (-L,L) \to \mathbb R^{n-1}$  is the limit of $\frac{1}{h_k} \int_{B_1(0)}
v'_k$. It follows from  Property~\eqref{it:convergence_beta} in Theorem~\ref{thm:muller-mora-3d} that
$h_k^{-1} v'_k \to \imath(y)$  in $L^2(\Omega)$. 
In view of~\eqref{eq:lsc_bound_vk} we also get convergence in $L^p(\Omega)$, for all $p < \infty$.
Using again the bound~\eqref{eq:lsc_bound_vk} and the Taylor expansion \eqref{eq:sqrt-metric-taylor} of the 
Fermi coordinates we get
\begin{align*}
  \frac{ (\overline{\tilde g}_k^{1/2}  \circ {v}_k - \Id) }{h_k^2}   
\to \mathcal T^{\tilde{\mathcal R}}_{\tilde\gamma,\underline{\tilde\nu}} \circ \imath(y) 
 \quad \text{in $L^q(\Omega)$, for all $q < \infty$.}
\end{align*}
 It follows that 
\begin{equation}
    \frac{Q_k^T  \, (\overline{\tilde g}_k^{1/2}  \circ v_k - \Id)  \,  d_{h_k}v_k \,  \overline{g}^{-1/2}\circ I_{h_k}}{h_k^2}
    \to  \mathcal T^{\tilde{\mathcal R}}_{\tilde\gamma,\underline{\tilde\nu}} \circ \imath(y)   \quad \text{in $L^2(\Omega)$.}
    \end{equation}
  It remains to analyze the term $h_k^{-2} (Q_k^T d_{h_k} v_k - \Id)$.  
  By~\eqref{eq:muller-mora-matrix-estimates}, this term is bounded in $L^2(\Omega)$ and hence there exists a  weakly convergent 
  subsequence, not relabeled, such that
  \begin{equation}  \label{eq:lsc_weakto_H}
   \frac{Q_k^T d_{h_k}v_k-\Id}{h_k^2} \rightharpoonup  H \quad \text{in $L^2(\Omega)$.}
  \end{equation}
  Hence
  \begin{equation}
    G_k \rightharpoonup  H  +  \mathcal T^{\tilde{\mathcal R}}_{\tilde\gamma,\underline{\tilde\nu}} \circ \imath(y) 
    - \mathcal T^{\mathcal R}_{\gamma,\underline\nu}.
  \end{equation}
  The last two terms on the right and side take values in symmetric matrices. 
  To compute $\sym H$ we note that by Property~\ref{item:muller-mora-conv-A} and Property~\ref{item:muller-mora-prop-5} in Theorem~\ref{thm:muller-mora-3d}
  we have 
   \begin{align*}
    \frac{d_{h_k} v_k - \Id}{h_k} \to A \quad \text{in $L^2(\Omega)$ and} \quad 
    \sym\frac{Q_k^T - \Id}{h_k^2} \to \frac{A^2}{2} \quad \text{uniformly.} 
   \end{align*}
      From Property~\ref{it:convergence_beta} in Theorem~\ref{thm:muller-mora-3d} we get
   \begin{align*}
    \sym \frac{d_{h_k} v_k - \Id}{h_k^2} \to   \partial_{x_1}w(x_1) e_1\otimes e_1
+
   \sym\left(\begin{array}{c|c|c|c}
        \partial_{x_1}A(x_1)\begin{pmatrix}
        0 \\ x'
      \end{pmatrix}
      & \partial_{x_2}\beta & \dots & \partial_{x_n}\beta
    \end{array}\right) 
   \end{align*}
   strongly in $W^{-1,2}(\Omega)$. 
   Since 
   \begin{align*}
    Q_k^T d_{h_k} v_k - \Id  = (Q_k^T - \Id) + (d_{h_k} v_k - \Id) + (Q_k^T - \Id) (d_{h_k} v_k - \Id)
   \end{align*}
   and since $A$ is skew-symmetric and thus  $A^2$ is symmetric we get
   \begin{align*}
 \sym     H = \partial_{x_1}w(x_1) e_1\otimes e_1
 +  \sym\left(\begin{array}{c|c|c|c}
        \partial_{x_1}A(x_1)\begin{pmatrix}
        0 \\ x'
      \end{pmatrix}
      & \partial_{x_2}\beta & \dots & \partial_{x_n}\beta
    \end{array}\right)  - \frac{A^2}{2}.
   \end{align*}

 \medskip

  Finally we show~\eqref{eq:liminf-ineq}. We first choose  a subsequence $k_j$ such that
  \begin{align*} 
    & \,   \lim_{j\rightarrow\infty}\frac{1}{h_{k_j}^4}\fint_{\Omega^*_{h_{k_j}}}
    \dist^2(d\tilde u_{k_j},SO(g, \tilde g))\,d\Vol_g   \\
    = & \,  \liminf_{k\rightarrow\infty}\frac{1}{h_k^4}\fint_{\Omega^*_{h_k}}
    \dist^2(d\tilde u_k,SO(g, \tilde g))\,d\Vol_g.
  \end{align*}
  Then we apply the previous reasoning to the subsequence $k_j$. Thus 
  there exists a further subsequence such that $G_{k_{j_l}} \rightharpoonup G$
  as $l \to \infty$ where $\sym G$ is given by~\eqref{eq:lem-sym-G}. We need to show 
  \begin{align*}
    \lim_{l\rightarrow\infty}\frac{1}{    h_{k_{j_l}}^4}
    \fint_{\Omega^*_{h_{k_{j_l}}}}
    \dist^2(d\tilde u_{k_{j_l}},SO(g, \tilde g))\,d\Vol_g   \ge \fint_{\Omega} |\sym G|^2.
  \end{align*}
  This follows by a Taylor expansion away from the set where $h_k^2 G_k$ is not small
  (and the set where $\tilde u_k \ne \tilde v_k$).
  as in \cite{Friesecke2002, Mora2003}.
  For the convenience of the reader we include the details for the argument.

  To simply the notation we now write again $G_k$ instead of $G_{k_{j_l}}$. 
  Let 
  \begin{align*}
    F'_k = \{ x \in \Omega^*_h:   \tilde v_k(x) \ne \tilde u_k(x)\},  \quad  F''_k\coloneqq \{
    x \in \Omega : (\psi_{\gamma, \underline \nu} \circ I_{h_k})(x) \in F'_k \}.
  \end{align*}
  Then, for $x \notin F''_k$ we have, by~\eqref{eq:rewrite-energy-SOn}, 
  \begin{align}
    & \, \dist(d\tilde u_k, SO(g, \tilde g))\circ  \psi_{\gamma, \underline \nu} \circ I_{h_k} \nonumber  \\
    = & \, 
    \dist(d\tilde v_k, SO(g, \tilde g))\circ  \psi_{\gamma, \underline \nu} \circ I_{h_k}   \nonumber \\
    = & \,  \dist \big( \, (\overline{\tilde{g}}_k\circ v_k)^{1/2}) \,  d_{h_k} v_k \, (\overline{g}^{-1/2} \circ I_{h_k}), SO(n)\big)
    \nonumber \\
    = & \, \dist( \Id + h_k^2 G_k, SO(n)).  \label{eq:dist_du_k_and G_k}  
  \end{align}
  where $\overline g$ and $\overline{\tilde g}_k$ are the metrics of $\mathcal M$ and $\tilde{\mathcal M}$
  in the Fermi coordinates induced by $\psi_{\gamma, \underline \nu}$
  and $\psi_{\hat \gamma_k, \underline{\hat \nu}_k}$, respectively.
  Set
  \begin{align*}
    F_k = F''_k   \cup \{ |h_k^2 G_k|>h_k\}.
  \end{align*}
  Then~\eqref{eq:def-conditions-lipschitz-approx}
  and the boundedness of $G_k$ in $L^2$ imply that
  \begin{align*}
    \mathcal L^n(F_k) \le C h_k^4  + C h_k^2  \to 0 \quad \text{as $k \to \infty$}.
  \end{align*}
  From~\eqref{eq:dist_du_k_and G_k} and the change of variables formula we get
    \begin{align*}
    & \,    \liminf_{k\rightarrow\infty}\frac{1}{h_k^4}\fint_{\Omega^*_{h_k}}
    \dist^2(d\tilde u_k,SO(g, \tilde g))\,d\Vol_g   \nonumber   \\
    \ge & \, 
      \frac{1}{h_k^4}  \int_{\Omega \setminus F_k} \dist( \Id + h_k^2 G_k, SO(n)) \,
      \sqrt{  \det \overline{g} }\circ I_{h_k} \, dx    \, \,      \frac{1}{  \int_{\Omega} \sqrt{  \det \overline{g}} \circ I_{h_k} \,  dx}.   
  \end{align*}
  We have  $\sqrt{\det\overline{g}} \circ I_{h_k} \to 1$ uniformly in $\Omega$.
  Thus it suffices to show that
  \begin{align*}
 \liminf_{k \to \infty}  \frac{1}{h_k^4}  \int_{\Omega \setminus F_k} \dist( \Id + h_k^2 G_k, SO(n)) \, dx
 \ge \int_\Omega |\sym G|^2 \, dx.
  \end{align*}
  Now for $x \in \Omega \setminus  F_k$ we have by Taylor expansion, 
  \begin{align*}
    |\dist(\Id + h_k^2 G_k) - \sym G_k| \le C h_k |h_k^2 G_k|.
  \end{align*}
  Thus 
  \begin{align*}
    & \, \liminf_{k \to \infty}  \frac{1}{h_k^4}  \int_{\Omega \setminus F_k} \dist( \Id + h_k^2 G_k, SO(n)) \, dx  \\
    \ge & \,  \liminf_{k \to \infty}  \frac{1}{h_k^4}  \int_\Omega   |\chi_{\Omega\setminus F_k} \sym G_k|^2 \, dx 
    \ge \int_\Omega |\sym G|^2 \, dx.
  \end{align*}
  For the last inequality we used that $\sym G_k  \, \chi_{\Omega\setminus F_k}$ converges weakly in $L^2$ to 
  $\sym G$ (since $\mathcal L^n(F_k) \to 0$) and the lower semicontinuity of the $L^2$ norm under weak
  convergence in $L^2$.
  This concludes the proof.
\end{proof}

\begin{proposition}\label{prop:limsup-ineq}
  Let $(\tilde\gamma,\underline{\tilde\nu},w,y,Z,\beta)\in\mathcal{X}^{\mathcal{G}}$. 
  Then there exist sequences $h_k\downarrow 0$ and $\overline{v}_{k}\in W^{1,\infty}(\Omega;\mathbb{R}^n)$ such that
  properties 1-6 of Theorem~\ref{thm:muller-mora-3d} hold.
  Furthermore, the sequence
  $\tilde u_k(x)\coloneqq\tilde\psi_{\tilde{\gamma},\underline{\tilde\nu}}\circ\overline{v}_k\circ\psi_{\gamma,\underline\nu}^{-1}(x)$ satisfies
  \begin{equation}\label{eq:recovery-seq-convergence}
    \limsup_{k\rightarrow\infty}\frac{1}{h_k^4}E_{h_k}(\tilde u_k)
    \leq\mathcal{I}^{\tilde\gamma,\underline{\tilde\nu}}(w,y,Z,\beta).
  \end{equation}
\end{proposition}

\begin{proof}
  The proof follows a similar strategy as in~\cite{Mora2003}.
  Suppose for the moment that $w,y,Z,\beta$ are smooth.
  For $h>0$ define
  \begin{align*}
    v_h(x_1,x')
    &\coloneqq\begin{pmatrix}
      x_1 \\ hx'
    \end{pmatrix}
    +\begin{pmatrix}
      h^2 w \\ hy
    \end{pmatrix}
    -h^2\begin{pmatrix}
      \displaystyle x'\cdot\partial_{x_1}y \\ Zx'
    \end{pmatrix}
    +h^3\beta.
  \end{align*}
  Then $v_h$ is Lipschitz and satisfies properties 1-6.
  Furthermore, 
  \begin{align*}
    d_h v_h&=\Id+\begin{pmatrix}
      h^2\partial_{x_1}w & -h(\partial_{x_1}y)^T \\
      h\partial_{x_1}y & Z
    \end{pmatrix} \\
    &\quad-h^2\left( \begin{array}{c|c|c|c}
      \begin{array}{c}
        \displaystyle\sum_{i,j}x_i\partial_{x_1}^2 y_j+x_j\partial_{x_1}^2 y_i \\
        Zx'
      \end{array} &
      \partial_{x_2}\beta & \dots & \partial_{x_n}\beta 
    \end{array}
    \right)+O(h^3).
  \end{align*}
  Using the identity $(\Id+B^T)(\Id+B)=\Id+2\sym B+B^T B$, we obtain 
  \begin{align*}
    (d_h v_h)^T d_h v_h 
    &=\Id+2h^2\partial_{x_1}w e_1\otimes e_1
    +2h^2\left( 
    \begin{array}{c|c|c|c}
      \partial_{x_1}A \begin{pmatrix}
        0 \\ x'
      \end{pmatrix}
      & \partial_{x_2}\beta & \dots & \partial_{x_n}\beta
    \end{array}
          \right)
    +h^2 A^T A+O(h^3).
  \end{align*}
  Taking the square root and recalling the definition of $G_h$, we have 
  \begin{align*}
    \left[(d_h v_h)^T d_h v_h\right]^{1/2}
    &=\Id+h^2 G_h+O(h^3).
  \end{align*}
  For $h$ sufficiently small we have $\det\left[(d_h v_h)^T d_h v_h\right]^{1/2}>0$, hence by frame indifference
  \begin{align*}
    \dist^2(d_h v_h,SO(n))
    &=\dist^2\left(\left[(d_h v_h)^T d_h v_h\right]^{1/2},SO(n)\right),
  \end{align*}
  By Taylor expansion of $\dist(\cdot,SO(n))$ we obtain
  \begin{align*}
    \frac{1}{h^2}\dist(d_h v_h,SO(n))
    \rightarrow|\sym G|\quad\text{in $L^2$,}
  \end{align*}
  and $\sym{G_h}\rightarrow\sym G$ in $L^2$.
  Finally, the inequality~\eqref{eq:recovery-seq-convergence} follows from~\eqref{eq:rewrite-energy-SOn}.
  Now note that any tuple 
  $(w,y,Z,\beta)\in\mathcal{X}$
  can be approximated by smooth functions $(w^k,y^k,Z^k,\beta^k)$ by density.
  We have $\mathcal{I}^{\tilde\gamma,\underline{\tilde\nu}}(w^k,y^k,Z^k,\beta^k)
  \rightarrow\mathcal{I}^{\tilde\gamma,\underline{\tilde\nu}}(w,y,Z,\beta)$.
  Then the upper bound~\eqref{eq:recovery-seq-convergence} follows by a diagonal argument, see e.g.~\cite[Remark 1.2]{Alberti}.
\end{proof}

The following fact on Lipschitz approximation of $\mathbb{R}^s$-valued Sobolev maps will be a crucial ingredient
in proving the compactness of Theorem~\ref{thm:main-thm}. 

\begin{lemma}\label{lem:existence-lipschitz-approx}
  Let $s,n\geq 1$ and $1\leq p<\infty$ and suppose
  $U\subset\mathbb{R}^n$ is a bounded Lipschitz domain. 
  Then there exists a constant $C=C(U,n,s,p)$ with the following property:
  For each $u\in W^{1,p}(U,\mathbb{R}^s)$ and each $\lambda>0$ there
  exists $v:U\rightarrow\mathbb{R}^s$ such that
  \begin{enumerate}[(i)]
    \item\label{item:lipschitz-approx-1} $\Lip v \le C \lambda$,
  \item\label{item:lipschitz-approx-2} $\mathcal L^n\left( \{x\in U:u(x)\neq v(x)\} \right)
    \leq\displaystyle\frac{C}{\lambda^p}
    \displaystyle\int_{\{x\in U:|du|_e>\lambda\}}|du|_e^p\,dx$.
  \end{enumerate}
  Here $|\cdot|_e$ denotes the Frobenius norm with respect to the standard scalar product on $\mathbb{R}^n$ and $\mathbb{R}^s$, respectively. 
  If $U=\Omega_h \coloneqq (-L,L)\times B_h(0)\subset\mathbb{R}^n$ with $h\in(0,L/2)$ 
  then the constant $C$ in~\eqref{item:lipschitz-approx-1} and~\eqref{item:lipschitz-approx-2} can be chosen independent of $h$ and $L$. 
\end{lemma}

\begin{remark}\label{re:lip_approx_bilipschitz_domains}
  It is easy to see that the constant $C$ can be chosen uniformly for all domains $U'$ which 
  are bilipschitzly equivalent to a given domain $U$, with a fixed bilipschitz constant $M$. 
\end{remark}

\begin{proof}
  For the result for a fixed set $U$, see, for example~\cite[Prop.\ A.1]{Friesecke2002}.
 Closely related results appeared earlier in  
  \cite{Liu1977} and   \cite[Sect 6.6.2, 6.6.3]{EvansGariepy1992}.

  For sets of the form $U=(-L,L)\times B_h(0)\subset\mathbb{R}^n$ one can obtain a bound which
  is uniform in $h$ by first extending to a set $(-L,L) \times (-2h, 2h)^{n-1}$ by reflection multiple reflections
  to extend to $(-L,L)^n$. For the convenience of the reader we give a detailed proof in 
  Appendix~\ref{se:lip_approximation_thin_tubes}.
 \end{proof}

\begin{proof}[{Proof of Theorem~\ref{thm:main-thm}~\eqref{item:main-thm-1}}]
  \begin{step}[Existence of Lipschitz approximations] \, \\
  The construction of the Lipschitz maps $\tilde{v}_h$ is similar to the construction in~\cite{Kroemer2025}. 
  We include the details for the convenience of the reader.
  Recall that by Nash's imbedding theorem~\cite[Theorem 3]{Nash1956},
  $\tilde{\mathcal{M}}$ can be viewed as a subset of $\mathbb{R}^s$ for sufficiently large $s$, 
  and the metric on the tangent space of $\tilde{\mathcal{M}}$ is induced by the Euclidean metric in $\mathbb{R}^s$.
  Define $\overline{u}_k:\Omega_{h_k}\rightarrow\tilde{\mathcal{M}}\subset\mathbb{R}^s$ by 
  \begin{align*}
    \overline{u}_k\coloneqq \tilde u_k\circ\psi_{\gamma,\underline\nu},
  \end{align*}
  and let $\overline{g}_{ij}(x)\coloneqq(\psi_{\gamma,\underline\nu}^* g)(x)(e_i, e_j)$ 
  be the coefficients of the metric on $\mathcal{M}$ in the standard Euclidean basis. 
  By Lemma~\ref{lem:metric-taylor-expansion} we have
  \begin{align*}
    |\overline g_{ij}(x)-\delta_{ij}|\leq Ch_k^2\quad\text{for all $x\in (-L,L)\times B_{h_k}(0)$.}
  \end{align*}
  Since the Frobenius norm, denoted by $|\cdot|_e$,
  of a map in $SO(n)$ is $\sqrt{n}$ and $\tilde{\mathcal{M}}$ is isometrically imbedded into $\mathbb{R}^s$,
  we have
  \begin{align*}
    |d\overline{u}_k|_e\leq(1+Ch_k^2)(\sqrt{n}+\dist(d\tilde u_k,SO(g,\tilde{g}))).
  \end{align*}
  In particular for sufficiently large $k$ we have the following implication:
  \begin{align*}
    \text{if }\quad|d\overline{u}_k|_e\geq 4\sqrt{n}
    &\quad\text{ then }\quad
    \dist(du_k,SO(g,\tilde{g}))\geq\frac{1}{2}|d\overline{u}_k|_e\geq 2\sqrt{n}.
  \end{align*}
  We apply Lemma~\ref{lem:existence-lipschitz-approx} with $u=\overline{u}_k,\,U=\Omega_h$ and $\lambda=4\sqrt{n}$,
  and denote the corresponding Lipschitz approximation by $\overline{v}_k$, 
  and set $E_k^2\coloneqq\{x\in\Omega_{h_k}:\overline{v}_k\neq\overline{u}_k\}$. 
  Then 
  \begin{equation}\label{eq:lip-vk-leq-C}
    \Lip\overline{v}_k\leq C.
  \end{equation}
  Using that $\det\overline{g}(x)\geq(1+Ch_k^2)^{-1}\geq\frac{1}{2}$ we get
  \begin{equation}\label{eq:bad-set-bound}
    \begin{split}
      \mathcal{L}^n(E_k^2) 
      &=\frac{C}{\lambda^2}\int_{\{x\in\Omega_{h_k}:|d\overline{u}_k|_e>\lambda\}}|d\overline{u}_k|_e\,dx \\
      &\leq\frac{C}{\lambda^2}\int_{\Omega_{h_k}}\dist{}^2(d\tilde u_k,SO(g,\tilde{g}))\,d\Vol_g \\
      &\leq C\mu(\Omega_{h_k})h_k^4.
    \end{split}
  \end{equation}
  In the last inequality we used that $\lambda=4\sqrt{n}$, so that $\frac{1}{\lambda^2}\leq 1$.
  Here, the map $\overline{v}_k$ generally takes values in $\mathbb{R}^s$, not in $\tilde{\mathcal{M}}$. 
 
  This issue can easily be resolved by projecting back to $\tilde{\mathcal{M}}$. 
  Indeed, since $\tilde{\mathcal{M}}$ is compact, there exists $\rho>0$ and a smooth projection 
  $\pi_{\tilde{\mathcal{M}}}$ which maps a tubular $\rho$-neighborhood of $\tilde{\mathcal{M}}$ in $\mathbb{R}^s$ to $\tilde{\mathcal{M}}$. 
    By~\eqref{eq:bad-set-bound} we have
    \begin{align*}
      \mathcal{L}^n(E_k^2)=\mathcal{O}(h_k^{n-1+4}),
    \end{align*}
    hence the radius of the largest ball in $E_k^2$ is bounded by $\mathcal{O}(h_k^{1+3/n})$. 
    In particular
    \begin{align*}
      \sup_{x\in E_k^2}\dist(\overline{v}_k(x),\tilde{\mathcal{M}})\leq Ch_k.
    \end{align*}
  Then $\overline{v}_k'\coloneqq\pi_{\tilde{\mathcal{M}}}\circ\overline{v}_k$ 
  is well defined for $h_k$ sufficiently small and satisfies $\Lip\overline{v}_k'\leq C$. 
  Since $\pi|_{\tilde{\mathcal{M}}}=\id$, we have $\{\overline{v}_k'\neq\overline{u}_k\}\subseteq\{\overline{v}_k\neq\overline{u}_k\}$. 
   By Remark~\ref{re:bilipschitz_fermi}  the restriction of $\psi_{\gamma,\underline\nu}$ to a sufficiently
  small cylinder $\Omega_\eps$ is bilipschitz. Thus  $\tilde v_k\coloneqq\overline{v}_k'\circ
   \psi_{\gamma,\underline\nu}^{-1}$ 
  satisfies~\eqref{eq:def-conditions-lipschitz-approx}.

 We also note that
  \begin{equation}  \label{eq:compactness_bound_energy_tildevk}
  E_{h_k}(\tilde v_k) \le E_{h_k}(\tilde u_k) + C h_k^4 \le C h_k^4  \, .
  \end{equation}
  Indeed, the measure of the set where $d\tilde v_k \ne d\tilde u_k$ is bounded by $C h_k^4$ and on that set
  $|d\tilde v_k|$ is uniformly bounded. 
  \end{step}

  \begin{step}[Approximation by a geodesic]
  We will show that  for each sufficiently large $k$ there exists a unit speed  geodesic $\check \gamma_k$ such 
  that 
  \begin{equation} \label{eq:compactness_approx_geod}
   \sup_{x \in \Omega_{h_k} }\dist(\tilde v_k(x), \check \gamma_k(x_1)) \le C h_k.
  \end{equation}
  We will first construct a continuous,  piecewise geodesic curve $\gamma^{pw}$ such that 
  \eqref{eq:compactness_approx_geod} holds with $\gamma^{pw}$ instead of $\check\gamma_k$
  and such that the $l_2$ norm of the jumps of the derivative of $\gamma^{pw}$ and  the quantity $|\gamma^{pw\prime}| -1$ are controlled in
  terms of the energy of $\tilde v_k$. Then an ODE argument shows that there exists a unit speed
  geodesic $\check \gamma_k$ such that \eqref{eq:compactness_approx_geod} holds, see
   Lemma~\ref{lem:approx-pw-geodesics} below.
  In the following we will always assume that $h_k$ is sufficiently small, in particularly much smaller then the injectivity radius of $\tilde M$. 
  
  To construct the piecewise geodesic curve  $\gamma^{pw}$,  let $N_k \coloneqq  \left\lceil  \frac{L}{h_k} \right\rceil$, $\eps_k = \frac{L}{N_k}$,
  and
  $$ t_i \coloneqq  i \eps_k  \quad \text{for $i = -N_k, -N_k+1, \ldots, N_k$,} \qquad
  s_i \coloneqq  \frac{t_i +t_{i+1}}{2}, $$
  and 
  $$ \Omega_{i} = (t_i, t_{i+1}) \times B_{h_k(0)}.$$
  Here $t_i$, $s_i$ and $\Omega_i$ also depend on $k$, but for easier readability we suppressed that dependence.
  Note that $ h_k/2 \le \eps_k \le h_k$. 
  
  Let $\overline{x}_{i}$ be the Riemannian center of mass of the map $\tilde v_k \circ \psi|_{\Omega_{i}}$, see Definition~\ref{def:center-of-mass}. Here we equip $\Omega_{i}$ with the normalized  Lebesgue measure. 
  Since the maps $\tilde v_k$ are uniformly Lipschitz we get
  \begin{equation}  \label{eq:compactess_dist_tildev} 
   \sup_{x \in \Omega_i} d_{\tilde M}(\tilde v_k(x), \overline{x}_i) \le C h_k, \quad
  d_{\tilde M}(\overline{x}_i, \overline{x}_{i+1}) \le C h_k.
  \end{equation}  
  Define on $[s_{-N}, s_{N-1}]$  a piecewise the $C^1$ curve $\gamma^{pw}$ by the condition
  \begin{align*} &  \text{$\gamma^{pw}|_{[s_i, s_{i+1}]}$ is the unique length-minimizing geodesic with } \\
&   \text{$\gamma^{pw}(s_i) = \overline{x}_{i}$ and $\gamma^{pw}(s_{i+1}) = \overline{x}_{i+1}$.}
  \end{align*}
  We extend $\gamma^{pw}$ uniquely to geodesic on $(-\infty, s_{-N+1})$ and $(s_{N-2}, \infty)$. 
  Using  \eqref{eq:compactness_approx_geod} for $x \in \Omega_i$,  we see that
  \begin{equation} \label{eq:compactness_approx_geod_pw}
   \sup_{x \in \Omega_{h_k} }\dist(\tilde v_k(x), \gamma^{pw}(x_1)) \le C h_k.
  \end{equation}
 
The derivative $\gamma^{pw}$ may jump at $s_i$ for $i = -N+1, \ldots N-2$. 
  For such values of $i$, we denote by $\gamma^{pw\prime}_+(s_i)$ the derivative of the geodesic defined on $[s_i, s_{i+1}]$ and by $\gamma^{pw\prime}_-(s_i)$ the derivative of the geodesic defined on $[s_{i-1}, s_i]$. 
  By definition of the exponential map these geodesics a mapped to straight lines under $\exp_{\overline{x}_{i}}^{-1}$
  and we have
  \begin{align} \label{eq:compactness_derivative_gammapw}
   \gamma^{pw\prime}_+(s_i) = \frac1{\eps_k} \exp_{\overline{x}_{i}}^{-1} \overline{x}_{i+1}, \quad
     \gamma^{pw\prime}_-(s_i) = - \frac1{\eps_k} \exp_{\overline{x}_{i}}^{-1} \overline{x}_{i-1},
  \end{align}
  To estimate the size of jump,  we will apply the rigidity estimate for low energy maps to the maps
  $$ f_{i} = \exp_{\overline{x}_{i}}^{-1} \circ \,  \tilde v_k \circ \psi_{\gamma,\underline\nu}.$$
 For $j \in \{-1,0,1\}$ we define
  $$ q_{i,j} \coloneqq\fint_{\Omega_{i+j}} f_{i}\,dx\in T_{\overline{x}_i}\tilde{\mathcal{M}}. $$
  Then $q_{i,0} = 0$ by the definition of the center of mass.  Corollary~\ref{cor:estimate-center-of-mass} implies that
    \begin{equation}\label{eq:estimate-xi-qij_new}
      \left|
      \exp_{\overline{x}_i}^{-1}\overline{x}_{i+j}-q_{i,j}
      \right|\leq Ch_k^3\quad\text{for $j=\pm 1$}.
    \end{equation}
    Thus 
    \begin{equation}  \label{eq:compactness_jump_gammaprime_q}
    \left|  \gamma^{pw\prime}_+(s_i) - \frac1{\eps_k} q_{i,1} \right| \le C h_k^2, \quad
    \left|  \gamma^{pw\prime}_-(s_i) + \frac1{\eps_k} q_{i,-1} \right| \le C h_k^2.
    \end{equation}
    Set 
    $$ A_i \coloneqq  \Omega_{i-1} \cup \Omega_i \cup \Omega_{i+1}.$$
      Then the quantitative rigidity estimate in \cite[Thm.\ 3.1]{Friesecke2002} implies that there exist
       $R_i\in SO(n)$ such that
    \begin{align*}
      \int_{A_{i}}|df_{i}-R_i|^2\,dx
      &\leq\int_{A_{i}}\dist^2(df_i,SO(n))\,dx.
    \end{align*}
    It follows from  \eqref{eq:compactess_dist_tildev}  and the expansion of the metric in normal coordinates in 
    $\tilde {\mathcal M}$
    and Fermi coordinates in $\mathcal M$ that 
    $$\dist(df_i ,SO(n)) \le \dist(d\tilde v_k, SO(g, \tilde g)) \circ \psi_{\gamma,\underline\nu} + C h_k^2.$$
    Set 
    $$ E_{k,i} \coloneqq\fint_{A_{i}}\dist^2(d\tilde{v}_k,SO(g,\tilde{g}))  \circ \psi_{\gamma,\underline\nu} \,dx$$
    and $c_i=\fint_{A_i}f_i\,dx$. Since $\fint_{A_i} x \, dx = s_i e_1$,  
    the Poincar{\'e} inequality shows  there exists a constant $C<\infty$ such that
    \begin{align}
    & \,   \frac{1}{|A_i|}\int_{A_i}|f_i(x)-R_i (x-  s_i e_1) -c_i|^2\,dx   \nonumber \\
      \leq  & \frac{C h_k^2}{|A_i|}\int_{A_i}|df_i-R_i|^2\,dx  \nonumber \\
      \leq  & \frac{C h_k^2}{|A_i|}\int_{A_i}\dist^2(df_i,SO(n))\,dx    
      \leq  C h_k^2(E_{k,i}+h_k^4).  \label{eq:compactness_fi_almost_affine}
    \end{align}
    Since $\fint_{\Omega_i} f_i \, dx = q_{i,0} = 0$ by the definition of the center of mass, and $\int_{\Omega_i} (x - s_i e_1) = 0$, Jensen's inequality implies that
    $$ |c_i| = \left| \fint_{\Omega_i}  f_i - R_i (x- s_i e_i) - c_i \, dx \right|  \le  C h_k (E_{k,i}^{1/2} + h_k^2).$$
    Similarly, by integrating over $\Omega_{i-1}$ and $\Omega_{i+1}$ and using the estimate for $c_i$ we get
    $$ |q_{i,-1} - R_i (-\eps_k e_i)| \le C h_k (E_{k,i}^{1/2} + h_k^2),
    \quad |q_{i,1} - R_i (\eps_k e_i)| \le C h_k (E_{k,i}^{1/2} + h_k^2).
    $$
    In combination with \eqref{eq:compactness_jump_gammaprime_q} we get
    \begin{equation}  \label{eq:compactness_estimate_gammapw'_final}
     \left|  \gamma^{pw\prime}_\pm(s_i) -  R_i  e_i  \right| \le C (E_{k,i}^{1/2} + h_k^2)
    \end{equation}
    and, in particular, 
    \begin{equation}
      \label{eq:compactness_jump_gammaprime_Eki}
    \left|  \gamma^{pw\prime}_+(s_i) - \gamma^{pw\prime}_-(s_i) \right| \le 
    C (E_{k,i}^{1/2} + h_k^2).
    \end{equation}

     Since $\det D\psi$ is bounded from above and below we deduce from 
     \eqref{eq:compactness_bound_energy_tildevk} that 
     \begin{equation} \label{eq:compactness_sum_Eki}  \sum_{i=-N+1}^{N-2} E_{k,i} \le C \frac{2L}{\eps_k} \fint_{ 
      \Omega_{h_k}^*} \dist^2(d \tilde v_k, SO(g, \tilde g))  \, d\Vol_{\mathcal M}
     \le C h_k^3
     \end{equation}
     and hence   $\sum_{i=-N+1}^{N-2} E_{k,i}^{1/2} \le C N^{1/2} h_k^{3/2} \le C h_k$.  
     Thus we finally get
\begin{equation}
    \label{eq:compactness_jump_gammaprime_final}
     \sum_{i=-N+1}^{N-2}  \left|  \gamma^{pw\prime}_+(s_i) - \gamma^{pw\prime}_-(s_i) \right| \le  C h_k.
        \end{equation}
  The estimate    \eqref{eq:compactness_sum_Eki}  implies $E_{k,i} \le C h_k^{3/2}$ for each $i$.
  Since $|\gamma^{pw\prime}|$ is constant on the segments where $\gamma^{pw}$ is a geodesic,
  it follows from \eqref{eq:compactness_estimate_gammapw'_final} that
  \begin{equation} \label{eq:compactness_length_gammpw}
  \sup_{s \in [-L,L]}  
  \left|  |\gamma^{pw\prime}(s) | - 1 \right| \le C h_k^{3/2}.
  \end{equation}

 Using Lemma~\ref{lem:approx-pw-geodesics},  we see that there exists a unit speed geodesic $\check\gamma_k$
 such that 
 $$\sup_{t \in [-L,L]}   \dist_{\tilde M}(\gamma^{pw}(t), \check \gamma_k(t)) \le C h_k.$$
 
 Combining this estimate with \eqref{eq:compactness_approx_geod_pw},  we get
\eqref{eq:compactness_approx_geod}.
 \end{step}
 
 \begin{step}[Compactness]

 Choose a parallel orthonormal frame   $\underline{\check \nu}$  along $\check \gamma$ and let
  $ \Phi_{\check \gamma, \underline{\check \nu}}$ be the corresponding lifting map constructed in Lemma~\ref{lem:existence-lift_new}.
 Set 
 $$\check v_k(x) \coloneqq   \Phi_{\check \gamma, \underline{\check \nu}}(x_1, \tilde \nu_k \circ
  \psi_{\gamma, \underline \nu} \circ I_{h_k}(x)).$$
 Then
 $$ \psi_{\check \gamma, \underline{\check \nu}} \circ \check v_k =   \tilde \nu_k \circ \psi_{\gamma, \underline \nu}
  \circ I_{h_k}.$$
 It follows from the expansion of the metric in Fermi coordinates and the estimate 
 \eqref{eq:compactness_approx_geod} that
 \begin{eqnarray} \int_{\Omega} \dist^2(d_{h_k} \check v_k, SO(n)) \, dx &\le &C h_k^4, 
 \label{eq:compactness_checkv_SOn}\\
 \int_\Omega \left |\check v_k(x) - \binom{x_1}{0} \right|^2 \, dx & \le & C h_k^2. 
   \label{eq:compactness_checkv_L^2}
 \end{eqnarray}
 The last estimate can be improved for the first component $(v_k)_1(x) - x_1$, at least on average, 
 by a small translation along the geodesic $\check \gamma$. For $\tau = \mathcal O(h_k)$
 define translated framed geodesics by
 $$ \check \gamma_{\tau}(x_1) = \check \gamma(x_1 + \tau), \quad
 \underline{\check \nu_\tau}(x_1) = \underline{\check \nu}(x_1 + \tau)  \, .$$
 Then 
 $$ \psi_{\check \gamma_{\tau}, \underline{\check \nu_\tau}}(x_1, x') = 
  \psi_{\check \gamma, \underline{\check \nu}}(x_1 +\tau, x').$$
  If we define 
  \begin{equation}  \label{eq:compactness_define_checkv_tau}
  \check v_{k,\tau}(x) \coloneqq 
  \Phi_{\check \gamma_\tau, \underline{\check \nu_\tau}}(x_1, \tilde \nu_k \circ \psi_{\gamma, \underline \nu} \circ I_{h_k}(x))  \, , 
   \end{equation}
  we see that
  \begin{equation}   \label{eq:compactness_checkv_tau_vs_check_v}
  \check v_{k,\tau}(x) = \check v_k(x) - \tau e_1.
  \end{equation}
  Thus
  $$ \fint_\Omega (\check v_{k,\tau})_1(x) - x_1 \, dx = - \tau + \fint_\Omega (\check v_k)_1(x) - x_1 \, dx.$$
  It follows from  \eqref{eq:compactness_checkv_L^2} that
   there exists a $\overline \tau$ with $|\overline \tau| \le C h_k$ and
  $$ \fint_\Omega (\check v_{k,\overline\tau})_1(x) - x_1 \, dx =  0. $$
  Taking into account  \eqref{eq:compactness_checkv_tau_vs_check_v},   \eqref{eq:compactness_checkv_SOn},  and \eqref{eq:compactness_checkv_L^2}
  we see that the sequence  $k \mapsto \check v_{k,\overline \tau}$ satisfies the
   assumptions of Corollary~\ref{eq:cor_mora-m}.
  Thus there exists $R_k \in SO(n-1)$ such the sequence of  maps $v_k$, defined by
  $$v_k = \begin{pmatrix} 1 & 0 \\ 0 & R_k \end{pmatrix}^T (\check v_{k,\tau})$$
 satisfies  \eqref{eq:muller-mora-matrix-estimates}
 and Properties~\ref{it:mora-n_convergence_w} to~\ref{it:convergence_beta}  of Theorem~\ref{thm:muller-mora-3d}.
  Set 
  $$ \tilde \gamma_k(x_1) = \check \gamma_\tau(x_1) $$
  and define a parallel orthonormal frame along $\tilde \gamma_k$ by
  $$ \tilde \nu_j(x_1) \coloneqq  \sum_{l=2}^n  (R_k)_{lj} (\check \nu_{\tau})_l(x_1).$$
  Since $\check \gamma_k$ is a unit speed geodesic we have
  $\dist_{\tilde M}(\tilde \gamma_k(x_1), \check \gamma_k(x_1)) \le \tau \le C h_k$.
  Thus it follows from   \eqref{eq:compactness_approx_geod} that
  \begin{equation}
   \sup_{x \in \Omega_{h_k} }\dist(\tilde v_k(x), \tilde  \gamma_k(x_1)) \le C h_k.
  \end{equation}

  Using that  $R^T_k R_k = \Id$ we get
  \begin{align*}
  \psi_{\tilde \gamma, \underline{\tilde \nu}}(v_k(x)) = & \, 
  \exp_{\tilde \gamma(v_1(x))} \sum_{j=2}^n  \tilde \nu_j(x_1)  \sum_{i=2}^n (R_k)_{ij} (\check v_{k,\tau})_l(x)  \\
  = & \,  \exp_{\tilde \gamma((\check v_{k, \tau})_1(x))} \sum_{i,j,l=2}^n 
   (\check \nu_{k, \tau})_l(x_1)  (R_k)_{lj} (R_k)_{ij} \check v_{k,\tau}(x)\\
  = & \, \psi_{\check \gamma_\tau, \underline{\check \nu_\tau}} (\check v_{k,\tau}(x)).
 \end{align*}
 Using  \eqref{eq:compactness_define_checkv_tau} we  get
 \begin{equation}  \psi_{\tilde \gamma, \underline{\tilde \nu}} \circ  v_k = \tilde v_k \circ \psi \circ I_{h_k},
\end{equation}
 and we have seen that  the sequence $k \mapsto  v_k$ 
 satisfies  \eqref{eq:muller-mora-matrix-estimates}
 and Properties~\ref{it:mora-n_convergence_w} to~\ref{it:convergence_beta}  of Theorem~\ref{thm:muller-mora-3d}. 
 
 Finally, the compactness of $\tilde {\mathcal M}$ implies that the space of framed unit speed geodesics
 on $(-L,L)$ is compact,  since a framed unit speed geodesic $(\tilde \gamma, \underline{\tilde \nu})$ 
 is determined by $\tilde \gamma(0)$ and an orthonormal frame in $T_{\gamma(0)} \tilde{ \mathcal M}$
 and the geodesic and its frame depend continuously on these data. Thus there exists a subsequence
 such that $(\tilde \gamma_k, \underline{\tilde \nu_k})$ converges to a framed unit speed geodesic
 $(\tilde \gamma, \underline{\tilde \nu})$.
 This concludes the proof of  assertion~\eqref{item:main-thm-1} of  Theorem~\ref{thm:main-thm}.
     \qedhere
     \end{step}
    
    \end{proof}

\begin{lemma}[Approximation of piecewise geodesics by geodesics]\label{lem:approx-pw-geodesics}
Let $\tilde{\mathcal{M}}$ be a smooth, compact Riemannian manifold.  
Let $\delta >0$ be as in Lemma~\ref{lem:existence-lift_new}. 
Then there exists a constant $C >0$ with the following property.

Let $\gamma_{pw}: [0, \infty) \to \tilde{\mathcal{M}}$ be continuous and assume that there exist $0 < t_1 < \ldots t_J < \infty$
such that $\gamma|_{(t_i, t_{i+1})}$ is a geodesic, for $i = 0, \ldots J$ with $t_0 = 0$ and $t_{J+1} = \infty$.
Assume further that 
\begin{equation} \label{eq:broken_error}
|\gamma'(t_j^{-}) - \gamma'(t_j^+)|   \le F_j \quad \text{for $j=1, \ldots, J$,}
\end{equation}
where $\gamma'(t_j^{-})$ is the derivative at $t_j$ of the extension of the geodesic defined on $(t_{j-1}, t_j)$
and  $\gamma'(t_j^{+})$ is the derivative at $t_j$ of the extension of the geodesic defined on $(t_{j}, t_{j+1})$.
Assume also that 
\begin{equation}  \label{eq:broken_error_0}
\left| \,  |\gamma'_{pw}(0^+)| - 1 \,  \right| \le F_0
\end{equation}
and 
\begin{equation}  \label{eq:broken_total_variation}
\sum_{j=0}^J F_j \le \frac12   \, .
\end{equation}
Set
\begin{equation}\label{eq:lemma-def-omega}
  \omega(t)\coloneqq e^{Ct} F_0 + \sum_{k : t_k < t} C  e^{C(t-t_k)} F_k.
\end{equation}

Let  $\tilde \gamma: \mathbb R \to \tilde{\mathcal{M}}$ be the unit speed geodesic with
$\tilde \gamma(0)= \gamma_{pw}(0)$ and $\gamma'(0) = \gamma'_{pw}(0^+)/ |\gamma_{pw}(0^+)|$.
Then 
\begin{equation}
d_{\tilde{\mathcal{M}}}(\gamma_{pw}(t), \tilde \gamma(t)) \le 2 \omega(t)
\qquad \text{if $\omega(t) < \frac{\delta}{4}$.}
\end{equation}
\end{lemma}

\begin{proof}
This is a natural  adaptation of the standard Gronwall type estimate
for differences of approximate  solutions of the geodesic equation $y'' = \Gamma(y',y') + f$ with $f \in L^1$
to the case where $f$ is a finite sum of Dirac masses.

Since $\gamma_{pw}$ is a geodesic on $(t_i, t_{i+1})$,  the speed $|\gamma'_{pw}(t)|$ is constant in this interval.
Thus, it follows from~\eqref{eq:broken_error},~\eqref{eq:broken_error_0}, and 
\eqref{eq:broken_total_variation} that
\begin{equation}  \label{eq:broken_bound_speed_pw}
| \gamma'_{pw}(t)| \le 2 \quad \text{for  $t \notin \{t_1, \ldots, t_J\}$.}
\end{equation}

Set $v = \gamma'_{pw}(0^+)/ |\gamma_{pw}(0^+)|$. By~\eqref{eq:broken_error_0} we have 
\begin{equation}
| \gamma'_{pw}(0^+) - v| \le F_0.
\end{equation}
Let $ \tilde \gamma: \mathbb R \to \tilde{\mathcal{M}}$ be the geodesic with 
\begin{align*}
  \tilde \gamma(0) = \gamma_{pw}(0), \quad \tilde \gamma'(0)= v.
\end{align*}
Fix a parallel orthonormal frame $\underline{\tilde\nu}$ of the normal bundle of 
$\tilde \gamma$ and let $\tilde \psi = \tilde \psi_{\tilde \gamma, \underline{\tilde\nu}}$
be the Fermi coordinate map. 

We will show the following:
if 
\begin{equation}  \label{eq:broken_admissible_Lprime}
\sup_{t  \in [0, L')} d_{\tilde {\mathcal M}}(\gamma_{pw}(t), \tilde \gamma(t)) < \delta/2
\end{equation}
for some $L' > 0$, then 
\begin{equation} d_{\tilde{\mathcal{M}}}(\gamma_{pw}(t), \tilde \gamma(t)) \le  2 \omega(t)
\quad
\text{for all $t \in [0, L')$.}
\end{equation}
From this the assertion of the lemma follows by the usual continuation argument.

Now assume that~\eqref{eq:broken_admissible_Lprime} holds for some $L' > 0$. 
Let $\tilde \Phi$ be the map in 
 the lifting lemma, Lemma~\ref{lem:existence-lift_new}, 
 and define, for $x_1 \in [0, L')$
$$ y_{pw}(x_1)\coloneqq  \tilde \Phi(x_1, \gamma_{pw}(x_1)).$$
Then $\tilde \psi \circ y_{pw} = \gamma_{pw}$  and $y_{pw}$ satisfies the geodesic equation
\begin{equation}
y''_{pw}(t) = \Gamma(y_{pw}(t))(y'_{pw}(t), y'_{pw}(t)) \quad \text{for $t \in (t_i, t_{i+1})$.}
\end{equation}
where $\Gamma$ denotes the Christoffel symbols in Fermi coordinates.
The lift of $\tilde \gamma$ is given by the trivial map 
\begin{align*}
  \tilde y(t) = (t,0).
\end{align*}
Let 
\begin{align*}
  z(t) = y_{pw}(t) - \tilde y(t).
\end{align*}
In Fermi coordinates we have $\Gamma((t,0)) = 0$. Since $\tilde {\mathcal M}$ is smooth an compact and $|y'_{pw}| \le 2$, 
it follows that there exists a constant $C$ which only depends on $\tilde {\mathcal M}$ such that
\begin{equation}\label{eq:broken_od_ineq_z}
|z''(t)| \le C |z(t)|  \quad \text{for $t \in (t_i, t_{i+1})$.}
\end{equation}
Considering the vector $\binom{z}{z'}$,  we get from the Gronwall inequality
\begin{equation}\label{eq:broken_jumps_z}
\left|  \binom{z}{z'}(t) \right| \le e^{C(t- t_i)} \left|  \binom{z}{z'}(t_i^+) \right|   \quad \text{if $t \in (t_i, t_{i+1})$.}
\end{equation}
Moreover,  (after enlarging $C$ if necessary) we 
have
\begin{equation}
\left|  \binom{z}{z'}(t_i^+) -  \binom{z}{z'}(t_i^-) \right|  = |z'(t_j^+) - z'(t_j^-)| \le C F_k, \quad \left|  \binom{z}{z'}(0^+)  \right| \le F_0.
\end{equation}
Note that $|z'(t_j^+) - z'(t_j^-)|$ refers to the Euclidean norm in $\mathbb R^n$ while the definition of $F_k$
involves the norm given by the metric on $\tilde M$. Thus the constant $C$ arises from the estimate of the metric
in Fermi coordinates.
We define
\begin{align*}
  \alpha(t)\coloneqq e^{-Ct} \left|  \binom{z}{z'}(t) \right|
\end{align*}
for $t \notin \{t_1, \ldots, t_J\}$ and similarly we define $\alpha(t_j^\pm)$. 
Then~\eqref{eq:broken_od_ineq_z} and~\eqref{eq:broken_jumps_z} imply that
\begin{equation} 
 \alpha(t) \le  F_0 + \sum_{k \ge 1: t_k \le t} C e^{-Ct_k} F_k,
 \end{equation}
where for $t \in \{t_1, \ldots, t_j\}$ the estimate holds for $\alpha(t^-)$ and $\alpha(t^+)$.
It follows that
\begin{align*}
  |z(t)| \le \omega(t)  \quad \text{for $t \in [0, L')$.}
\end{align*}
By Lemma~\ref{lem:existence-lift_new} we have $|D\tilde \Phi| \le 2$
in a convex set which contains the values of 
$y_{pw}(t)$ and $\tilde y(t)$. Hence
$$d_{\tilde{\mathcal{M}}}(\gamma_{pw}(t), \tilde \gamma(t)) \le  2  |z(t)| \le 2 \omega(t).$$
\end{proof}

\section{Convergence of the energy}

Define 
\begin{align*}
  m^{\tilde\gamma,\underline{\tilde\nu}}\coloneqq\inf_{(w,y,Z,\beta)\in\mathcal{X}}\mathcal{I}^{\tilde\gamma,\underline{\tilde\nu}}(w,y,Z,\beta).
\end{align*}

\begin{theorem}\label{thm:lim-E-minimizer}
  Let $\mathcal{M}$ be complete, $\tilde{\mathcal{M}}$ compact and let $(\gamma,\underline\nu)$ be a
  framed unit speed geodesic in $\mathcal{M}$. 
  Then
  \begin{equation}\label{eq:lim-E-minimizer}
    \lim_{h\rightarrow 0}\frac{1}{h^4}\inf_{u\in W^{1,2}(\Omega_h^*;\tilde{\mathcal{M}})}E_h(u)
    =\overline{m}
    \coloneqq\min_{\mathcal{G}}m^{\tilde\gamma,\underline{\tilde\nu}},
  \end{equation}
  where $\mathcal{G}$ denotes the set of all framed unit speed geodesics $(\tilde\gamma,\underline{\tilde\nu})$ in $\tilde{\mathcal{M}}$.
\end{theorem}

\begin{proof}
  This is a standard consequence of Theorem~\ref{thm:main-thm}. 
  We include the details for the convenience of the reader. 
  First, note that $\mathcal{G}$ is isomorphic to the set $\tilde{\mathcal M}\times SO(n)$ and therefore compact. 

  Consider the map $f: (\tilde \gamma, \underline{\tilde \nu})\mapsto m^{\tilde\gamma,\underline{\tilde\nu}}$. 
  Clearly $f$ is continuous. Since $\mathcal G$ is compact, the minimum on the right hand side of~\eqref{eq:lim-E-minimizer} exists.

  Upper bound:
  Set $L^+\coloneqq\limsup_{h\rightarrow 0}h^{-4}\inf_{\tilde u\in W^{1,2}(\Omega_h^*;\tilde{\mathcal{M}})}E_h(\tilde u)$
  and let $h_k\rightarrow 0$ be a subsequence which realizes the limit superior. 
  Let $(\tilde\gamma,\underline{\tilde\nu})$ be a framed unit speed geodesic and
  let $x\in\mathcal{X}$ be a minimizer of $\mathcal{I}^{\tilde\gamma,\underline{\tilde\nu}}$. 
  It follows from Theorem~\ref{thm:main-thm}~\eqref{item:main-thm-3} that $L^+\leq m^{\tilde\gamma,\underline{\tilde\nu}}$. 
  Optimizing over $(\tilde\gamma,\underline{\tilde\nu})$ we get $L^+\leq\overline{m}$. 

  Lower bound: 
  Set $L^-\coloneqq\liminf_{h\rightarrow 0}h^{-4}\inf_{\tilde u\in W^{1,2}(\Omega_h^*;\tilde{\mathcal{M}})}E_h(\tilde u)$
  and let $h_k\rightarrow 0$ be a subsequence which realizes the limit inferior. 
  Then there exist maps $\tilde u_k$ such that 
  \begin{align*}
    \lim_{k\rightarrow\infty}\frac{1}{h_k^4}E_{h_k}(\tilde u_k)=L^-.
  \end{align*}
  By Theorem~\ref{thm:main-thm}~\eqref{item:main-thm-1} there exists a subsequence $u_{k_j}$
  which converges to $(\tilde\gamma,\underline{\tilde\nu},x)$ with $x\in\mathcal{X}$ in the sense of Definition~\ref{def:new-convergence}. 
  By Theorem~\ref{thm:main-thm}~\eqref{item:main-thm-2} we have $L^-\geq\mathcal{I}^{\tilde\gamma,\underline{\tilde\nu}}(x)\geq\overline{m}$.
\end{proof}

\section{Minimization   of the limiting functional}

Recall that the $\tilde {\mathcal T} := \mathcal T^{\tilde{\mathcal R}}_{\tilde\gamma,\underline{\tilde\nu}}:
(-L,L) \to \mathbb R^{n \times n}$ and   $\mathcal T := \mathcal T^{\mathcal R}_{\gamma,\underline{\nu}}:
(-L,L) \to \mathbb R^{n \times n}$ 
describe the quadratic correction between metric in Fermi coordinates in $\tilde{\mathcal M}$ and
 $\mathcal M$, respectively,  and the Euclidean metric, see~\eqref{eq:def-T}.

\begin{theorem}   \label{th:charaterization_minimum}
The infimum of the functional  $\mathcal{I}^{\tilde\gamma,\underline{\tilde\nu}}$, defined in 
\eqref{eq:def-I},  is attained on $\mathcal{X}$ and we have
\begin{equation}   m^{\tilde\gamma,\underline{\tilde\nu}}\coloneqq\min_{(w,y,Z,\beta)\in\mathcal{X}}
\mathcal{I}^{\tilde\gamma,\underline{\tilde\nu}}(w,y,Z,\beta)
= \min_{\beta \in \mathcal B} \mathcal{J}^{\tilde\gamma,\underline{\tilde\nu}}(\beta),
\end{equation}
where 
\begin{align}
\mathcal{J}^{\tilde\gamma,\underline{\tilde\nu}}(\beta)
= &\,  \fint_{-L}^L \fint_{B_1(0)}  | \hat{\mathcal T}_{11}(x_1, x')   - \overline{ \hat{\mathcal T}}_{11}(x_1) |^2 \, dx'\,dx_1  \nonumber\\
+ &\,  \frac{1}{2}\fint_{-L}^L  \fint_{B_1(0)} \sum_{j=2}^n  |\partial_{x_j} \beta_1(x)    + \hat{\mathcal T}_{1j}(x_1,x')|^2 \, dx'\,dx_1
\nonumber   \\
+ & \, \fint_{-L}^L \fint_{B_1(0)} |\sym \nabla'  \beta'(x)    + \hat{\mathcal T}''(x_1,x')|^2 \, dx'\,dx_1, 
\label{eq:J_decomp_three_terms}
\end{align}
with 
\begin{eqnarray} \hat{\mathcal T}_{ij} &=& \tilde{\mathcal T}_{ij} - \mathcal T_{ij}, \\
\overline{ \hat{\mathcal T}}_{11}(x_1) &=& \fint_{B_1(0)}  \hat{\mathcal T}_{11}(x_1, x') \, dx'  \, , 
\end{eqnarray} 
  and where $\hat{\mathcal T}''$ denotes the $(n-1) \times (n-1)$ submatrix of $ \hat{\mathcal T}$
  corresponding to the  entries $i,j \ge 2$. 
\end{theorem}

It is easy to see that the infimum of the quadratic functional $\mathcal J$ is attained, see 
Proposition~\ref{eq:minimization_cross_section} below.
Since $\mathcal J$ is a quadratic functional in $\beta$,  its minimizers
depends linearly on $\hat{\mathcal T}$. It follows that $ m^{\tilde\gamma,\underline{\tilde\nu}}$
is a quadratic form in the difference of the curvature tensors
$\overline{\tilde{\mathcal R}} - \overline{\mathcal R}$ along the framed  unit speed geodesics  $(\gamma, \underline \nu)$ and
  $(\tilde \gamma, \underline{\tilde \nu})$,
 respectively. Indeed,  the minimization of $\mathcal J$ over $\beta$ can be carried out separately for 
 each $x_1$, since the functional does not
 involve derivatives of $\beta$ with respect to $x_1$. Thus there exists a quadratic from $\mathbb Q$
 on the space of $4$-tensors which satisfy the symmetries of the curvature tensor  such that
 \begin{equation} 
 m^{\tilde\gamma,\underline{\tilde\nu}} = \fint_{-L}^L \mathbb Q\left(\overline{\tilde{ \mathcal R}}(x_1) -
 \overline{\mathcal R}(x_1) \right) \, dx_1.
 \end{equation}
In addition, $\mathbb Q$ has the following properties.

\begin{corollary}   \label{eq:Q_positive_definite}
The quadratic form $\mathbb Q$ is positive definite on the space of $4$-tensors  $\mathcal A$ which have the symmetries
of the curvature tensor\footnote{I.e., 
$ \mathcal A_{ijkl} = - \mathcal A_{jikl} = - \mathcal A_{ijlk} = \mathcal A_{klij}$
and   $ \mathcal A_{ijkl} + \mathcal A_{iklj} + \mathcal A_{iljk} = 0$}
and on those tensors it can be written as
\begin{equation}  \label{eq:decomposition_mathbbQ}
\mathbb Q(\mathcal A) = \mathbb Q_1(\mathcal A^{\mathrm{par}}) +
\mathbb Q_2(\mathcal A') + \mathbb Q_3(\mathcal A''),
\end{equation}
where
\begin{equation}
\mathcal A^{\mathrm{par}}_{kl} = \mathcal A_{1k1l}, \quad 
\mathcal A'_{jkl} = \mathcal A_{1kjl}, \quad \mathcal A''_{ikjl} = \mathcal A_{ikjl}, 
\quad \text{for $i,j,k,l \ge 2$}
\end{equation}
and

\begin{equation} \mathbb Q_1(\mathcal A^{\mathrm{par}})
 \frac{1}{2 (n+1)(n+3)}  | \mathcal A^{\mathrm{par}}|^2 -   \frac{1}{2(n+1)^2(n+3)}  \left(
\tr  
\mathcal A^{\mathrm{par}} \right)^2,
\end{equation}
\begin{equation}  \label{eq:formula_Q2}
\mathbb Q_2(\mathcal A') = \min_{b_1}   \frac{1}{2} \int_{ \{0\} \times  B_1(0)}
\sum_{j \ge 2}  \left(\partial_{x_j} b_1(x')  - \frac13 \mathcal A'_{jkl} x'_k x'_l \right)^2  \, dx', 
\end{equation}
\begin{equation}   \label{eq:formula_Q3}
\mathbb Q_3(\mathcal A'') = \min_{b'}   \int_{\{0\}  \times B_1(0)} 
\sum_{i,j \ge 2}   \left((\sym  \nabla' b')_{ij}(x') - \frac{1}{6} 
\mathcal A''_{ikjl} x'_k x'_l  \right)\, dx'. 
\end{equation}
Note that $\mathbb Q_1$ is positive definite since 
$ (
\tr  
\mathcal A^{\mathrm{par}})^2 \le (n-1) |A^{\mathrm{par}}|^2$.

In particular, we have $\lim_{h \to 0} h^{-4} \inf_{\tilde u} E_h(\tilde u) = 0$ if and only if 
$\overline{\tilde{\mathcal R}}(x_1) = \overline{\mathcal R}(x_1)$ for every $x_1 \in (-L,L)$. 
\end{corollary}

We note that the expression for $\mathbb Q_3$, which captures the behavior normal to $\tilde \gamma$
agrees with the formula for the optimal embedding of $(n-1)$ dimensional balls derived in 
\cite{Kroemer2025}.

We will prove Corollary~\ref{eq:Q_positive_definite}] at the end of this section.
We first recall the following classical result which implies that  the functional $\mathcal J$ attains its minimum.

\begin{proposition}  \label{eq:minimization_cross_section}
  Let $B_1(0) \subset \mathbb R^{n-1}$ and consider the function spaces
$$ \mathscr B_1 \coloneqq \{ b \in W^{1,2}(B_1(0)) : \int_{B_1(0)} b \,  dx' = 0\},$$
$$ \mathscr B' \coloneqq \{ b' \in W^{1,2}(B_1(0); \mathbb R^{n-1}) : \int_{B_1(0)} b' \,  dx' = 0, \, \,   \mathrm{skew} \left(\int_{B_1(0)}
\nabla' b \, dx\right) = 0\}    \, , $$
where $\nabla' = (\partial_{x_1}, \ldots, \partial_{x_{n-1}})$.
Then the following assertions hold. 
\begin{enumerate}
\item   \label{it:min_crosssection_exists}  If $f \in  L^2(B_1(0); \mathbb  R^{n-1})$ and $f' \in
L^2(B_1(0); \mathbb R^{(n-1) \times (n-1)})$  then there exists a unique $\overline b_1 \in \mathscr B_1$ and a unique $\overline b' \in \mathscr B'$ such that
$$ \| f + \nabla' \overline b_1\|_{L^2} = \min_{b_1 \in \mathscr B_1} \| f + \nabla' b_1\|_{L^2}, \quad 
 \| f' + \nabla' \overline b'\|_{L^2} = \min_{b' \in \mathscr B'} \| f + \nabla' b'\|_{L^2}.$$
Moreover,  the maps $f \mapsto \overline b_1$ and $f' \mapsto \overline b'$ are linear and $\| \overline b_1 \|_{L^2} \le \|\overline f\|_{L^2}$ and $ \| \overline b'\|_{L^2} \le \|\overline f'\|_{L^2}$.
\item \label{it:cross_section_linear_compatible} If $f': B_1(0) \to \mathbb R^{(n-1) \times (n-1)}$ is affine  and $(f')^T = f'$ then there exists a unique $b' \in \mathscr B'$ such that
$$ \sym (\nabla' b') = f'.$$
\end{enumerate}
\end{proposition}

\begin{proof} By the Poincar\'e inequality the space $\{ \nabla' b_1 : b_1 \in \mathscr B_1\}$ is a closed subspace of 
$L^2$. Hence $\overline b_1$ is given by the orthogonal projection of $-f$  onto that subspace. Similarly,
the space $\{ \nabla' b' : b' \in \mathscr B'\}$ is a closed subspace of $L^2$ by the Korn and the Poincar\'e inequalities 
and $\overline b'$ is the orthogonal projection of $-f'$ onto that subspace. This concludes the proof 
Assertion~\ref{it:min_crosssection_exists}.

Assertion~\ref{it:cross_section_linear_compatible} is well known. 
Indeed, if  $f'_{ij}(x) = A_{ijk} x_kn+ a_{ij}$ with  $A_{ijk} = A_{jik}$ and $a_{ij} = a_{ji}$,  then it suffices to take
$$ b'_i(x) = \frac{1}{2} (A_{ijk} - A_{jki} + A_{kij}) x_j x_k + a_{ij} x_j  \, .$$
Uniqueness follows from Korn's inequality. 
\end{proof}

\begin{proof}[Proof of Theorem~\ref{th:charaterization_minimum}] 
It follows from Proposition~\ref{eq:minimization_cross_section}  that attains $J$  its minimum on $\mathcal B$. 
Indeed, the minimizer $\overline \beta$ is given 
by
$$ \overline \beta(x_1,x') = \binom{(\overline b_{x_1})_1(x')}{(\overline b'_{x_1})_1(x')}$$
where
$(\overline b_{x_1})_1(x') \in \mathscr B_1$ is the minimizer in
 Proposition~\ref{eq:minimization_cross_section} 
 for $f(x') = (\hat{\mathcal T}_{12}(x_1, x'), \ldots  \hat{\mathcal T}_{1n}(x_1, x'))$ 
and $(\overline b'_{x_1})(x') \in \mathscr B'$  is the minimizer for $f'(x') = \hat{\mathcal T}''(x_1, x')$.

 Set 
 $$\overline m = \min_{\beta \in \mathcal B}\mathcal{J}^{\tilde\gamma,\underline{\tilde\nu}}(\beta).$$
 We  next show that $m^{\tilde\gamma,\underline{\tilde\nu}} \le  \overline m$. 
 We first note that for $y = 0$, $Z = 0$ we have $A = 0$ and 
 \begin{align*}  & \, \sym G(x)  +  \tilde{\mathcal T}(x_1, y(x_1) + x')  - \mathcal T(x_1)  
   =  \sym \left(  \binom{\partial_{x_1} w_1}{0}  
   \bigg| \partial_{x_2} \beta \bigg| \ldots \bigg| \partial_{x_n} \beta  \right) +  \hat{\mathcal T} \, .
 \end{align*}
 If we take
 \begin{align*}
   \partial_{x_1}  \overline w(x_1)  =- \fint_{B_1(0)}  \hat{\mathcal T}_{11}(x_1, x')  \, dx',
 \end{align*}
then we get
\begin{align*}
m^{   \tilde \gamma, \underline{\tilde \nu}  } \le 
& \, \mathcal I^{\tilde \gamma, \underline{\tilde \nu}}(\overline w, 0,0, \overline \beta) \\
  =  & \, \fint_{-L}^L \fint_{B_1(0)} \left |  \sym \left(  \binom{\partial_{x_1} w_1}{0}  
  | \partial_{x_2} \overline  \beta | \ldots | \partial_{x_n} \overline \beta  \right) +  \hat{\mathcal T}\right|^2 \, dx' \, dx_1 \\
  = & \, \mathcal{J}^{\tilde \gamma, \underline{\tilde \nu}}(\overline \beta) = \overline m.
  \end{align*}

\medskip 

Thus it suffices to show that  
$\mathcal I^{\tilde \gamma, \underline{\tilde \nu}}(\overline w, y, Z,   \overline \beta)  \ge \overline m$. 
Then the infimum of $   \mathcal I^{\tilde \gamma, \underline{\tilde \nu}}$
is attained by $(\overline w, 0, 0, \overline \beta)$ and $m^{   \tilde \gamma, \underline{\tilde \nu}  } = \overline m$.
To get a lower bound for $   \mathcal I^{\tilde \gamma, \underline{\tilde \nu}}$  we first note that  for each 
$x_1 \in (-L,L)$ the map
$$ x' \mapsto  \mathbb L(x_1, x') \coloneqq \tilde{\mathcal T}(x_1, y(x_1) + x') - \tilde{\mathcal T}(x_1,  x') - 
\tilde{\mathcal T}(x_1, y(x_1))$$
is linear,  since $ z' \mapsto \tilde{\mathcal T}(x_1, z')$ is a quadratic form. 
Recalling the definition of 
$ \mathcal I^{\tilde \gamma, \underline{\tilde \nu}}$ in \eqref{eq:def-I}--\eqref{eq:def-I_def-A}
we get
$$ \mathcal I^{\tilde \gamma, \underline{\tilde \nu}}(\overline w, y, Z, \overline \beta) = \fint_{-L}^L \fint_{B_1(0)}
|F|^2  \, dx' \, dx_1$$
where 
\begin{align*} F(x) =  & \, \begin{pmatrix} F_{11}(x) & (F')^T(x) \\
F'(x)  &   F''  \end{pmatrix}    \\  
= & \, 
 \sym \left(\begin{array}{c|c|c|c}
      \left(
        \partial_{x_1}w(x_1) 
      \right)e_1+(\partial_{x_1}A(x_1))\begin{pmatrix}
        0 \\ x'
      \end{pmatrix}
      & \partial_{x_2}\beta & \dots & \partial_{x_n}\beta
    \end{array}
    \right)
    -\frac{1}{2}A^2(x_1) \\
     & \,  +   \hat{\mathcal  T}(x) + \mathbb L(x_1, x') + \tilde{\mathcal T}(x_1, y(x_1)).
\end{align*}
If $f: B_1(0) \to \mathbb R$ is symmetric, i.e., $f(-x) = f(x)$ and $\hat f: B_1(0) \to \mathbb R$ is antisymmetric
then $\int_{B_1(0)} f \hat f \, dx' = 0$. Moreover 
$$\int_{B_1(0)} |f|^2 \ dx'  \ge \int_{B_1(0)}\left| f - \fint_{B_1(0)} f\right|^2 \, dx'.$$
Thus we get
\begin{equation}    \label{eq:lowerbound_I_F11} 
\int_{B_1(0)} F_{11}^2(x_1, x')  \, dx'
\ge   \int_{B_1(0)} \left|  \hat{\mathcal  T}(x_1, x')  - \fint_{B_1(0)}  \hat{\mathcal  T}(x_1, x') \, dx' \right|^2 \, dx'
\, .
\end{equation}

To estimate $F'$, define the symmetric and antisymmetric part of  a map $f': \Omega \to \mathbb R^{n-1}$ by
$2 f'_s(x_1, x') = f'(x_1, x') + f(x_1, x')$ and $2 f'_a(x_1, x') = f'(x_1, x') - f(x_1, x')$ so that $f' = f'_s + f'_a$
Then we get
\begin{align}  \nonumber 
& \,2  \int_{B_1(0)} |F'(x_1, x')|^2 \, dx'  
\ge  \frac{1}{2} \, \int_{B_1(0)} | (\nabla' \beta_1)_s(x_1, x') + c(x_1) + \hat{\mathcal T}'(x_1, x')|^2 \, dx' \\
 \ge  &  \, \frac{1}{2}  \min_{\hat \beta_1} \int_{B_1(0)} | \nabla' \hat  \beta_1(x_1, x') + \hat{\mathcal T}'(x_1, x')|^2 \, dx'
 \,.
 \label{eq:lowerbound_I_F'}
\end{align}
For the second inequality if suffices to note that the function 
$$\hat \beta(x) \coloneqq  \frac{1}{2}\left(  \beta(x_1, x') - \beta(x_1, - x') \right) +  c(x_1) x'$$
satisfies $\nabla' \hat  \beta =  (\nabla' \beta)_s(x_1, x') + c(x_1)$. 

Finally, to estimate $F''$ we use that by Assertion~\ref{it:cross_section_linear_compatible}  in 
Proposition~\ref{eq:minimization_cross_section} there exists 
a $\check \beta': \Omega \to \mathbb R^{n-1}$ such that 
$$ \sym \nabla' \check \beta'(x_1, x') =  \left[- \frac{1}{2} A^2(x_1) + \mathbb L(x_1, x') + \tilde{\mathcal T}(x_1)\right]''.$$
Thus
\begin{align}
& \, \int_{B_1(0)} |F''(x_1, x')|^2 \, dx'  \ge  \int_{B_1(0)} |\sym \nabla' \beta'(x_1, x') + \sym \nabla' \check \beta'(x_1,x') + 
\hat {\mathcal T}(x_1, x')| \, dx' \nonumber \\
\ge & \,  \min_{\hat \beta'} \int_{B_1(0)} | \nabla'  \hat \beta'(x_1, x') + \hat{\mathcal T}''(x_1, x')|^2 \, dx' \,.
 \label{eq:lowerbound_I_F''}
\end{align}
Now it follows from~\eqref{eq:lowerbound_I_F11}, \eqref{eq:lowerbound_I_F'}, and~\eqref{eq:lowerbound_I_F''}
that $\mathcal I^{\tilde \gamma, \underline{\tilde \nu}}(\overline w, y, Z,   \overline \beta)  \ge \overline m$.
This concludes the proof of
Theorem~\ref{th:charaterization_minimum}.
\end{proof}

\begin{proof}[Proof of Corollary~\ref{eq:Q_positive_definite}] The last statement
 follows from the fact that $\mathbb Q$ is positive definite
on the space of tensors which have the symmetries of the curvature tensor and
Theorem~\ref{thm:lim-E-minimizer}. 

The formula \eqref{eq:decomposition_mathbbQ} and the expressions
  \eqref{eq:formula_Q2} and   \eqref{eq:formula_Q3} for $\mathbb Q_2$ and 
  $\mathbb Q_3$ follow from the formula  \eqref{eq:J_decomp_three_terms} for  $\mathcal J$ in Theorem~\ref{th:charaterization_minimum}
  and the formulae for $\mathcal T$ and $\tilde{\mathcal  T}$ in~\eqref{eq:def-T}. 
For $\mathbb Q_1$ we get
\begin{equation}   \label{eq:initial_formula_Q1}
\mathbb Q_1(\mathcal A^{\mathrm{par}}) = \int_{\{0\} \times B_1(0)}  \frac{1}{4}
|\mathcal A^{\mathrm{par}}_{kl} x'_k x'_l - c|^2 = \frac14  \int_{\{0\} \times B_1(0)} 
|\mathcal A^{\mathrm{par}}_{kl} x'_k x'_l|^2 - \frac14 c^2,
\end{equation}
with 
\begin{equation} c = \fint_{\{0\} \times B_1(0)}
A^{\mathrm{par}}_{kl} x'_k x'_l  \, dx'.
\end{equation}
To show the identity
\begin{equation}   \label{eq:formula_Q1_bis} \mathbb Q_1(\mathcal A^{\mathrm{par}})
=   \frac{1}{2 (n+1)(n+3)}  | \mathcal A^{\mathrm{par}}|^2 -   \frac{1}{2(n+1)^2(n+3)}  \left(
\tr  
\mathcal A^{\mathrm{par}} \right)^2,
\end{equation}
for the symmetric matrix $\mathcal A^{\mathrm{par}}$, 
we first note that it suffices to show this identity for diagonal matrices, since both sides are invariant
under orthogonal transformations $\mathcal A^{\mathrm{par}} \mapsto Q^{-1} \mathcal A^{\mathrm{par}} Q$
with $Q \in SO(n)$. Here we use that the measure $n-1$ dimensional Lebesgue measure on $B_1(0)$
is invariant under the action of $Q$. For diagonal matrices $D$, the map $D \mapsto \mathbb Q_1(D)$
is  a quadratic form  of the diagonal entries $(d_1, \ldots, d_{n-1})$ which is invariant under permuation
of the $d_i$. Thus $\mathbb Q_1$ is of the form
$$ \mathbb Q_1(D) = \alpha \sum_i d_i^2 + \beta (\sum_i d_i)^2 = \alpha |D|^2 + \beta (\tr D)^2.$$

To determine $\alpha$ and $\beta$,  we set $m=n-1$ and use that
$$ I_2:= \fint_{\{0 \} \times B_1(0)} {x'_j}^2   \, dx' = \frac1m \fint_{B_1(0)} |x'|^2  dx'  = \frac{1}{m+2}.$$
and (see below for a proof)
\begin{equation}  \label{eq:4th_moment}
I_4:= \fint_{   \{0 \} \times B_1(0)}    (x'_j)^4  \, dx'  = \frac{3}{(m+2) (m+4)}
\end{equation}

Now the definition of $\mathbb Q_1$ implies that
\begin{align}
& \, 4 (m \alpha + m^2 \beta) =  4 \mathbb Q_1(\Id) =    \fint_{   \{0\} \times B_1(0) }
 |x'|^4 \, dx' -  (m I_2)^2
    \nonumber
=  \frac{m}{m+4} - \frac{m^2}{(m+2)^2}   \\
= & \,     \frac{4m}{(m+2)^2 (m+4)}.
\end{align}
 and
 \begin{align}
& \,  4 (\alpha + \beta) = 4 \mathbb Q_1(e_1 \otimes e_1) = I_4 - I_2^2  = 
 \frac{  3(m+2)}{  (m+2)^2 (m+4) } - 
 \frac{m+4}{(m+2)^2 (m+4)} \\
 = & \,\frac{2m + 2}{(m+2)^2 (m+4)}  \, .
 \end{align}
 Thus 
 \begin{eqnarray*}4 \beta &=&  \frac{1}{m-1}  \,   \frac{4 - 2m - 2}{(m+2)^2 (m+4)} = - \frac{2}{(m+2)^2 (m+4)}, \\
4 \alpha &= &\frac{2m + 4}{(m+2)^2 (m+4)} = \frac{2}{(m+2)(m+4)}.
\end{eqnarray*}
Since $m=(n-1)$, this is equivalent to the desired identity~\eqref{eq:formula_Q1_bis}.

 It only remains to show the identity~\eqref{eq:4th_moment} for $I_4$. Denote by $B_R^m$ the open  ball of radius $R$  in
  $\mathbb R^m$ and set
 $$ \omega_m := \mathcal L^m(B_1^m),  \quad a_k = \int_0^{\pi/2} \cos^k t \, dt.$$
 Using the substitution $r = \sin t$,  we get
 \begin{align*} \omega_m =   \int_{B_1^m} 1 \, dy =  \int_0^1  1 \int_{B^{m-1}_{\sqrt{1-r^2}} } \,  dy'  \, dr
 =  \int_0 ^1  (1 - r^2)^{(m-1)/2}    \,  \omega_{m-1} \, dr =  \omega_{m-1}    a_m \, .
 \end{align*}
 Using that $\sin^4 t = (1- \cos^2 t)^2 = 1 - 2 \cos^2 t + \cos^4 t$, we get similarly. 
 \begin{align*}
 \omega I_4 = &  \,  \int_{B^1_m}  y_1^4  \, dx = \int_0^1  r^4   \int_{B^{m-1}_{\sqrt{1-r^2}} } dy'   \, dr
 =  \int_0^1   r^4  (1 - r^2)^{(m-1)/2}    \omega_{m-1} \, dr \\
 = & \, \int_0^1  \sin^4 t \cos^m t \,dt = a_{m+4} - 2 a_{m+2} + a_m  \, .
 \end{align*}
Integration by parts and the identity $\sin^2 t = 1 - \cos^2 t$  give,  for $k \ge 2$, 
\begin{align*}  a_k = \int_0^{\pi/2} \cos^{k-1} t \, \cos t \, dt =
\int_0^{\pi/2}  (k-1) \cos^{k-2} t  \sin t  \,\,  \sin t  \, dt =  (k-1)  (a_{k-2} - a_k).
\end{align*}
Thus $a_k = \frac{k-1}{k} a_{k-2}$ and hence
\begin{align*} I_4 =  & \, \frac{a_{m+4} -2 a_{m+2} + a_m}{a_m}
= \frac{m+3}{m+4}  \,  \frac{m+1}{m+2} - 2 \,  \frac{m+1}{m+2} + 1\\
=  & \, \frac{(m+3)(m+1) - 2 (m+1) (m+4) + (m+2)(m+4)}{(m+2)(m+4)} \\
= & \, \frac{3}{(m+2)(m+4)}.
\end{align*}
This concludes the proof of \eqref{eq:4th_moment}.

\medskip

We now show that $\mathbb Q$ is positive definite  (on the space of tensors which satisfy  the symmetries of the curvature tensor).
It follows from the definition of the quadratic forms $\mathbb Q_i$, 
 that $\mathbb Q$ is  positive semidefinite on that space. 
 To prove positive definiteness,
 assume that $\mathbb Q(\mathcal A) = 0$. Then 
 \begin{align*}
   \mathbb Q_1(\mathcal A^{\mathrm{par}}) = 0, \quad 
   \mathbb Q_2(\mathcal A') = 0, \quad  \text{and} \quad \mathbb Q_3(\mathcal A'') = 0.
 \end{align*}
 
 By~\eqref{eq:formula_Q1_bis}, we get $\mathcal A^{\mathrm{par}} = 0$. 
 The condition $\mathbb Q_2(\mathcal A') = 0$ implies that there exists a $\beta_1$ such that
 \begin{align*}
   -3  \partial_{x_j} \beta_1 =  \mathcal A_{1kjl} x_k x_l  \quad \text{on $\{0\} \times B_1(0)$}
 \end{align*}
and thus
\begin{align*}
  -3 \partial_{x_m} \partial_{x_j} \beta_1 = \mathcal A_{1mjl} x_l +\mathcal  A_{1ljm} x_l.
\end{align*}
Since the second derivatives are symmetric this gives
\begin{align*}
  0 = A_{1mjl} + A_{1ljm} - A_{1jml} - A_{1lmj}.
\end{align*}
Now $- A_{1lmj} = A_{1ljm}$ and 
\begin{align*}
  0 = A_{1ljm} + A_{1jml} + A_{1mlj} =  A_{1ljm}  -( - A_{1jml} +  A_{1mjl}).
\end{align*}
Thus 
$A_{1mjl} - A_{1jml}  = A_{1ljm}$ and hence $A_{1mlj} = 0$, for all $m,l,j \ge 2$. 

Finally, using  the compatibility conditions for a symmetrized gradient one easily concludes that
$\mathbb Q_3(\mathcal A'') = 0$ implies that $\mathcal A'' = 0$, see, for example,  \cite{Maor2019} or
\cite{Kroemer2025}. 
\end{proof}

\section{Uniqueness of the limit in $\mathcal{X}^{\mathcal{G}}/\sim$}  \label{se:unique_limit}

In this section we prove Proposition~\ref{pr:unique_limit}.
In Proposition~\ref{pr:unique_limit} two things vary. We have two uniformly Lipschitz maps $\hat{\tilde v}$ and $\check{\tilde v}$
which agree on a large set and we consider the lift with respect to two different  framed geodesics. 
Thus Proposition~\ref{pr:unique_limit} will follow from Propositions~\ref{pr:unique_same_geodesic}
and~\ref{pr:unique_same_tilde_v} below, where we consider separately the effects
of changing $\hat{\tilde v}$ to  $\check{\tilde v}$ and of the lift with respect to two different framed geodesics.

\begin{proposition} \label{pr:unique_same_geodesic}
Suppose that there exist Lipschitz maps $\hat{\tilde v}$ and $\check{\tilde v}$
from $\Omega^*_{h_k} = \psi(\Omega_{h_k})$ to $\tilde{\mathcal{M}}$ with the following properties:
\begin{equation}  \label{eq:unique_lip_restated} \sup_{k} \max(\Lip \hat{\tilde v}_k, \Lip \check{\tilde v}_k) \le l ,
\end{equation}
\begin{equation}  \label{eq:unique_small_set_restated}
\sup_k \frac{1}{h_k^4} \frac{\mu  \{x \in \Omega_{h_k}^* : 
\hat{\tilde v}_k(x) \ne \check{\tilde v}_k(x) \} }{\mu (\Omega_{h_k})}
< \infty,
\end{equation}
where $\mu$ is the volume measure on $\mathcal M$.
Then 
\begin{equation}  \label{eq:trafo_same_geod_sup_bound}
 \sup_{x \Omega_{h_k}^*} \dist(\hat{\tilde v}(x), \check{\tilde v}(x))  \le C h_k^{1 + 3/n}.
\end{equation} 

If, in addition, 
 there  exist  framed  unit speed geodesics 
$(\check \gamma_k, \underline{\check \nu_k})$ such that 
\begin{equation}
\sup_k   \frac1{h_k}  \dist( \check{\tilde v}_k \circ \psi(x), \check \gamma_k(x_1)) < \infty.
\end{equation}
and if the   lifts  $\check v_k: \Omega \to \mathbb R^n$ defined by 
$$\check \psi_k \circ  \check  v_k = \check{\tilde v}_k \circ \psi \circ I_{h_k}$$
satisfy Properties~\ref{it:mora-n_convergence_w} to~\ref{it:convergence_beta} of Theorem~\ref{thm:muller-mora-3d},  then the lifts
$\hat v_k: \Omega \to \mathbb R^n$ defined by 
$$\check  \psi_k \circ  \hat  v_k = \hat{\tilde v}_k \circ \psi \circ I_{h_k}$$
exist and 
\begin{equation} \label{eq:trafo_same_geod_L2_bound}  
\| \hat v_k - \check v_k \|_{L^2(\Omega)} \le C h_k^{3 + 3/n}. 
\end{equation}
Moreover,  $\hat v_k$ also satisfy Properties~\ref{it:mora-n_convergence_w} to~\ref{it:convergence_beta} of Theorem~\ref{thm:muller-mora-3d}
and
\begin{equation}  \label{eq:trafo_same_geod_limits}
 (\hat w, \hat y, \hat Z, \hat \beta) = (\check w, \check y, \check Z, \check \beta).
 \end{equation}
\end{proposition}

\begin{proof}[Proof of Proposition~\ref{pr:unique_same_geodesic}]
It follows from~\eqref{eq:unique_small_set_restated} that 
\begin{align*}
  \mu  \{x \in \Omega_{h_k}^* : \hat v_k(x) \ne \check v_k(x) \}  = \mathcal O(h_k^{n-1+4}).
\end{align*}
Thus the radius of the largest ball in $\{x \in \Omega_{h_k}^* : 
\hat v_k(x) \ne \check v_k(x) \} $ is bounded by $\mathcal O(h_k^{1+3/n})$.
Together with~\eqref{eq:unique_lip_restated} this implies~\eqref{eq:trafo_same_geod_sup_bound}.

Thus 
$$ \sup_k   \frac1{h_k}  \dist( \hat v_k \circ \psi(x), \check \gamma_k(x_1)) < \infty$$
and therefore  the lift  $\hat v_k$ exists and 
$$ \sup_{x \in \Omega}| \hat v_k(x) - \check v_k(x)| \le C h_k^{1 + 3/n},  $$
$$ \mathcal L^n(  \{ x \in \Omega :  \hat v_k(x) \ne \check v_k(x)  \}  \le C h_k^4.$$
These two estimates imply  \eqref{eq:trafo_same_geod_L2_bound}.
It follows from  \eqref{eq:trafo_same_geod_L2_bound}
that 
$$ \frac{1}{h_k^2} \| \hat y^k - \check y^k \|_{L^2} +
\frac{1}{h_k} \| \hat w_k - \check w_k\|_{L^2} + \frac{1}{h_k} \| \hat Z - \check Z\|_{L^2}
+ \| \hat \beta - \check \beta\|_{L^2} \le C h_k^{3/n}.$$
This implies 
 \eqref{eq:trafo_same_geod_limits}.
\end{proof}

We now come to the heart of the matter, the change of the limits under the change of framed geodesics.

\begin{proposition} \label{pr:unique_same_tilde_v}
Let $\psi = \psi_{\gamma, \underline \nu}$  
and let  $\tilde v_k$ be a sequence of Lipschitz  maps from $\Omega^*_{h_k} = \psi(\Omega_{h_k})$ to $\tilde{\mathcal{M}}$ with 
$\sup_k \Lip \tilde v_k \le l$. 
Assume  that there  exist  framed  unit speed geodesics $(\hat \gamma_k, \underline{\hat \nu}_k)$ and
$(\check \gamma_k, \underline{\check \nu}_k)$ such that 
\begin{equation}  \label{eq:unique_limit_close_to_geodesic}
\sup_k   \frac1{h_k} \sup_{x \in \Omega_{h_k}}   \dist( \tilde v_k \circ \psi(x), \hat \gamma_k(x_1))
+  \frac1{h_k}  \sup_{x \in \Omega_{h_k}}   \dist( \tilde v_k \circ \psi(x), \check \gamma_k(x_1)) < \infty.
\end{equation}
Then there exist  lifts  $\hat v_k: \Omega \to \mathbb R^n$ and $\check  v_k:  \Omega \to \mathbb R^n$ such that
\begin{align*}
  \hat \psi_k \circ \hat v_k = \tilde v_k \circ \psi \circ I_{h_k}, \qquad
\check \psi_k \circ  \check v_k = \tilde v_k \circ \psi \circ I_{h_k},
\end{align*}
where $\hat \psi_k$ and $\check \psi_k$ are the Fermi coordinate maps for the framed geodesics 
 $(\hat \gamma_k, \underline{\hat \nu}_k)$ and
$(\check \gamma_k, \underline{\check \nu}_k)$, respectively,  and where $I_{h_k}(x_1, x') = (x_1, h_k x')$. 
Assume that $\hat v_k$ and $\check  v_k$ satisfy Properties~\ref{it:mora-n_convergence_w} to~\ref{it:convergence_beta} of Theorem~\ref{thm:muller-mora-3d} and let 
$(\hat \gamma, \underline{\hat \nu}, \hat w, \hat y,  \hat Z, \hat \beta)$ and 
$(\check \gamma, \underline{\check \nu},\check w, \check y, \check Z, \check  \beta)$ 
denote the corresponding 
limits defined in Definition~\ref{def:new-convergence}.

Then 
\begin{equation} (\hat \gamma, \underline{\hat \nu}) = (\check \gamma, \underline{\check \nu})
\end{equation}
and 
there exists a Jacobi field  $J: (-L,L) \to \mathbb R^n$ and an infinitesimal change of frame map
$B: (-L,L) \to \mathbb R^{n \times n}_{\mathrm skew}$  such that
\begin{equation}  \label{eq:sametilde_J_and_B}
J_1 = 0, \quad B e_1 = \partial_1 J,
\end{equation}
\begin{equation}  \label{eq:sametilde_J_Jacobi}
  \partial_{x_1}^2  J =-  \overline{\tilde{\mathcal R}}(J, e_1) e_1   \, ,
 \end{equation}
and
\begin{equation}  \label{eq:sametilde_B_Jacobi_style}
  \partial_{x_1} B =  - \overline{\tilde {\mathcal R}}(J, e_1) 
\end{equation}
such that 
for $j, l \ge 2$, 
\begin{eqnarray}
\hat Z_{jl} &=& \check Z_{jl} + B_{jl}     \label{eq:trafo_Z}   \, ,  \\
\hat y &=& \check y + J  \label{eq:trafo_y}   \, , \\
\partial_{x_1} \hat w(x_1) + \tilde{\mathcal T}_{11}(x_1, \hat y(x_1))  
       &=&  \partial_{x_1} \check w(x_1)  +  \tilde{\mathcal T}_{11}(x_1, \check y(x_1)) \nonumber   \\
 & & + \frac{1}{2} (\hat A^2 - \check A^2)_{11}(x_1)     \,  ,   \label{eq:trafo_w} \\
 \partial_{x_j} \hat \beta_1(x_1, x') 
+  2\tilde{ \mathcal T}_{1j}(x_1, \hat y(x_1) + x' )   &= & \nabla' \check \beta_1(x) + 
2 \tilde{\mathcal T}_{1j}(x_1, \check y(x_1) + x') 
\nonumber  \\
                                                       & & e_j \cdot \partial_{x_1}( \check Z  - \hat Z)(x_1)  x'   + (\hat A^2 -  \check A^2)_{1j}(x_1)
 \label{eq:trafo_beta1}  \, ,  \\
(\sym \nabla' \hat \beta')_{jl}(x_1,x') + \tilde{ \mathcal T}_{jl}(x_1, \hat y(x_1) + x' ) 
 &=&(\sym  \nabla'  \hat \beta')_{jl}(x_1, x')  + \tilde{\mathcal T}_{jl}(x_1, \check y(x_1) + x') \nonumber  \\
& &  +  \frac{1}{2} (\hat A^2 -  \check A^2)_{jl}(x_1)  \, , 
  \label{eq:trafo_beta_prime} 
\end{eqnarray}
where  $\nabla' = (\partial_{x_2}, \ldots, \partial_{x_n})$, $\beta' = (\beta_2, \ldots, \beta_n)$ and
 \begin{align*}
   \hat A = \begin{pmatrix}  0 & - (\partial_{x_1} \hat y)^T  \\ \partial_{x_1} \hat y &  \hat Z
  \end{pmatrix},  \quad  
     \check A =  \begin{pmatrix}  0 & - (\partial_{x_1} \check y)^T  \\ \partial_{x_1} \check y &  \check  Z
  \end{pmatrix}.
 \end{align*}
 Here we used the abbreviation $\tilde{\mathcal T} = \mathcal T^{\tilde {\mathcal R}}_{\tilde \gamma,
 \underline{\tilde \nu}}$ with  $(\tilde \gamma,
 \underline{\tilde \nu}) := (\hat \gamma,
 \underline{\hat \nu} )= (\check \gamma,
 \underline{\check \nu})$.
\end{proposition}
\begin{proof}
It follows from \eqref{eq:unique_limit_close_to_geodesic} that $\lim_{k \to \infty}\sup_{x_1 \in (-L,L)} \dist(\hat \gamma_k(x_1),
\check \gamma_k(x_1)) = 0$. Hence
$ \hat \gamma = \check \gamma$ and we will write
$$ \tilde \gamma \coloneqq \hat \gamma = \check \gamma.$$

To prove the other assertions in Proposition~\ref{pr:unique_same_tilde_v}, 
we will show in Proposition~\ref{pr:change_of_geodesic} below  that there exists a lift $\Phi_k: \Omega_{C h_k} 
\to \mathbb R^n$ such that
$$ \check \psi_k = \hat \psi_k \circ \Phi_k$$
and $\Phi_k$ is well behaved, in particular $|D^2 \Phi_k| = \mathcal O(h_k)$ 
and $\dist (D\Phi_k, SO(n)) = \mathcal O(h_k^2)$ on $\Omega_{C h_k}$.
It follows from the definition of $\Phi_k$ that 
$  \hat \psi_k \circ \hat v_k = \check \psi_k \circ \check v_k = \hat \psi_k \circ \Phi_k \circ \check v_k$
and hence
\begin{equation}  \label{eq:unique_hatvk_checkvk}  \hat v_k = \Phi_k \circ \check  v_k.
\end{equation}

The convergence properties of $\hat v_k$ and $\check v_k$ will imply that $h_k^{-1}(D\Phi_k - \Id)$ is bounded
in $L^2$,  and in Proposition~\ref{pr:change_of_geodesic_convergence} 
we analyze the limits of $h_k^{-1}(D\Phi_k - \Id)$ and
$h_k^{-2} \big[(D\Phi_k^T D\Phi_k) - \Id \big]$ and 
show \eqref{eq:sametilde_J_and_B}--\eqref{eq:sametilde_B_Jacobi_style}.

The formulae  
\eqref{eq:trafo_Z}--\eqref{eq:trafo_beta_prime} then follow from 
Proposition~\ref{pr:trafo_formulae_simpler} below. 
\end{proof}

\begin{proposition}\label{pr:change_of_geodesic}
Let $(\tilde{\mathcal{M}}, \tilde g)$ be a smooth,  oriented, compact Riemannian manifold. 
Let $L > 0$ and  $C> 0$. Then there exists a constant $C' >0$ with the following property.
Let $h_k > 0$ with $\lim_{k \to \infty} h_k = 0$,  and let  $\hat \gamma_k$ and $\check \gamma_k$ be unit speed  geodesics in $\tilde{\mathcal{M}}$ with 
\begin{align}   \label{eq:change_geodesic_distance_gamma_hat_check}
  \dist(\check{\gamma}_k(x_1), \hat \gamma_k(x_1)) \le C h_k  \quad \text{for $x_1 \in (-L,L)$.}
\end{align}
Let $\underline{\hat \nu}_ k$ and $\underline{\check \nu}_k$ be positively oriented orthonormal parallel frames of the 
normal bundles of $\hat \gamma_k$ and $\check \gamma_k$, respectively,  and let $\hat \psi_k$ and $\check \psi_k$
denote the corresponding Fermi coordinate maps. 

Then for $k \ge k_0$ there exists a $\rho > 0$  a unique map  $\Phi_k: (-2L, 2L) \times B_\rho(0) \to \mathbb R^n$
such that 
\begin{align*}
  \check \psi_k = \hat \psi_k \circ \Phi_k.
\end{align*}
Moreover, $\Phi_k$ is a smooth diffeomorphism onto its image and,
for $x \in (-2L, 2L) \times B_{Ch_k}(0)$, 
\begin{eqnarray}  \label{eq:Phi_k_almost_SOn}
\dist(D\Phi_k(x), SO(n)) &\le& C' h_k^2,\\
 \label{eq:DPhik_bound_second_derivative}
| D^2 \Phi_k(x) | &\le& C' h_k.
\end{eqnarray}
\end{proposition}

\begin{proof} 
  In the course of the proof we denote all constants which depend only on $C$, $L$, and the manifold
  $(\tilde{\mathcal{M}}, \tilde g)$ by $C'$. 
  
Since $\hat \gamma_k$ and $\check \gamma_k$ are geodesics,  a Gronwall type argument shows that the estimate \eqref{eq:change_geodesic_distance_gamma_hat_check} also holds on $(-3L, 3L)$    with a different constant $C'$. 

The existence of $\Phi_k$(for the  interval  $(-3L, 3L)$) implies 
\begin{equation}  \label{eq:Phik_bounded_etak}  |\Phi_k(x_1,0) - (x_1,0)| \le C' h_k, 
\end{equation} since the maps $\Phi_k$ are  uniformly bilipschitz. 
It follows that $\Phi_k$ maps $(-2L, 2L) \times B_{C h_k}(0)$ to a subset of $(-3L, 3L) \times  B_{ C' h_k}(0)$. 

We denote by $ \hat g_k: = {\hat \psi_k}^* \tilde g$  the pullback of  the metric in $\tilde{\mathcal{M}}$ and we write
 $\hat \Gamma_k$ for the Christoffel symbols of $\hat g_k$. Analogously we define $\check g_k$ and $\check \Gamma_k$. 
 To prove~\eqref{eq:Phi_k_almost_SOn},  we use that 
 \begin{equation}  
   \check g_k = \Phi_k^* \hat g_k.
\end{equation}
If $|z'| \le C h_k$ then the estimates for the metric in Fermi coordinates gives
$$ \delta_{ij} + \mathcal O(h_k^2) =  \check g_k(z)(e_i, e_j) = \hat g_k(\Phi_k(z))(D\Phi_k e_i, D\Phi_k e_j)
= (D\Phi_k(z)^T D\Phi_k(z))_{ij}  + \mathcal O(h_k^2)$$
and this implies \eqref{eq:Phi_k_almost_SOn}.

Similarly, the estimate  \eqref{eq:DPhik_bound_second_derivative} for $D^2 \Phi_k$ follows from  the estimates
of the Christoffel symbols in Fermi coordinates.  
If we consider the curve $\alpha(t) = x + tv $ and express the covariant derivative of the vector field
$W(t) : = D\check \psi ( \alpha(t))  w = (D \hat \psi \circ \Phi_k (\alpha(t)) \,  D\Phi_k(\alpha(t)) w$ along the curve $\check \psi  \circ \alpha$ \ in both Fermi coordinates we get
\begin{equation}  \label{eq:Phi_k_trafo_christoffel}
D^2 \Phi_k(x)(v,w)   +  \hat \Gamma_k(\Phi_k(x))(D\Phi_k(x) v, D\Phi_k(x)) = D\Phi_k(x)(\check \Gamma_k(x)(v,w)),
\end{equation}
By Lemma~\ref{lem:metric-taylor-expansion} the 
Christoffel symbols are bounded by $C |x'|$. Thus  we get  \eqref{eq:DPhik_bound_second_derivative}.
\end{proof}

\begin{proposition}   \label{pr:change_of_geodesic_convergence} Assume that the hypotheses of Proposition~\ref{pr:change_of_geodesic} hold.
Set
\begin{align*}
  \eta_k(x_1)\coloneqq\Phi_k(x_1,0), 
  \quad Q_k(x_1)\coloneqq D \Phi_k(x_1, 0).
\end{align*}
and assume that, in addition,
\begin{align}   \label{eq:Phik_boundedness_Bk}
  B_k\coloneqq\frac{Q_k - \Id}{h_k} \quad \text{is bounded}
\end{align}
and $(\hat \gamma_k, \underline{\hat \nu}_k)$ converges to $(\tilde \gamma, \underline {\tilde \nu})$ on $(-L,L)$.
Then for subsequences (not relabeled)  we have, 
\begin{align*}
  \frac{1}{h_k} \left( \eta_k(x_1)  - \binom{x_1}{0}\right)  \to J,  \quad B_k \to B,  \quad \text{uniformly},
\end{align*}
\begin{equation}
  B e_1 = \partial_{x_1} J e_1, \quad B^T = -B,   \label{eq:Be1_eq_dJ}
 \end{equation}
and
\begin{eqnarray}  \label{eq:Phik_Jacobi_J}
  \partial_{x_1}^2 J(x_1) &=& -  \overline{\tilde{\mathcal{R}}}(x_1)(J(x_1), e_1)  \, e_1, \quad
  \\
  \partial_{x_1} B(x_1)  
                      &=&-\overline{\tilde{\mathcal{R}}}(x_1)(J(x_1), e_1) \, .
  \label{eq:Phi_k_transport_B}   
 \end{eqnarray}
Moreover, 
\begin{align} 
 & \,  \frac{(D\Phi_k(z_1, h_k z'))^T D\Phi_k(z_1, h_h z')) - \Id}{h_k^2}
 \to \tilde{\mathcal T}(z_1, z')  - \tilde{\mathcal T}(z_1, J'(x_1) + z')   \nonumber \\
 & \, \text{uniformly on $(-2L, 2L) \times B_C(0)$, }
  \label{eq:Phik_convergence_metric_tensor}
\end{align}
where $\tilde{\mathcal T} \coloneqq \mathcal T^{\overline{\tilde{\mathcal R}} }$.
\end{proposition}

\begin{proof}
Let $J_k(x_1) = h_k^{-1} (\eta_k(x_1) - (x_1,0))$.
Then $J_k$ is bounded in $C^1$ by~\eqref{eq:Phik_bounded_etak} and~\eqref{eq:Phik_boundedness_Bk}. 
The geodesic equation for $\eta_k$ can be rewritten as
\begin{align*}
  \partial_{x_1}^2  J_k = - \frac{1}{h_k} \hat \Gamma_k((x_1 e_1 + h_k J_k)(e_1 + h_k \partial_{x_1} J_k, e_1 + h_k \partial_{x_1} J_k).
\end{align*}
Since $\hat \Gamma_k(x_1,0) = 0$   
 and $|\hat \Gamma_k((x_1 e_1 + h_k J_k)|  \le C h_k |J_k|$ by Lemma~\ref{lem:metric-taylor-expansion},
it follows that $J_k$ is bounded in $C^2$. 
Hence a subsequence converges in $C^1$ and by~\eqref{eq:taylor-christoff} 
the limit $J$ satisfies
\begin{align}   \label{eq:geodesic_eqn_for_J}
  \partial_{x_1}^2 J = - D\Gamma_{11}(x_1, 0)(J) = - \overline{ \tilde  R}(J, e_1) e_1.
\end{align}
By the definition of the Fermi coordinates $\check \psi_k$, the vector fields
\begin{align*}
  Q_k(x_1) e_j = D\Phi_k(\eta_k(x_1)) e_j = D \check \psi_k(x_1,0)e_j
\end{align*}
are parallel. 
Since $Q_k = \Id + h_k B_k$ we get
\begin{align*}
  \partial_{x_1} B_k(x_1) e_j = \frac{1}{h_k} \partial_{x_1} Q_k(x_1) e_j = - \frac{1}{h_k}\Gamma(\eta_k(x_1))(e_1, e_j + h_k B_k(x_1) e_j).
\end{align*}
Using again~\eqref{eq:taylor-christoff}, we see that
\begin{align*}
  \partial_{x_1} B(x_1) e_j = - \overline{\tilde{\mathcal{R}}}(x_1)(J(x_1),e_1) e_j.
\end{align*}
or, equivalently $ \partial_{x_1} B(x_1) = - \overline{\tilde{\mathcal{R}}}(x_1)(J(x_1), e_1)$.
This proves \eqref{eq:Phi_k_transport_B}.
The  first identity in \eqref{eq:Be1_eq_dJ} follows from the fact that $B_k e_1 = h_k^{-1} (Q_k e_1 - e_1)
= \partial_{x_1} h_k^{-1} (\eta_k - x_1)$.
The skew symmetry of $B$ follows from the fact that 
$ \Id + 2 h_k \sym B_k  + B_k^T B_k = Q_k^T Q_k = \Id + \mathcal O(h_k^2)$.

\medskip

To show  \eqref{eq:Phik_convergence_metric_tensor},  we use that  $\check \psi_k = \hat \psi_k \circ \Phi_k$ implies 
 $ \check g_k = \Phi_k^* \hat g_k$.
Let $\hat{\mathcal R}_k = \hat \psi_k^* \tilde{\mathcal R}$ be the curvature tensor in the Fermi coordinates $\hat \psi_k$
and let $\hat{\mathcal T}_k(x_1, x') := \mathcal T^{ \tilde{\mathcal R}}_{\hat \gamma_k, \underline{\hat \nu_k}}(x_1, x')$ denote the quadratic form in $x'$ which describes the deviation of
the metric $\hat g_k = {\hat \psi}_k^* \tilde g$ form the Euclidean metric  (see~\eqref{eq:def-T} for the definition of this quadratic form).
The quadratic form $x' \mapsto \check{\mathcal T}_k(x_1, x')$ is defined analogously.

Since $\check g_k = \Phi_k^* \hat g_k$ we have
\begin{align*}
  \check g_k(z)(e_i, e_j) = \hat g_k(\Phi_k(z))(D\Phi_k e_i, D\Phi_k e_j).
\end{align*}
Thus, for $|z'| \le C h_k$ we get
\begin{align*}
  \delta_{ij} + (\check{\mathcal T}_k)_{ij}(z_1, z') + \mathcal O(h_k^3) 
  =(D\Phi_k(z))^T D\Phi_k(z))_{ij} + (\hat{\mathcal T}_k)_{ij}(\Phi_k(z))  + \mathcal O(h_k^3).
\end{align*}
Here we used that $D\Phi_k(z) = \Id + \mathcal O(h_k)$. 
Now the pull back curvature tensors $\hat{\mathcal R}_k $ and $\check{\mathcal R}_k $ only differ by 
$\mathcal O(h_k)$ since $D\Phi_k = \Id + \mathcal O(h_k)$ 
and $\Phi_k(z) - z = \mathcal O(h_k)$.
Hence 
\begin{align*}
  (\check{\mathcal T}_k)_{ij}(z_1, z') = (\hat {\mathcal T}_k)_{ij}(z_1, z') + \mathcal O(h_k^3),
\end{align*}
for  $z' = \mathcal O(h_k)$. 
Also 
\begin{align*}
  (\hat {\mathcal T}_k)_{ij}(\Phi_k(z)_1, \Phi'_k(z)) =  (\hat {\mathcal T}_k)_{ij}(z_1, \Phi'_k(z))  + \mathcal O(h_k^3).
\end{align*}
Thus 
\begin{align*}
  (D\Phi_k(z))^T D\Phi_k(z)) - \Id
  =\hat {\mathcal T}_k(z_1, z')  - \hat {\mathcal T}_k(z_1, \Phi'_k(z))+O(h_k^3).
\end{align*}
Since $D\Phi_k - \Id = \mathcal O(h_k)$ and $\Phi_k(z_1, 0) = \eta_k(z_1)$ we have
\begin{align*}
  \Phi'_k(z) = \eta'_k(z_1)  + z' + \mathcal O(h_k^2)
\end{align*}
and therefore
\begin{align*}
  \frac{(D\Phi_k(z))^T D\Phi_k(z)) - \Id}{h_k^2}
  =\hat {\mathcal T}_k(z_1, h_k^{-1} z')  - \hat {\mathcal T}_k(z_1, h_k^{-1} \eta'_k(z_1) + h_k^{-1} z') + \mathcal O(h_k).
\end{align*}
Using that $\hat {\mathcal T}_k$ is a quadratic form in the second argument and that 
$h_k^{-1} \eta'_k \to J'$ uniformly we get
\begin{align*}
 \frac{(D\Phi_k(z))^T D\Phi_k(z)) - \Id}{h_k} = 
  - \hat {\mathcal T}_k(z_1, h_k^{-1} z')  - \hat {\mathcal T}_k(z_1, J'(x_1) + h_k^{-1} z')   + o(1).
\end{align*}
The pullback curvature tensors $\hat \psi_k^* \tilde R$ converge to $\tilde \psi^* \tilde R$ 
uniformly on thin cylinders since $(\hat \gamma_k, \underline{\hat \nu_k})$ converges 
to $(\tilde \gamma, \underline{\tilde \nu})$ (note that is also implies the convergence of $\hat \gamma'_k$
to $\tilde \gamma'$).
Thus the quadratic forms  $\hat{\mathcal T}_k(x_1)$ converge to $\tilde{\mathcal T}(x_1)$ uniformly,  and we get~\eqref{eq:Phik_convergence_metric_tensor}.
\end{proof}

\begin{proposition}  \label{pr:trafo_formulae_simpler}
Let $\Omega = (-L,L) \times B_1(0) \subset \mathbb R^n$. 
Let $\Phi_k$ be as in Propositions~\ref{pr:change_of_geodesic} 
and~\ref{pr:change_of_geodesic_convergence}  and suppose that 
$\hat v_k: \Omega \to \mathbb R^n$ and $\check v_k: \Omega \to \mathbb R^n$
satisfy properties 1.\ to 6.\ of  Theorem~\ref{thm:muller-mora-3d} and assume that 
$$|\check v_k(x) - (x_1,0)| + |\hat v_k(x) - (x_1,0)| \le C  h_k.$$
Assume further that  $\hat v_k = \Phi_k \circ \check v_k$.
Let $\eta_k(x_1) = \Phi_k(x_1, 0)$ and $Q_k(x_1) = D\Phi_k(x_1, 0)$ be  as in 
Proposition~\ref{pr:change_of_geodesic_convergence}. Then 
\begin{eqnarray}  \label{eq:trafo_convergence_B}
 B_k(x_1)\coloneqq \frac{Q_k(x_1) - \Id}{h_k} &\to& B(x_1),\\
 \label{eq:trafo_convergence_eta}
 \frac{\eta_k(x_1) - x_1}{h_k} &\to& J(x_1)
\end{eqnarray}
uniformly, as $k \to \infty$. In addition, 
$J_1 = 0$  and $J$ and $B$ satisfy   \eqref{eq:Phik_Jacobi_J} and  \eqref{eq:Phi_k_transport_B}, respectively.

Moreover, 
the limits $(\check w,\check y, \check Z, \check \beta)$ and 
$(\hat w,\hat y, \hat Z, \hat \beta)$ satisfy
  \eqref{eq:trafo_Z}--\eqref{eq:trafo_beta_prime}.
\end{proposition}

\begin{proof}
First, we show that 
\begin{align*}
  B_k(x_1) = \frac{D\Phi_k(x_1,0) - \Id}{h_k} \to \hat A(x_1) - \check A(x_1)
\end{align*}
and $\underline{\hat \nu} = \underline{\check \nu}$.

We have $d_{h_k} \hat v_k = (D\Phi_k \circ \check  v_k) \, D \check v_k$ and thus
\begin{equation}\label{eq:trafo_hatAl_checkAk}
  \frac{d_{h_k}  \hat v_k - \Id}{h_k}
  =\frac{D\Phi_k \circ \check v_k - \Id}{h_k} +  D\Phi_k \circ \check v_k \, 
  \frac{ d_{h_k} \check v_k - \Id}{h_k}. 
\end{equation}
By Property~\ref{item:muller-mora-conv-A} of Theorem~\ref{thm:muller-mora-3d} we have
\begin{align*}
  \frac{d_{h_k}\check v_k - \Id}{h_k}  \to \check A, \quad  \frac{d_{h_k}\hat v_k - \Id}{h_k}  \to \hat A \quad \text{in $L^2(\Omega)$.}
\end{align*}
Since $D\Phi_k$ is bounded, it follows  from~\eqref{eq:trafo_hatAl_checkAk}
that $h_k^{-1} (D\Phi_k \circ \check v_k - \Id)$ is bounded in $L^2(\Omega)$ and hence
\begin{align*}
  D\Phi_k \circ \check v_k \to  \Id \quad \text{in $L^2(\Omega)$.}
\end{align*}
Thus
\begin{align}      \label{eq:trafo_DPhi_hatA}
 \frac{D\Phi_k \circ \check v_k - \Id}{h_k}  \to \hat A - \check A  \quad \text{in $L^1(\Omega)$.}
\end{align}
By assumption, $|\check v_k(x) - (x_1,0)| \le C h_k$.
  Moreover, by~\eqref{eq:DPhik_bound_second_derivative},  we know that  $h_k^{-1} | D^2\Phi_k(z_1, z')|$ is bounded  if $|z'| \le C h_k$. Thus
$$ |(D \Phi_k \circ \check v_k)(x) - B_k(x_1)| = |D \Phi_k( \check v_k(x)) - D\Phi_k(x_1,0)| \le C h_k |\check v_k(x) - (x_1,0)| 
 \le C h_k^2.$$
Therefore  $ h_k^{-1} (D\Phi_k \circ \check v_k - B_k) \to 0$ uniformly
and hence
\begin{align}
  B_k \to B \quad \text{in $L^1((-L,L))$} \quad \text{and} \quad 
  \hat A - \check A = B.   \label{eq:trafo_proof_hatA_checkA}
\end{align}
Since $| \partial_{x_1} B_k(x_1)| \le  h_k^{-1} | D^2\Phi_k(x_1,0)|$ is bounded, we have  $B_k \to B$ uniformly. 
By Proposition~\ref{pr:change_of_geodesic_convergence},  $B$ satisfies~\eqref{eq:Phi_k_transport_B} and is skew symmetric.
The formulae for $\check A$ and $\hat A$ imply that
\begin{equation}
 B_{jl} = \hat Z_{jl} - \check Z_{jl}, 
\quad \text{for $j, l \ge 2$.}
\end{equation}
Hence  \eqref{eq:trafo_Z} holds.  
By definition of the Fermi coordinates and the map $\Phi_k$,  the $j$-th vector in the frame $\underline{\check \nu_k}$
satisfies 
$$( \check \nu_k)_j(0) = D\check \psi_k(0)  e_j= D\hat \psi_k(\Phi_k(0))  \, \, D\Phi_k(0) e_j
= D\hat \psi_k(\eta_k(0))  (D\Phi_k(0))_{lj} e_l.$$
Since $\eta_k(0) \to 0$ and $D\Phi_k(0) \to \Id$ it follows that the limiting frames agree at $x_1=0$ and hence 
everywhere, i.e.
$\underline{\hat \nu} = \underline{\check \nu}$.

\medskip

Next, we show the convergence of $h_k^{-1} \eta_k \to J$ and the identities $J_1 = 0$ and
$J_j = \check y_j - \hat y_j$ for $j \ge 2$. 
By Property~\ref{it:convergence_beta} in Theorem~\ref{thm:muller-mora-3d}
\begin{eqnarray}   \label{eq:trafo_property4_hat}
  \int_{\Omega}  \left| \hat v_k(x_1, x') - \binom{x_1}{h_k \hat y_k(x_1) + h_k x'} \right|^2  \, dx
  &\le& C h_k^4, \\
  \label{eq:trafo_property4_check}
  \int_{\Omega}  \left| \check v_k(x_1, x') - \binom{x_1}{h_k \check y_k(x_1) + h_k x'} \right|^2  \, dx
 &\le& C h_k^4.
\end{eqnarray}
Now, for $|z'| \le C h_k$,  the boundedness of $h_k^{-1} D^2 \Phi_k$ and the boundedness of $B_k$
implies that
\begin{align*}
  D\Phi_k(z_1, z') = \Id + h_k B_k(z_1) + \mathcal O(h_k^2) = \Id + \mathcal O(h_k).
\end{align*}
Thus,  for $|z'| \le C h_k$, 
\begin{align*}
  \Phi_k(z_1, z' ) - \eta_k(z_1)   = \Phi(z_1, z') - \Phi(z_1, 0) = \binom{0}{z'} + \mathcal O(h_k^2)
\end{align*}
It follows that 
\begin{align*}
  \hat v_k(x) = \Phi_k (\check v_k(x))   = \eta_k((\check v_k)_1(x)) + \binom{0}{\check v'_k} + \mathcal O(h_k^2).
\end{align*}
Since $\partial_{x_1} \eta_k(x_1) = D\Phi_k(x_1, 0) e_1$ is bounded, we get
\begin{align*}
  \left|  \hat v_k(x) - \eta_k(x_1) - \binom{0}{\check v'_k(x)} \right|
  \le C h_k^2 + C |(\check v_k)_1(x) - x_1|.
\end{align*}
With~\eqref{eq:trafo_property4_hat} and~\eqref{eq:trafo_property4_check} it follows that 
\begin{align*}
& \,   \int_{\Omega} \left|    \binom{x_1}{h_k \hat y_k(x_1) + h_k x'} - \eta_k(x_1) -  \binom{0}{h_k \check y_k(x_1) + h_k x'}
  \right|^2   \, dx \\
  \le  & \, C h_k^4  + \int_{\Omega} |(\check v_k)_1(x) - x_1|^2 \, dx \le C h_k^4.
\end{align*}
This implies that 
\begin{align*}
  \lim_{k \to \infty} \frac{1}{h_k} \left(   \eta_k  - \binom{\id}{0} \right)
  =\lim_{k \to \infty}  \binom{0}{\hat y_k -\check y_k}
  =\binom{0}{\hat y - \check y}   \quad \text{in $L^2((-L,L))$.}
\end{align*}
Since $\partial_{x_1} \eta_k$ is bounded, the maps $h_k^{-1} (\eta_1(x_1) - (x_1,0))$ converge uniformly to $(0, \check y - \hat y)$.
By Proposition~\ref{pr:change_of_geodesic_convergence} we get 
$J = \binom{0}{\hat y - \check y}$ and $J$ satisfies~\eqref{eq:Phik_Jacobi_J}.
In particular,~\eqref{eq:trafo_y} holds.
 
\medskip

Finally,  we derive the formulae by $\partial_{x_1} (\hat w - \check w)$, $\nabla' (\hat \beta_1 - \check \beta_1)$,  and
$\sym \nabla'( \hat \beta' - \check \beta')$.
By Property~\ref{it:convergence_beta} in Theorem~\ref{thm:muller-mora-3d} we have
\begin{align}   \label{eq:trafo_weak_convergence_S}
  \sym \frac{d_{h_k} \hat v_k - \Id}{h_k^2} \rightharpoonup \hat S,
  \quad\sym \frac{d_{h_k} \check v_k - \Id}{h_k^2} \rightharpoonup \check S
    \qquad \text{weakly in $W^{-1,2}(\Omega)$,}
\end{align}
 where
 \begin{align}  \label{eq:trafo_formula_checkS}
  \check S(x) = 
\begin{pmatrix}
  \check \partial_{x_1} w_1(x_1) - \partial_{x_1}^2 \check y \cdot x'     & \frac{1}{2} (  \nabla' \beta_1(x) +   \partial_{x_1} Z(x_1) x')^T\\
  \frac{1}{2} (  \nabla' \beta_1(x) +   \partial_{x_1} Z(x_1) x')           & \sym \nabla' \beta'(x)
\end{pmatrix}
 \end{align}
with the analogous formula for $\hat S$.
Since $d_{h_k} \hat v_k = D\Phi_k \circ \check v_k   \, \, d_{h_k} \check v_k$, we have
\begin{align}
  \sym \frac{d_{h_k} \hat  v_k - \Id}{h_k^2}  = \sym \frac{D\Phi_k \circ \check v_k - \Id}{h_k^2} +
 \sym \frac{d_{h_k} \check v_k - \Id}{h_k^2}  
 + \sym \left(    \frac{   D\Phi_k \circ \check v_k - \Id}{h_k}    \, \, 
 \frac{d_{h_k} \check v_k - \Id}{h_k}  \right).   \label{eq:trafo_difference_Sk}
\end{align}
 It follows from  \eqref{eq:trafo_DPhi_hatA},  \eqref{eq:trafo_proof_hatA_checkA}, 
 and Property~\ref{item:muller-mora-conv-A} of Theorem~\ref{thm:muller-mora-3d} that
 \begin{equation}  \label{eq:trafo_S_AB_term}
  \lim_{k \to \infty} \sym \left(    \frac{   D\Phi_k \circ \check v_k - \Id}{h_k}    \, \, 
   \frac{d_{h_k} \check v_k - \Id}{h_k}  \right) = \sym(B \check A) =\frac{1}{2}\left(B \check A + \check A B\right).
 \end{equation}
For the last identity we used that $B$ and $\check  A$ are skew symmetric.

To analyze term  $h_k^{-2} \sym (D\Phi_k \circ \hat v_k - \Id)$
we use the identity
\begin{align*}
  Q^T Q = (\Id + (Q-\Id))^T (\Id + (Q-\Id) ) = \Id + 2 \sym (Q - \Id)  + (Q-\Id)^T (Q-\Id).
\end{align*}
Thus
\begin{equation}   \label{eq:trafo_sym_DPhik}
2 \sym \frac{D\Phi_k \circ \check  v_k - \Id}{h_k^2} = \frac{ (D\Phi_k \check v_k)^T D \Phi_k \check v_k - \Id}{h_k^2}
- \frac{(D\Phi_k \circ \check v_k - \Id)^T}{h_k} \,\ \frac{D\Phi_k \circ \check v_k - \Id}{h_k}.
\end{equation}
The second  term on the right hand side converges to $- B^T B = B^2$.  

For the first term on the right we use~\eqref{eq:Phik_convergence_metric_tensor}.
By assumption, $\check v_k$ has values in $(-2L, 2L) \times B_{C h_k}$. 
Since the convergence in~\eqref{eq:Phik_convergence_metric_tensor} is uniform on $(-2L, 2L) \times B_{C}(0)$,
we get 
\begin{align}
& \, \lim_{k \to \infty}  \frac{ (D\Phi_k \check  v_k)^T D \Phi_k \check v_k - \Id}{h_k^2}   \nonumber \\
  = & \, \lim_{k \to \infty} \tilde{\mathcal T}((\check v_k)_1, h_k^{-1} \check v'_k) -
  \tilde{\mathcal T}((\check v_k)_1,  J + h_k^{-1}  \check v'_k)   \nonumber \\
  = & \,  \tilde{\mathcal T}\circ \check  \ell  - 
  \tilde{\mathcal T}\circ \hat \ell   \label{eq:trafo_convergence_metric_error}   
\end{align}
in $L^1(\Omega)$, 
where 
\begin{align*}
  \check \ell(x) = (x_1, \check y(x_1) + x'), \qquad  \hat \ell(x) = (x_1, J(x_1) + \check y(x_1) + x')
  = (x_1,  \hat y(x_1) + x').
\end{align*}
For the last identity in \eqref{eq:trafo_convergence_metric_error}  we used the  pointwise estimate
\begin{align*}\left| \tilde{\mathcal T}((\check v_k)_1(x), h_k^{-1} \check v'_k(x))  - \tilde{\mathcal T}((x_1), h_k^{-1} \check v'_k(x)) \right|
\le C |(v_k)_1(x) - x_1| \le C h_k
\end{align*}
as well as the corresponding estimate for $T((\check v_k)_1(x), J (x_1) +  \frac1{h_k} \check v'_k(x))$,
 the bound  \eqref{eq:trafo_property4_check}, 
  and the fact that $z' \mapsto \mathcal T(z_1, z')$ is quadratic.
  
Combining  \eqref{eq:trafo_weak_convergence_S}, \eqref{eq:trafo_difference_Sk},  \eqref{eq:trafo_S_AB_term},   \eqref{eq:trafo_sym_DPhik},
and \eqref{eq:trafo_convergence_metric_error} 
we get,
\begin{align}
\hat S(x) - \check S(x) =& \,  \frac{1}{2} (B \check A + \check A B)(x_1) + B^2(x_1) + 
\tilde{\mathcal T}(x_1, \check y(x_1)) - \tilde{\mathcal T}(x_1, \hat y(x_1)) \nonumber 
\end{align}
or, equivalently, 
\begin{align}  \label{eq:trafo_final_trafo_S}
  \hat S(x) +  \tilde{\mathcal T}(x_1, \hat y(x_1)) =  \check  S(x) +  \tilde{\mathcal T}(x_1, \check y(x_1))+  \frac{1}{2} \hat A^2 - \frac12 \check A^2.
\end{align}
Together with  \eqref{eq:trafo_formula_checkS} and the corresponding formula for $\hat S$
we get \eqref{eq:trafo_w}, \eqref{eq:trafo_beta1}, and \eqref{eq:trafo_beta_prime}.
To obtain  the identity  \eqref{eq:trafo_w} for $\partial_{x_1} (\hat w - \check w)$ we integrate 
the $(11)$ component of \eqref{eq:trafo_final_trafo_S} in $x'$ over $B_1(0)$
and we observe that
the map 
$x'  \mapsto \tilde{\mathcal T}(x_1, \check y(x_1) +x') - \tilde{\mathcal T}(x_1, \hat y(x_1)+ x')$
is affine since $z' \mapsto \tilde{\mathcal T}(x_1, z')$ is a quadratic form and that  for linear maps
the integral in $x'$  over $B_1(0)$ vanishes.
\end{proof}

\appendix

\section{Lipschitz approximation on thin tubes}  \label{se:lip_approximation_thin_tubes}

\begin{lemma}\label{lem:appendix-existence-lipschitz-approx}
  Let $s,n\geq 1$ and $1\leq p<\infty$ and suppose
  $U\subset\mathbb{R}^n$ is a bounded Lipschitz domain. 
  Then there exists a constant $C=C(U,n,s,p)$ with the following property:
  For each $u\in W^{1,p}(U,\mathbb{R}^s)$ and each $\lambda>0$ there
  exists $v:U\rightarrow\mathbb{R}^s$ such that
  \begin{enumerate}[(i)]
    \item\label{item:appendix-lipschitz-approx-1} $\Lip v \le C \lambda$,
    \item\label{item:appendix-lipschitz-approx-2} $\mathcal L^n\left( \{x\in U:u(x)\neq v(x)\} \right)
    \leq\displaystyle\frac{C}{\lambda^p}
    \displaystyle\int_{\{x\in U:|du|_e>\lambda\}}|du|_e^p\,dx$.
  \end{enumerate}
  Here $|\cdot|_e$ denotes the Frobenius norm with respect to the standard scalar product on $\mathbb{R}^n$ and $\mathbb{R}^s$, respectively. 
  If $U=(-L,L)\times B_h(0)\subset\mathbb{R}^n$ with $h\in(0,L/2)$ 
  then the constant $C$ in~\eqref{item:lipschitz-approx-1} and~\eqref{item:lipschitz-approx-2} can be chosen independent of $h$ and $L$. 
\end{lemma}

\begin{proof}
  For the result for a fixed set $U$, see~\cite[Prop.\ A.1]{Friesecke2002}. 
  Now let $U=(-L,L)\times B_h(0)\subset\mathbb{R}^n$ with $h\in(0,L/2)$ and let $u\in W^{1,p}(U,\mathbb{R}^s)$. 
  We first extend $u$ to $(-L,L)\times(-2h,2h)^{n-1}$ and then to a function which is $8h$-periodic in the second argument. 
  Let $\phi\in C^1(B_1(0)\setminus\overline{B_{1/2}(0)};\mathbb{R}^{n-1})\cap C^0(\overline{B_1(0)}\setminus B_{1/2}(0);\mathbb{R}^{n-1})$
  be a $C^1$ diffeomorphism from $B_1(0)\setminus\overline{B_{1/2}(0)}\subset\mathbb{R}^{n-1}$ to $(-2,2)^{n-1}\setminus\overline{B_1(0)}$ 
  with 
  \begin{align*}
    \sup_{B_1(0)\setminus\overline{B_{1/2}(0)}}|D\phi|+|(D\phi)^{-1}|<\infty
  \end{align*}
  and $\phi|_{\partial B_1(0)}=\id_{\partial B_1(0)}$.
  Then $\phi_h$ given by $\phi_h(y)=h\phi\left(\frac{1}{h}y\right)$ is a $C^1$ diffeomorphism from $B_h(0)\setminus\overline{B_{h/2}(0)}$ 
  to $(-2h,2h)^{n-1}\setminus\overline{B_h(0)}$ with 
  \begin{align*}
    \sup_{B_h(0)\setminus\overline{B_{h/2}(0)}}|D\phi_h|+|(D\phi_h)^{-1}|
    =\sup_{B_1(0)\setminus\overline{B_{1/2}(0)}}|D\phi|+|(D\phi)^{-1}|.
  \end{align*}
  For $u\in W^{1,p}((-L,L)\times B_h(0);\mathbb{R}^s)$ we define an extension $\overline{u}:(-L,L)\times(-2h,2h)^{n-1}\rightarrow\mathbb{R}^s$ by 
  \begin{align*}
    \overline{u}(x_1,x')
    &=\begin{cases}
      u(x_1,x')\qquad&\text{if $x'\in B_h(0)$,} \\
      u(x_1,\phi_h^{-1}(x'))&\text{if $x'\in (-2h,2h)^{n-1}\setminus B_h(0)$.}
    \end{cases}
  \end{align*}
  Then $|d\overline{u}|_e(x_1,x')\leq C'|du|_e(x_1,\phi_h^{-1}(x'))$ and therefore
  \begin{align*}
    \mathop{\int}_{\substack{(-L,L)\times(-2h,2h)^{n-1} \\ |d\overline{u}|>C'\lambda}}
    |d\overline{u}|_e^p\,dx
    &\leq C''\mathop{\int}_{\substack{(-L,L)\times B_h(0) \\ |du|>\lambda}}
    |du|_e^p\,dx 
    \eqqcolon E.
  \end{align*}
  Now we extend $\overline{u}$ to a map $\hat{u}:(-L,L)\times\mathbb{R}^{n-1}\rightarrow\mathbb{R}^s$
  by successive symmetric extension across the planes $x_i=2h+4h\ell,\,\ell\in\mathbb{Z},\,i=2,\dots,n$. 
  Then $x'\rightarrow\hat{u}(x_1,x')$ is $8h$-periodic in $\mathbb{R}^{n-1}$, i.e.\ $\hat{u}(x_1,x'+z')=\hat{u}(x_1,x')$ 
  if $z'\in 8h\mathbb{Z}^{n-1}$. 
  Then 
  \begin{align*}
    \mathop{\int}_{\substack{(-L,L)\times(-4h\ell,4h\ell)^{n-1} \\ |d\hat{u}|>C''\lambda}}|d\hat{u}|_e^p\,dx 
    &\leq(2\ell)^{n-1}E.
  \end{align*}
  Thus for $\ell=\left\lceil\frac{4L}{8h}\right\rceil=\left\lceil\frac{L}{2h}\right\rceil$ 
  we get $\ell\geq 2$ and 
  \begin{align*}
    \mathop{\int}_{\substack{(-L,L)\times(-4L,4L)^{n-1} \\ |d\hat{u}|>C''\lambda}}
    |d\hat{u}|_e^p\,dx
    &\leq(2\ell)^{n-1}E.
  \end{align*}
  Applying~\eqref{item:appendix-lipschitz-approx-1} and~\eqref{item:appendix-lipschitz-approx-2} from Lemma~\ref{lem:appendix-existence-lipschitz-approx}
  to the set $U=(-1,1)\times(-4,4)^{n-1}$ and the map $x\mapsto\frac{C'}{L}\hat{u}(Lx)$ with $\lambda$ replaced by $C''\lambda$ 
  we see that there exists a map $\hat{v}:(-L,L)\times(-4L,4L)^{n-1}\rightarrow\mathbb{R}^s$ such that $\Lip\hat{v}\leq CC''\lambda$
  and
  \begin{align*}
    \mathcal{L}^n(\{x\in(-L,L)\times(-4L,4L)^{n-1}:\hat{v}(x)\neq\hat{u}(x)\})
    &\leq(2\ell+1)^{n-1}\frac{C}{\lambda^p}E.
  \end{align*}
  Let $\ell'=\ell-1$. 
  Then $8\ell' h\leq 4L$ and there exists $z'\in 8h\{-(\ell'-1),\dots,\ell'-1\}^{n-1}$ such that
  \begin{align*}
    \mathcal{L}^n(\{x\in(-L,L)\times(z'+(-4h,4h)^{n-1}):\hat{v}(x)\neq\hat{u}(x)\})
    &\leq\frac{(2\ell)^{n-1}}{(2\ell'-1)^{n-1}}\frac{C}{\lambda^p}E
    \leq 4^{n-1}\frac{C}{\lambda^p}E,
  \end{align*}
  since $2\ell'-1=2\ell-3$ and $\ell\geq 2$. 
  Set $v(x_1,x')=\hat{v}(x_1,x'+z')$ for $x'\in B_h(0)$. 
  Then
  \begin{align*}
    \Lip v\leq CC'\lambda\quad\text{and}\quad u(x_1,x')=\hat{u}(x_1,x'+z'),
  \end{align*}
  since $x'\mapsto\hat{u}(x_1,x')$ is $8h$-periodic. 
  Therefore
  \begin{align*}
    \mathcal{L}^n(\{x\in(-L,L)\times B_h(0):v(x)\neq u(x)\})
    &\leq 4^{n-1}\frac{C}{\lambda^p}E 
    =4^{n-1}\frac{CC''}{\lambda^p}
    \mathop{\int}_{\substack{(-L,L)\times B_h(0) \\ |du|>\lambda}}|du|_e^p\,dx. 
  \end{align*}
  This concludes the proof.
\end{proof}

\section{Fermi coordinates}\label{se:fermi_proof}

\begin{proof}[Proof of Lemma~\ref{lem:metric-taylor-expansion}]
The Fermi map $\psi$ is smooth since the solutions of the ordinary differential equation for geodesics depends smoothly
on the initial data. To estimate the metric $\overline g = \psi^*g$ and its Christoffel symbols  $\overline \Gamma$ we first note that the definition of
the Fermi coordinates implies that
\begin{equation} \label{eq:fermi_g_Gamma_on_axis}
 \overline g(x_1,0)(e_i, e_j) = \delta_{ij} \quad \text{and} \quad \overline \Gamma(x_1,0) = 0.
 \end{equation}
We define a frame $(E_1(x), \ldots, E_n(x))$ as follows: $E_i(x_1, x')$ is obtained by parallel transport
 of the canonical basis vector $e_i$ of $\mathbb R^n$ from $(x_1,0)$ to $(x_1, x')$ along the
  geodesic $t \mapsto (x_1, t x')$. 
Then 
$$ \overline g(x)(E_i(x), E_j(x)) = \delta_{ij}$$
since the covariant derivative of the metric vanishes. 
We define $A^j_i(x)$ by
$$ e_i = A^{j}_i(x) E_j(x).$$
Then 
\begin{equation}  \label{eq:fermi_metric_in_A}
\overline g_{ij}(x) = \overline g(x)(e_i, e_j) = A^k_i(x) A^l_j(x) \, \overline g(x)(E_k, E_l) = A^k_i(x) A^k_j(x).
\end{equation}

\medskip

To estimate $A^{j}_i - \delta^{j}_i$,  we will use Jacobi field estimates. For $i \ge 2$ the maps $t \mapsto (x_1, t(x' + s e_i))$ are geodesics and hence 
$$J(t) \coloneqq t e_i = t A^j_i(x_1,tx') E_j(x_1, t x')$$
is a Jacobi fields along the geodesic $t \mapsto (x_1, t x')$. 
The Jacobi field equation $\frac{D}{dt} \frac{D}{dt} J = -\overline{\mathcal R}(x_1,  tx')(J, x') x'$ implies that
\begin{align*}  & \, \partial_t^2 (t A^j_i(x_1, tx') ) E_j(x_1, tx')  =  -\overline{\mathcal R}(x_1, tx')(x',J, x')    \\
=  &\, -\overline{\mathcal R}(x_1, tx')^k(x',te_i,x')e_k = -\overline{\mathcal R}^k(x',e_i,x')  t A^l_k(x_1, tx') E_l(x_1, tx').
\end{align*}

With the notation 
$$A(t;x) \coloneqq A(x_1 + t x'), \qquad
M^k_i(t;x) \coloneqq -\overline{\mathcal R}^k(x_1 + tx')(x',e_i,x')$$  this can be written as
\begin{equation}  \label{eq:ODE_Aij}
\partial_t^2 \left( t A^j_i(t;x) \right)
=  M^k_i(t;x) \, \left(  t  A^j_k(t;x) \right)  \quad \text{if $i \ge 2$.}
\end{equation}
Using the first identity in  \eqref{eq:fermi_g_Gamma_on_axis}, 
we see that $B^{j}_i(t;x) \coloneqq t A^j_i(t;x)$ satisfies the initial conditions
\begin{equation}   \label{eq:initial_Aij} 
 B^{j}_i(t;x) = 0, \quad  \partial_t B^{j}_i(0;x) = A^{j}_i(0;x) = \delta^j_i.
\end{equation}

Similarly, $t \mapsto (x_1 + s, t x')$ are geodesics, so that $e_1$ 
is a Jacobi field along the geodesic $t \mapsto (x_1, t x')$. This yields the ODE 
\begin{equation}  \label{eq:ODE_A1j}
\partial_t^2  A^j_1(t;x) 
=  M^k_1(t;x) \,   A^j_k(t;x)
\end{equation}
and by  \eqref{eq:fermi_g_Gamma_on_axis} we have the initial conditions
\begin{equation}   \label{eq:initial_A1j} 
 A^{j}_1(t;x) = \delta^j_1 \quad  \partial_t A^{j}_1(0;x) = 0.
\end{equation}
Next we show  that the system of ODEs~\eqref{eq:ODE_Aij} and~\eqref{eq:ODE_A1j}
with initial conditions~\eqref{eq:initial_Aij} and~\eqref{eq:initial_A1j}
implies that 
\begin{eqnarray}
|A^{j}_i(x) - \delta^j_i - \frac16 M^j_i(0;x)|  &\le& C |x'|^3,  \quad \text{for $i \ge 2, j \ge 1$,}  \label{eq:fermi_bound_Aij}\\
|A^{j}_1(x)  - \delta^j_1 - \frac{1}{2} M^j_1(0;x)| & \le & C |x'|^3, \quad \text{for  $j \ge 1$}  \label{eq:fermi_bound_A1j}
\end{eqnarray}
and 
\begin{eqnarray}
|\partial_{x_k} A^{j}_i(x) -  \frac16 \partial_{x_k} M^j_i(0;x)|  &\le& C |x'|^2,  \quad \text{for $i, \ge 2, j \ge 1$,} 
\label{eq:fermi_bound__dervative_Aij}
\\
|\partial_{x_k} A^{j}_1(x) -  \frac{1}{2}  \partial_{x_k} M^j_i(0;x) |  &\le& C |x'|^2,  \quad \text{for $ j \ge 1$.} 
\label{eq:fermi_bound__dervative_A1j}
\end{eqnarray}
Here the constant $C$ can be estimated  in terms of the supremum of the pulled back curvature tensor  $\overline{\mathcal R}$ and its first derivative.

  To show~\eqref{eq:fermi_bound_Aij} and~\eqref{eq:fermi_bound_A1j},
  we use that the system of ODEs given by~\eqref{eq:ODE_Aij} and~\eqref{eq:ODE_A1j}  
  with initial conditions~\eqref{eq:initial_Aij} and~\eqref{eq:initial_A1j}
  is equivalent to the  following fixed point problem
   in $C^0([0;1]; \mathbb R^{n \times n})$:
   \begin{equation} A = T_x[A] \coloneqq \Id + S_x[A]
   \end{equation}
    where $\Id$ denotes the constant function from $[0,1]$ to $\mathbb R^{n \times n}$ with value $\Id$ and the linear
    operator $S_x$ is given by
    \begin{eqnarray*} (S_x) ^j_1[A](t) &\coloneqq &\int_0^t \int_{0}^s  M^k_1(\sigma;x) A^j_k(\sigma)\, d\sigma \, ds,  
    \quad \text{for $j \ge1$,}
    \\
    (S_x) ^j_i[A](t) &\coloneqq &\frac1t \int_0^t \int_{0}^s  \sigma  M^k_1(\sigma;x) A^j_k(\sigma) \, d\sigma \, ds, 
    \quad \text{for $i \ge 2$, $j \ge1$.}
    \end{eqnarray*}

Let $ \rho > 0$ and set
$$ L_0 \coloneqq \sup_{ x \in (-2L,2L) \times B_\rho(0)} |\mathcal R|, \quad
 L_1 \coloneqq \sup_{ x \in (-2L,2L) \times B_\rho(0)} |D\mathcal R|.$$
 Then, for $\sigma \in [0,1]$, 
 $$  |M(\sigma;x)| \le C L_0  |x'|^2  \quad \text{for $x \in  (-2L,2L) \times B_\rho(0)$}  $$
 and thus the operator norm of $S_x$ with respect to the supremum norm on $C^0([0;1]; \mathbb R^{n \times n})$
 satisfies
 \begin{equation}  \label{eq:fermi_norm_Sx}  \| S_x\| \le C L_0  |x'|^2  \quad \text{for $x \in  (-2L,2L) \times B_\rho(0)$}.
 \end{equation}
 Hence $T_x$ is a contraction if 
 $L_0 \rho^2$ is sufficiently small. 
 We will assume this from now on and we assume in particular that $L_0 \rho^2 \le 1$.
 Then it follows that  for all
 $x \in  (-2L,2L) \times B_\rho(0)$
 there exists
 a unique fixed point $\overline A_x \in  C^0([0;1]; \mathbb R^{n \times n})$  of $T_x$ and
 $\overline A_x(t) = A(t;x)$.
 Moreover $\overline A_x(t)  = \Id + \sum_{k=1}^\infty S^k[\Id](t)$ and hence~\eqref{eq:fermi_norm_Sx} implies that
 \begin{align*}
  \sup_{t \in [0,1]} |\overline A_x(t) - \Id - S_x[\Id](t)|   \le C L_0^2 |x'|^4.
 \end{align*}
 In particular we have 
 \begin{equation}  \label{eq:fermi_trivial_bound_A}
  \sup_{t \in [0,1]} |\overline A_x(t) - \Id| \le C (L_0 |x'|^2 +  L_0^2 |x'|^4) \le C L_0 |x'|^2.
  \end{equation}
 since $L_0 \rho^2 \le 1$. 
 It is easy to see that 
 $$ |M(t;x) - M(0;x)| \le C L_1 |x'|^3  \quad \text{for $t \in [0,1]$.}$$
 Together with~\eqref{eq:fermi_trivial_bound_A} and the definition of $S_x$ this implies that,
 for $i \ge 2$ and $j \ge 1$, 
 \begin{eqnarray*}
  \sup_{t \in [0,1]}  \left|   (\overline A_x)^j_1(t)  -\delta^j_1 - \frac{1}{2} t^2 M^j_1(0;x) \right| &\le& C(L_0^2 |x'|^4 + L_1 |x'|^3), \\
   \sup_{t \in [0,1]}  \left|   (\overline A_x)^j_1(t)  -\delta^j_1 - \frac16 t^2 M^j_1(0;x) \right| &\le& C(L_0^2 |x'|^4 + L_1 |x'|^3).
 \end{eqnarray*}
   Using this estimate for $t=1$, we get~\eqref{eq:fermi_bound_Aij} and~\eqref{eq:fermi_bound_A1j} since $A(x) =A(1;x) = \overline A_x(1)$.
   
   Now the map $x \mapsto S_x$ is a $C^1$ map from $(-2L,2L) \times B_\rho(0)$ to the space of bounded linear
   operators on $C^0([0,1]; \mathbb R^{n \times n})$. Hence $x \mapsto \overline A_x$ is $C^1$ and the derivative
   $F \coloneqq \partial_{x_k} \overline A_x$ satisfies
   \begin{align*}
    F = S_x[F] + (\partial_{x_k} S_x)[\overline A_x].
   \end{align*}
   It is easy to see that, for $x \in  (-2L,2L) \times B_\rho(0)$ and $t \in [0,1]$,
   \begin{eqnarray}    \label{eq:fermi_partialM_partialM0} 
    |\partial_{x_k} M(t;x) - \partial_{x_k} M(0;x)|& \le& C L_1 |x'|^2, \\
      | \partial_{x_k} M(0;x)| &\le& C L_0 |x'|.    \label{eq:fermi_partialM0} 
   \end{eqnarray}
   Together with~\eqref{eq:fermi_trivial_bound_A} we get,  for $i \ge 2$ and $j \ge 1$, 
   \begin{eqnarray}  \label{eqref:partialA_1j}
  \sup_{t \in [0,1]}  \left|  F^j_1(t)  -(S_x)^j_1[F](t)  - \frac{1}{2} t^2 \partial_{x_k}  M^j_1(0;x) \right| &\le& C( L_1 |x'|^2 + L_0^2 |x'|^3)\\
   \sup_{t \in [0,1]}    \left|  F^j_i(t)  -(S_x)^j_i[F](t)  - \frac16 t^2   \partial_{x_k} M^j_1(0;x) \right|   &\le& C( L_1 |x'|^2 + L_0^2 |x'|^3).  \label{eqref:partialA_ij}
   \end{eqnarray}
   Since $\| S_x\| \le C L_0 |x'|^2$ and $| \partial_{x_k} M(0;x)| \le C L_0 |x'|$ this implies 
   \eqref{eq:fermi_bound__dervative_Aij} and \eqref{eq:fermi_bound__dervative_A1j}.

\medskip

We now turn to the estimates for the pull back metric $\overline g$ and its Christoffel symbols $\overline \Gamma$. 
For ease of notation we simply write $g$ and $\Gamma$ instead of $\overline g$ and $\overline \Gamma$.
It follows from~\eqref{eq:fermi_metric_in_A},~\eqref{eq:fermi_bound_Aij}, and~\eqref{eq:fermi_bound_A1j} that  
$$\overline g_{ij} - \delta_{ij} = (\mathcal A^j_i - \delta^j_i) + (\mathcal A^{i}_j - \delta^{i}_j) + \mathcal O(|x'|^4).$$
Moreover we have 
\begin{align*}
  M^j_i(0; x) = - \overline{\mathcal R}_{jkil}(x_1) x'_k x'_l
\end{align*}
where 
\begin{align*} 
& \, \overline{\mathcal R}_{jkil}(x_1)\coloneqq   (\psi^*\mathcal R)(x_1,0)(e_j, e_k, e_i, e_l) 
=\, \mathcal R(\psi(x_1, 0))(a_j, a_k, a_i, a_l)
\end{align*}
with $a_j = D\psi(x_1,0) e_j$, i.e., $a_1 = \gamma'(x_1)$ and $a_j = \nu_j(x_1)$ for $j \ge 2$.

In particular,  $M^j_i(0;x)$ is symmetric in $i$ and $j$, since $ \overline{\mathcal R}_{jkil} =
\overline{\mathcal R}_{iljk}$. Thus
\begin{eqnarray}
\left|\overline g_{ij}(x) - \delta_{ij} -  \frac13 M^j_{i}(0;x) \right|&\le& C |x'|^3,  \quad \text{for $i \ge 2, j \ge 2$,}  \label{eq:fermi_metric_bound_ij}\\
\left|\overline g_{1j}(x) - \delta_{ij} -  \frac23 M^j_{1}(0;x) \right|&\le& C |x'|^3,  \quad \text{for $ j \ge 2$,} 
 \label{eq:fermi_metric_bound_1j}\\
 \left|\overline g_{11}(x) - \delta_{ij} -   M^1_{1}(0;x) \right|&\le& C |x'|^3. \label{eq:fermi_metric_bound_11}
\end{eqnarray}
This shows that \eqref{eq:taylor-metric} holds.

\medskip

To estimate the Christoffel symbols we note that~\eqref{eq:fermi_metric_in_A} and 
 \eqref{eq:fermi_bound_Aij}--\eqref{eq:fermi_bound__dervative_A1j} imply that
 $$ \partial_{x_k}\overline g_{ij} =  \partial_{x_k} A^{j}_i + \partial_{x_k} A^{i}_j + \mathcal O(|x'|^2). $$
 Using the symmetry of $M^j_i(0;x)$ under exchange of $i$ and $j$  we get, for $i,j,k \ge 2$,
 \begin{eqnarray*}
 3 \partial_{x_k}\overline g_{ij}(x) &=&-[\overline{\mathcal R}_{jkil}(x_1) + \overline{\mathcal R}_{jlik}(x_1)]  x'_l + \mathcal O(|x'|^2)
 =  [\overline{\mathcal R}_{kjil}(x_1) + \overline{\mathcal R}_{kijl}(x_1)]  x'_l + \mathcal O(|x'|^2)  \, , \\
  3 \partial_{x_i}\overline g_{jk}(x) &=&- [ \overline{\mathcal R}_{kijl}(x_1) + \overline{\mathcal R}_{klji}(x_1)]  x'_l + \mathcal O(|x'|^2) = 
 -  [ \overline{\mathcal R}_{kijl}(x_1) + \overline{\mathcal R}_{jikl}(x_1)]  x'_l  +  \mathcal O(|x'|^2)  \, , \\
   3 \partial_{x_j}\overline g_{ki}(x) &=&- [ \overline{\mathcal R}_{ijkl}(x_1) + \overline{\mathcal R}_{ilkj}(x_1)]  x'_l + \mathcal O(|x'|^2)
   =  - [ \overline{\mathcal R}_{ijkl}(x_1) + \overline{\mathcal R}_{kjil}(x_1)]    x'_l + \mathcal O(|x'|^2)  \, .
 \end{eqnarray*}
 Here we also used that $\overline{\mathcal R}_{ijkl} = \overline{\mathcal R}_{klij}$ 
 and that $\overline{\mathcal R}$
 is antisymmetric in the first two arguments.  Using this antisymmetry again, we get
 $\overline{\mathcal R}_{jikl}(x_1) + \overline{\mathcal R}_{ijkl}(x_1)  = 0$ and hence
 \begin{align} 6 \overline\Gamma_{kij}(x) =   & \,  3(  \partial_{x_i}\overline g_{jk}(x) + \partial_{x_j}\overline g_{k1}(x)  - \partial_{x_k}\overline g_{ij}(x))
 \nonumber  \\
 =  & \,-  2 [\overline{\mathcal R}_{kjil}(x_1) +  \overline{\mathcal R}_{kijl}(x_1) ] x_l +  \mathcal O(|x'|^2)
 \quad \text{for $i,j,k \ge 2$.}    \label{eq:fermi_Gamma_ijk_2}
 \end{align} 
 For $k=1$  and $i, j \ge 2$ we get $\partial_{x_1}\overline g_{ij} = 0$ and
 \begin{eqnarray}
  \frac32 \partial_{x_i}\overline g_{j1}(x) &=& 
 - [ \overline{\mathcal R}_{1ijl}(x_1) + \overline{\mathcal R}_{ji1l}(x_1)]  x'_l  +  \mathcal O(|x'|^2)   \label{eq:fermi_derivative_g1j_a}  \, , \\
   \frac32  \partial_{x_j}\overline g_{1i}(x) &=&
   - [   \overline{\mathcal R}_{ij1l}(x_1) +\overline{\mathcal R}_{1jil}(x_1) ]   x'_l + \mathcal O(|x'|^2)
   \label{eq:fermi_derivative_g1j_b}  \, .
 \end{eqnarray}
 Thus 
  \begin{align*} 3 \overline\Gamma_{1ij}(x) =   & \,   \frac32 (  \partial_{x_i}\overline g_{j1}(x) + \partial_{x_j}\overline g_{1i}(x)  - \partial_{x_1}\overline g_{ij}(x))   \nonumber \\
 =  & \, -  [\overline{\mathcal R}_{1jil}(x_1) +  \overline{\mathcal R}_{1ijl}(x_1) ] x'_l +  \mathcal O(|x'|^2)
 \quad \text{for $i,j  \ge 2$.}    
 \end{align*}
 
 Combining this identity with~\eqref{eq:fermi_Gamma_ijk_2} we get
 \begin{equation}  \label{eq:fermi_Gamma_normal}
  \overline\Gamma_{kij}(x) = - \frac13 [\overline{\mathcal R}_{kjil}(x_1) +  \overline{\mathcal R}_{kijl}(x_1) ] x'_l +  \mathcal O(|x'|^2)
  \quad \text{if $i,j \ge 2$, $k \ge 1$.}
 \end{equation}
 In view of the estimate for the metric we have
 \begin{equation}  \label{eq:fermi_Gamma_upper_lower} 
  | \overline\Gamma^{k}_{ij}(x) - \overline\Gamma_{kij}(x)|  \le  C |x'|^2.
 \end{equation}
 Thus~\eqref{eq:fermi_Gamma_normal} implies the first estimate in~\eqref{eq:taylor-christoff}.
  
 \medskip
 
 Similarly,~\eqref{eq:fermi_derivative_g1j_a} and~\eqref{eq:fermi_derivative_g1j_b} imply that
 \begin{align} 3\overline \Gamma_{ij1}(x) =   & \,  \frac32 (  \partial_{x_j}\overline g_{1i}(x) + \partial_{x_1}\overline g_{ij}(x)  - \partial_{x_i}\overline g_{j1}(x)) \nonumber  \\
 =  & \,- [2 \overline{\mathcal R}_{ij1l}(x_1) -  \overline{\mathcal R}_{j1il}(x_1)  -   \overline{\mathcal R}_{1ijl}(x_1) ] x'_l +  \mathcal O(|x'|^2)    \nonumber \\
 = & \, - 3 \overline{\mathcal R}_{ij1l}(x_1) x_l  +  \mathcal O(|x'|^2)
 \quad \text{for $i,j \ge 2$.}  \label{eq:fermi_Gamma_ij1}
 \end{align} 
 Here we used in the last step that
  \begin{align*}  
 & \,  \overline{\mathcal R}_{ijkl} +  \overline{\mathcal R}_{jkil} +  \overline{\mathcal R}_{kijl}  = 
 \overline{\mathcal R}_{klij} +  \overline{\mathcal R}_{iljk} +  \overline{\mathcal R}_{jlki}    
 = -[\overline{\mathcal R}_{lkij} +  \overline{\mathcal R}_{lijk} +  \overline{\mathcal R}_{ljki} ] = 0,
 \end{align*}
 since $\mathcal R(X,Y)Z + \mathcal R(Y,Z)X + \mathcal R(Z,X)Y = 0$.

 Moreover,
  \begin{align}
 2\overline \Gamma_{k11}(x) =  & \, - \partial_{x_k}\overline g_{11}(x) =  
 [\overline{\mathcal R}_{1k1l}(x_1) +   \overline{\mathcal R}_{1l1k}(x_1) ]x'_l  +   \mathcal O(|x'|^2)  \nonumber \\
 = & \, 2 \overline{\mathcal R}_{1k1l}(x_1) x_l +  \mathcal O(|x'|^2) \nonumber  \\
 = & \,-  2  \overline{\mathcal R}_{k11l}(x_1) x'_l  +  \mathcal O(|x'|^2) \quad \text{for $k \ge 2$.}
 \label{eq:fermi_Gamma_k11}
 \end{align}
 In addition we have  $\overline\Gamma_{111} = \mathcal O(|x'|^2)$,  since $\partial_{x_1}\overline g_{11} = \mathcal O(|x'|^2)$ and
 the antisymmetry of $\overline{\mathcal R}$ in the first two arguments gives
 $ \overline{\mathcal R}_{111l} = 0$. Hence  the identity  \eqref{eq:fermi_Gamma_k11} also holds for $k=1$. 
 Combining this observation with  \eqref{eq:fermi_Gamma_ij1} (with $(i,j)$ replaced by $(k,i)$  we
 get
 \begin{equation} \label{eq:fermi_Gamma_tangential}
  \overline\Gamma_{ki1}(x)  =-  \overline{\mathcal R}_{ki1l}(x_1) x_l  +  \mathcal O(|x'|^2) \quad \text{for $i,k \ge 1$.}
 \end{equation}
 Together with \eqref{eq:fermi_Gamma_upper_lower}   we get the second estimate in \eqref{eq:taylor-christoff}.
 
 \medskip

  Finally we discuss the independence of the constants in the estimates~\eqref{eq:taylor-metric} and~\eqref{eq:taylor-christoff}
  of the  framed unit speed geodesic $(\gamma, \underline \nu)$.
 Set
 $$ C_\rho \coloneqq   (-2L,2L) \times B_\rho(0).$$
 It follows from~\eqref{eq:fermi_trivial_bound_A} that there exists an $\eps_0$ such that
  $$\sup_{x \in C_\rho}  |\overline g - \Id | \le 
 C \rho^2 \sup_{x \in C_\rho} |\overline{\mathcal R}(x)|$$
 as long as
 $$ \rho^2 \sup_{x \in C_\rho} |\overline{\mathcal R}(x)| \le \eps_0$$
 where $|\mathcal R(p)|$ is an intrinsic norm of the curvature tensor at $p \in \mathcal M$. 
 On the other hand,  if $ \frac{1}{2} \Id \le \overline g \le 2 \Id$ then
  \begin{equation}   \label{eq:fermi_intrinsic_C0}
   \sup_{C_\rho} |\overline{\mathcal R}(x)| \le C \sup_{ p \in \psi(C_\rho)} |\mathcal R(p)|.
  \end{equation}
  Thus a continuity argument shows that 
   $ \frac{1}{2} \Id \le \overline g \le 2 \Id$ and  \eqref{eq:fermi_intrinsic_C0} hold provided that
   $$ \rho^2  \sup_{ p \in \psi(C_\rho)} |\mathcal R(p)| \le \eps_1$$
   for some $\eps_1 > 0$. 
   
   Similarly,  the other constants in the estimates for $\overline g$ and $\overline \Gamma$ are controlled
   by $ \sup_{x \in C_\rho} |\overline{\mathcal R}(x)|$ and $\sup_{x \in C_\rho} |D\overline{\mathcal R}(x)|$.
   Now,  for $|\Gamma| \le 1$   and $ \frac{1}{2} \Id \le\overline g \le 2 \Id$,  
   the quantity $\sup_{x \in C_\rho} |D\overline{\mathcal R}(x)|$
   is controlled by the $C^0$ norm of $\mathcal R$ and 
   its covariant derivative $\nabla \mathcal R$ in $\psi(C_\rho)$. 
   Thus another  continuity argument based on  \eqref{eq:fermi_partialM0}--\eqref{eqref:partialA_ij}  shows that that the constants are uniformly controlled if
   \begin{align*}
    \rho  \sup_{ p \in \psi(C_\rho)} |\mathcal R(p)|  +  \rho^2  \sup_{ p \in \psi(C_\rho)} |\mathcal \nabla\mathcal R(p)| 
   \le \eps_2
   \end{align*}
   for some $\eps_2 > 0$.  
   In particular,  for a compact manifold, the radius $\rho$ and the constants $C$ in 
   can be chosen independent of the frame unit speed geodesic $(\gamma, \underline \nu)$. 
   
\end{proof}

%%%%%%%%%%%%%%%%%%%%%%%%%%%%%%%%%%%%%%%%%%%%%%%%%
%%% bibliography
%%%%%%%%%%%%%%%%%%%%%%%%%%%%%%%%%%%%%%%%%%%%%%%%%

\bibliographystyle{abbrv}
\bibliography{bibfile.bib}

\end{document}